\documentclass[11pt,a4paper]{article}
\usepackage[utf8]{inputenc}
\usepackage{authblk}
\title{\large\textbf{Sub-Riemannian Navier-Stokes system on the Heisenberg group:\\ Weak solutions, well-posedness and smoothing effects\footnote{The author is partially supported by the Project TRECOS ANR-20-CE40-0009 funded by the ANR. This work benefited from  the financial support of the University of Bordeaux within an initiative of excellence, under the "France 2030" plan.}}
}
\author{
\small Adrien TENDANI-SOLER\footnote{ \textit{E-mail}: adrien.tendani-soler@math.u-bordeaux.fr}}
\affil{\scriptsize Institut de Math\'ematiques de Bordeaux UMR 5251,\\ Universit\'e de Bordeaux, Bordeaux INP, CNRS\\
F-33405 Talence, France}
\usepackage[nottoc]{tocbibind}
\usepackage[T1]{fontenc}
\usepackage{amsmath}
\usepackage{minitoc}
\usepackage{graphicx}
\usepackage{mathrsfs}
\usepackage{amsthm}
\usepackage{amsfonts}
\usepackage{amssymb}
\usepackage{setspace}
\usepackage{hyperref}
\usepackage{varioref}
\usepackage{color}
\hypersetup{
    colorlinks=true,
    linkcolor=blue,
    filecolor=magenta,     
    urlcolor=cyan,
    pdfpagemode=FullScreen,
    }
\usepackage{makeidx}
\usepackage
{geometry}
\usepackage{fullpage}
\usepackage{tikz}
\usepackage{enumitem}
\usetikzlibrary{shapes.geometric}
\usetikzlibrary{intersections}
\usetikzlibrary{3d}

\usetikzlibrary{calc}
\usepackage{appendix}
\newtheorem{theo}{Theorem}[subsection]

\newtheorem{rem}[theo]{Remark}
\newtheorem{defi}[theo]{Definition}
\newtheorem{notation}[theo]{Notation}
\newtheorem{lem}[theo]{Lemma}
\newtheorem{prop}[theo]{Proposition}
\newtheorem{coro}[theo]{Corollary}

\DeclareMathOperator{\Id}{Id}

\DeclareMathOperator{\rad}{rad}
\DeclareMathOperator{\Tr}{Tr}

\newcommand{\C}{\mathbb{C}}
\newcommand{\R}{\mathbb{R}}

\newcommand{\N}{\mathbb{N}}

\newcommand{\Hg}{\mathbb{H}}

\newcommand{\mbb}{\mathbb}

\newcommand{\mc}{\mathcal}

\newcommand{\vphi}{\varphi}

\newcommand{\Jk}{\mathsf{J}_{k}}
\newcommand{\Ja}{\mathsf{J}_{a}}
\newcommand{\TJk}{\Tilde{\mathsf{J}}_{k}}

\newcommand{\enstq}[2]{\left\{#1~\middle|~#2\right\}}

\DeclareMathOperator{\divergence}{\mathop{}div}
\newcommand*\Laplace{\mathop{}\!\mathbin\bigtriangleup}

\usepackage{abstract}
\renewcommand{\leq}{\leqslant}
\renewcommand{\geq}{\geqslant}

\newcommand{\revdots}{\mathinner
{ \mkern1mu\raise1pt\vbox{\kern7pt\hbox{.}} \mkern2mu\raise4pt\hbox{.} \mkern2mu\raise7pt\hbox{.}\mkern1mu}
}

\newcommand{\croixdots}{ \mathinner{
\mkern1mu\raise7pt\vbox{\kern7pt 
\hbox{.}}
\mkern-5mu
\raise7pt\vbox{\hbox{.}}
\mkern1mu }
}

\renewcommand{\abstitlestyle}[1]{\noindent}
 \numberwithin{equation}{section}

\begin{document}

{\setlength{\baselineskip}{0.1\baselineskip}
\maketitle}

\hfill
\begin{abstract}
This article is devoted to the derivation of the incompressible sub-Riemannian Euler and the sub-Riemannian Navier-Stokes systems, and the analysis of the last one in the case of the Heisenberg group. In contrast to the classical Navier-Stokes system in the Euclidean setting, the diffusion is not elliptic but only hypoelliptic, and the commutator of the Leray projector and the hypoelliptic Laplacian is of order two. Yet, we study the existence of solutions in two different settings: within the $L^2$ setting which provides global existence of weak solutions;  within a critical scale-invariant Sobolev-type space, associated with the regularity of the generators of the first stratum of the Lie algebra of right-invariant vector fields. In this latter class, we establish global existence of solutions for small data and a stability estimate in the energy spaces which ensures the uniqueness of the solutions in this class. Furthermore, we show in this setting that these solutions instantly become analytic in the vertical direction. Surprisingly, we obtain a larger lower bound of the radius of analyticity in the vertical direction for large times than for the usual incompressible Navier-Stokes system in the Euclidean setting. Finally, using the structure of the system, we recover the $\mathcal{C}^{\infty}$ smoothness in the other directions by using the analyticity in the vertical variable.
\end{abstract}
\hfill
{\setlength{\baselineskip}{0.95\baselineskip}
\scriptsize\tableofcontents\par}
\hfill
\section{Introduction}\label{sec:Introduction}
In this article, we derived a new class of anisotropic models for incompressible fluids, called the sub-Riemannian Euler and Navier-Stokes system, and we investigate the Cauchy theory of the latter one in the case of the Heisenberg group. We study three problems in the case of the Heisenberg group: the global existence of finite energy (weak) solutions, the global well-posedness in a Sobolev-type critical framework and the (analytic and $\mathcal{C}^\infty$) regularity of the solutions. Before stating our results and methods, let us introduce the models.

\subsection{Presentation of the model}
\label{subsec:Presentation of the model}
Many PDEs from fluid mechanics are anisotropic as, for instance,  Prandtl system, Ekman layer for rotating fluids, or  models of the wind-driven oceanic circulation and liquid-crystal models (see for instance \cite{GeophysicalFluidDynamics} Chapters 4 and 5 and \cite{TheoriesofFluidswithMicrostructureAnIntroduction} Chapter 4 for more examples). The mathematical analysis of these equations presents several challenging questions about the properties of their solutions. On the other hand, investigating classical PDEs (such as the Laplace equation, heat equation, wave equation and Schr\"odinger equation) within the context of the sub-Riemannian geometry, by considering their subelliptic counterpart, has recently attracted considerable attention (see for  instance \cite{LocaldispersiveandStrichartzestimatesfortheSchrodingeroperatorontheHeisenberggroup}, \cite{Principedemoindreactionpropagationdelachaleuretestimeessouselliptiquessurcertainsgroupesnilpotents} and \cite{Subellipticwaveequationsareneverobservable}). This has led to a reappraisal of the usual properties of their solutions and the development of new approaches. This work is at the intersection of these two topics: we consider anisotropic models, describing incompressible fluids on the whole space, whose anisotropy is naturally encoded by a sub-Riemannian structure.
\paragraph{Motivations.} We aim to write a class of equations which describes the motion of a fluid when the velocity field is anisotropic and, more precisely, cannot take all possible directions. Keeping in mind that geometrically, a vector field is a section of the tangent bundle, we are naturally led to investigate the situation in which the anisotropy comes from the velocity field being a section (called horizontal vector fields) of a subbundle (called horizontal bundle) of the tangent bundle. The sub-Riemannian geometry provides a natural context to describe such situations. In this geometric framework, the choice of a horizontal velocity field translates into the fact that particles of fluid cannot move in all directions according to predefined restrictions. In this work, the choice of the Heisenberg group is led by the fact that this is the simplest non-trivial example of a sub-Riemannian structure on the Euclidean space $\R^3$.
\paragraph{Formal derivation of the sub-Riemannian Euler and the sub-Riemannian Navier-Stokes systems on the Heisenberg group $\Hg^d$.} In this paragraph, we present the model and its derivation, and the underlying group structure will be given in detail in Section \ref{sec:Heisenberg group}.
We begin by defining what we will call a horizontal vector field on the Heisenberg group $\Hg^d=\R^{2d+1}$. 
Let us consider a vector field $v:=\!^t(v_1,\dots,v_{2d+1})$ on $\R^{2d+1}$, where for any $i\in[\![1,2d+1]\!]$, $v_i:(Y,s)\in\R^{2d}\times\R\mapsto v_i(Y,s)\in\R$. We say that $v$ is a \textit{horizontal vector field} if the following relation holds 
$$
v=\mathcal{R} \begin{pmatrix}
v_1 \\ \vdots \\ v_{2d} \end{pmatrix},
$$
where $\mathcal{R}:\R^{2d}\rightarrow\mathcal{L}(\R^{2d},\R^{2d+1})$ is a map given, for any $Y \in \R^{2d}$, by
\begin{equation}
\mathcal{R}_{Y}:=\left (
   \begin{array}{c}
       I_{2d} \\
       
      \hline
      
      ^t(\mathfrak{S} Y)\\ 
   \end{array}
\right),
\quad \text{ with } 
\label{matrice mathfrak S}
\mathfrak{S}:=\left (
   \begin{array}{c|c}
      0 & 2I_{d} \\
      \hline
      -2I_{d} & 0\\ 
   \end{array}
\right).
\end{equation}
(Here, $I_{p}$ denotes the identity matrix of size $p$.)
In other words, $v:\R^{2d+1}\rightarrow\R^{2d+1}$ is a horizontal vector field if
$$
\forall (y,\eta,s)\in\R^d\times\R^d\times\R, \quad v_{2d+1}(y,\eta,s)=2\sum_{j=1}^{d}\left( \eta_{j}v_j(y,\eta,s) - y_{j}v_{j+d}(y,\eta,s)\right).
$$
Because a horizontal vector field is determined by its $2d$-first coordinates, we slightly abuse the notation by calling maps $v:\R^{2d+1}\rightarrow\R^{2d}$ horizontal vector fields in the following, meaning that the component of $\mathcal{R}v$ is defined as above.
 
Using $\mathcal{R}$, we can define the following (left-invariant, see Section \ref{sec:Heisenberg group}) differential operators
\begin{equation}
    \label{Operators sR avec matrice R}
    \nabla_{\Hg}:=\!^t\mathcal{R}\nabla,\ \ \divergence_{\Hg}:=\divergence(\mathcal{R}\cdot)\ \ \text{and}\ \ \Laplace_{\Hg}:=\divergence(\mathcal{R}\, \!^t\mathcal{R}\nabla\cdot).
\end{equation}
See also \eqref{HeisenbergOp-InTermsOf-Pj} for another definition of those operators in terms of generators of the Lie algebra of $\Hg^d$.

The leading idea of our derivation of the sub-Riemannian Euler or Navier-Stokes system is to \textit{describe the dynamic of a fluid in the Euclidean space for which the velocity field is horizontal}. To derive the Euler system, we follow the usual derivation of the Euler system on $\R^{2d+1}$, but we further impose that the trajectories of the fluids particles follow a horizontal velocity field. In the following, the velocity field of the particles will be considered to be a smooth function $u:(t,x)\in\R_{+}\times\R^{2d+1}\mapsto u(t,x)\in\R^{2d}$.\\
\textit{Incompressibility.} We assume that the horizontal vector field $\mathcal{R}u$ is incompressible, that is, with the operator introduced in \eqref{Operators sR avec matrice R}
$$
\divergence_{\Hg}(u)=0.
$$
\textit{Lagrangian description.} The motion of a particle of fluid with horizontal velocity field $u$ is  described by the flow map $\Psi$ defined as follows
\begin{equation*}
\begin{cases}
\partial_t\Psi(t,\cdot)=\mathcal{R}u(t,\Psi(t,\cdot)),\\
\Psi(0,\cdot)=I_{{2d+1}}.
\end{cases}
\end{equation*}
The variation of the velocity of a particle of the fluid which moves on $\R^{2d+1}$ is then given by
$$
\frac{d}{dt}u_{\Psi}(t,x)=\left(\partial_tu+\mathcal{R}u\cdot\nabla u\right)(t,\Psi(t,x)),
$$
where $u_{\Psi}(t,x):=u(t,\Psi(t,x))$ with $(t,x)\in\R_{+}\times\R^{2d+1}$. By using \eqref{Operators sR avec matrice R}, we deduce that
$$
\frac{d}{dt}u_{\Psi}(t,x)=\left(\partial_tu+u\cdot\nabla_{\Hg}u\right)(t,\Psi(t,x)).
$$
According to the Newton laws, the term $\partial_tu+u\cdot\nabla_{\Hg}u$ is the sum of the forces acting on the fluid.\\
\textit{Pressure.} We follow the physical construction of the pressure term in the Euler and Navier-Stokes system on the Euclidean setting by keeping in mind that the flow is horizontal, see for instance \cite{LandauLifshitzV6}. Given a volume of fluid $\Omega\subset\R^{2d+1}$, we assume that the pressure acts on $\partial\Omega$ with respect to the horizontal normal $\overrightarrow{n}_{\Hg}:= \!^t\mathcal{R}\overrightarrow{n}$, where $\overrightarrow{n}$ is the outward-pointing unit normal vector field on $\R^{2d+1}$, since the fluid moves through the boundary $\partial \Omega$ following horizontal velocity fields. Thus the action of the pressure on the fluid domain $\Omega$ is given by the term $-\nabla_{\Hg}p$.\\
At this step we obtain the \textit{sub-Riemannian Euler equations on the Heisenberg group}
\begin{equation*}
\begin{cases}
\partial_tu+u\cdot\nabla_{\Hg}u=-\nabla_{\Hg}p,\\
\divergence_{\Hg}(u)=0.
\end{cases}
\end{equation*}
This system formally satisfies the conservation of the kinetic energy: for any sufficiently smooth solutions $(u,p)$ of the sub-Riemannian Euler system, we have 
$$
\frac{d}{dt}\|u\|_{L^2}^{2}=0.
$$
The \textit{sub-Riemannian Navier-Stokes equations on the Heisenberg group} are the following viscous perturbation of the preceding system, perturbed by the Heisenberg sub-Laplacian:
\begin{equation}
\label{eq:NSH derivation}
\begin{cases}
\partial_tu+u\cdot\nabla_{\Hg}u-\Laplace_{\Hg}u=-\nabla_{\Hg}p,\\
\divergence_{\Hg}(u)=0,
\end{cases}
\end{equation}
so that for any sufficiently smooth solutions $(u,p)$ of \eqref{eq:NSH derivation}, we have the dissipation of the kinetic energy according to the law
\begin{equation}
\label{eq:Dissipation}
\frac{1}{2}\frac{d}{dt}\|u\|_{L^2}^{2}+\|\nabla_{\Hg}u\|_{L^2}^{2}=0.
\end{equation}

The choice of the operators $\Laplace_\Hg$ is driven by the fact that this operator is the simplest natural diffusion in the geometric setting of the Heisenberg group (see \cite{GeometryAnalysisandDynamicsonSubRiemannianManifolds} Chapter 2, Section 5.4 or Subsection \ref{subsec:Lie group structure} in this article), and it also guarantees the above dissipation of the energy.

Note that there are other models of incompressible Navier-Stokes systems on the Heisenberg group, in particular in \cite{AnexistenceanduniquenessresultfortheNavierStokestypeequationsontheHeisenberggroup}, but we think that our model is more physically justified than the one in \cite{AnexistenceanduniquenessresultfortheNavierStokestypeequationsontheHeisenberggroup}, for which there is no clear dissipation law such as \eqref{eq:Dissipation}  - in fact, the corresponding Euler equation in \cite{AnexistenceanduniquenessresultfortheNavierStokestypeequationsontheHeisenberggroup} does not preserve the kinetic energy of the fluid.

\paragraph{Comment on the geometric aspects.}
We now describe the geometric framework underlying the above derivation of the sub-Riemannian Euler and Navier-Stokes systems.
 The notion of horizontal vector fields we use is associated with the so-called left-invariant sub-Riemannian structure of the Heisenberg group $\Hg^{d}$ identified to $\R^{2d+1}$ with a suitable group law. Note that our derivation of the Sub-Riemannian system on the Heisenberg group can be extended to more general sub-Riemannian manifolds, for instance stratified Lie groups (see Section \ref{sec:Comments on the generalization and open problems}). The importance of the case of stratified Lie groups comes from the fact that they are model spaces in sub-Riemannian geometry (as the Euclidean spaces are model spaces in
 Riemannian geometry). 
 
From the point of view of the mathematical analysis, the importance of stratified Lie groups in sub-Riemannian geometry is illustrated by the results of \cite{Hypoellipticdifferentialoperatorsandnilpotentgroups}. In \cite{Hypoellipticdifferentialoperatorsandnilpotentgroups}, the authors obtained sharp regularity results for the sum of squares of H\"ormander vector fields by using analysis on nilpotent Lie groups. A crucial step of their proof is the
construction of new vector fields (on a larger manifold), which lift
the original vector fields and can be locally approximated by left-invariant vector fields on a stratified group.

Also note that homogeneous Lie groups (which include stratified Lie groups) provide a natural setting to generalize many tools from Euclidean harmonic analysis and in particular to define a suitable Fourier transform on (non-compact) locally compact Lie groups (see \cite{QuantizationonNilpotentLieGroups} and Section \ref{subsec:Fourier transform}). 

\paragraph{Regarding the Heisenberg sub-Laplacian.} The diffusion induced by the hypoelliptic operator $\Laplace_{\Hg}$ is anisotropic: the diffusion is generated by only $2d$ independent vector fields, while $\Hg^d$ is of dimension $2d+1$ (this implies in particular that $\Laplace_{\Hg}$ is not elliptic). The influence of this viscous perturbation in the sub-Riemannian Euler system on $\Hg^d$, through smoothing effect on the solutions of the sub-Riemannian Navier-Stokes system on $\Hg^d$, will be studied in detail in this article. In particular, we will point out that, in the sub-Riemannian Navier-Stokes system on $\Hg^d$, the energy is also dissipated in the vertical direction, although it does not appear explicitly in the dissipation law \eqref{eq:Dissipation}. We will underline the importance of this phenomenon on the regularity of the solutions of \eqref{eq:NSH derivation}.

\subsection{Main results: Existence of weak solutions, well-posedness and smoothing effects}\label{subsec:Main results}

\begin{notation}
For $T>0$, $p\in[1,+\infty)\cup\{\infty\}$ and $E$ a Banach space, we set
$$
L^{p}_{T}(E):=L^{p}((0,T);E),\ \ L^{p}(E):=L^{p}(\R_{+};E)\ \ \text{and}\ \ L^{p}_{loc}(E):=L^{p}_{loc}(\R_{+};E).
$$
In the same spirit, if $f$ belongs to $L^{p}_{T}(E)$, respectively to $L^{p}(E)$, we set
$$
\|f\|_{L^{p}_{T}(E)}:=\left(\int_{0}^{T}\|f(t)\|_{E}^{p}dt\right)^{\frac{1}{p}},\ \ \text{respectively}\ \ \|f\|_{L^{p}(E)}:=\left(\int_{\R_+}\|f(t)\|_{E}^{p}dt\right)^{\frac{1}{p}}.
$$
We also denote by $\mathcal{C}_{w}([0,T);E)$, respectively $\mathcal{C}_w(\R_+, E)$, the space of measurable functions from $[0,T)$ to $E$, respectively from $\R_+$ to $E$, which are continuous for the weak topology on $E$. When $F$ is a functional space, we use the same notations for $F$ and $F^{2d}$.
\end{notation}
In this article, we consider the following Cauchy problem:
\begin{equation}
\label{NSH}
\begin{cases}
\partial_tu-\Laplace_{\Hg}u+u\cdot\nabla_{\Hg}u=-\nabla_{\Hg}p\ &\text{in}\ \ \R_{+}\times\Hg^{d},\\
\divergence_{\Hg}(u)=0\ &\text{in}\ \ \R_{+}\times\Hg^{d},
\end{cases}
\end{equation}
with the initial condition
\begin{equation}
\label{NSH initial condition}
u_{|_{t=0}}=u_0\ \ \ \text{in}\ \ \Hg^{d},
\end{equation}
where $u_0$ belongs to $\mathcal{S}'(\Hg^d)^{2d}$ and satisfies $\divergence_\Hg(u_0)=0$. In this article, we investigate two frameworks for the Cauchy problem \eqref{NSH}-\eqref{NSH initial condition}.
\paragraph{Weak solutions in $L^2$.} For the incompressible Navier-Stokes system in the Euclidean setting, the question of the existence of the global weak solutions in the energy space, the so called Leray solution, was solved in \cite{Essaisurlemouvementdunliquidevisqeuxemplissantlespace}. Still, their uniqueness is nowadays largely open despite some recent progress on this question (see \cite{NonuniquenessofweaksolutionstotheNavierStokesequation, NonuniquenessofLeraysolutionsoftheforcedNavierStokesequations}).

In the case of the sub-Riemannian Navier-Stokes, the structure of the system is adapted to obtain the existence of Leray-type solution for \eqref{NSH}-\eqref{NSH initial condition} with initial data in $L^2(\Hg^d)$. Let us first introduce the notion of weak solutions for \eqref{NSH}-\eqref{NSH initial condition}.
\begin{defi}[Definition of the weak solutions]
\label{def:Definition of weak solution for NSH}
Let us define the space 
$$
\mathcal{D}_\sigma:=\enstq{\vphi\in\mathcal{D}([0,+\infty)\times\Hg^d)^{2d}}{\divergence_{\Hg}(\vphi)=0}.
$$
Let $u_0\in L^2(\Hg^d)$ be a horizontal vector field satisfying $\divergence_\Hg(u_0)=0$. We say that a horizontal vector field $u$ belonging to $L^{2}_{loc}(\R_+\times\Hg^d)$ is a \textbf{global weak solution} of \eqref{NSH}-\eqref{NSH initial condition} if
\begin{enumerate}[topsep=0pt,parsep=0pt,leftmargin=*]
\item (Integrability conditions) $u$ belongs to $\mathcal{C}_{w}([0,+\infty);L^2)\cap L^{\infty}(\R_+; L^2)$ and $\nabla_{\Hg}u $ to $ L^2(\R_+;L^2)$,
\item (Initial condition) $\lim_{t\rightarrow 0^+}u(t)=u_0$ for the weak topology of $L^2(\Hg^d)$,
\item (Momentum equation) for any $t' \leq t$ in $[0,+\infty)$ and for any $\vphi\in\mathcal{D}_\sigma$, we have
$$
\int_{\Hg^d} u(t)\cdot\vphi(t)dx - \int_{t'}^{t}\int_{\Hg^d}\left(u\cdot\partial_t\vphi+u\cdot\Laplace_\Hg \vphi+(u\otimes u)\cdot\nabla_\Hg \vphi\right) dxd\tau = \int_{\Hg^d}u(t')\cdot\vphi(t')dx,
$$
\item (Continuity equation) For all $t>0$, we have $\divergence_{\Hg}(u(t))=0$ in $\mathcal{D}'(\Hg^d)^{2d}$.
\end{enumerate} 
\end{defi}
Let us remark that any weak solution has finite energy, since it belongs to $L^2(\Hg^d)$ at any times. Our first main result is the following one, whose proof is given in Section \ref{sec:Weak solutions}.

\begin{theo}[Existence of weak solutions]
\label{th main:Existence of weak solutions}
Let $u_0$ be a horizontal vector field belonging to $L^2(\Hg^d)$ and satisfying $\divergence_{\Hg}(u_0)=0$. Then there exists a global weak solution $u$ of \eqref{NSH}-\eqref{NSH initial condition}, satisfying the following energy estimate
\begin{equation*}
\|u\|_{L^{\infty}(L^2)}^{2}+2\|\nabla_{\Hg} u\|_{L^2(L^2)}^{2}\leq \|u_0\|_{L^2}^{2}.
\end{equation*}
\end{theo}
We extend this result to the sub-Riemannian Navier-Stokes system on general stratified Lie group in \ref{sec:Comments on the generalization and open problems}. The key step of the proof of Theorem \ref{th main:Existence of weak solutions} consists of defining a suitable approximate problem. 
This approximate problem follows a suitably modified version of the Friedrichs method and involves the Fourier transform on the Heisenberg group. The analysis of the approximate problem reveals technical and conceptual difficulties due to the intrinsic structure of the system (pressure, nonlinear terms and divergence free condition), see Section \ref{sec:Main ideas}, and the non-ellipticity of $\Laplace_\Hg$, see Section \ref{sec:Derivation of a suitable approximate system}. As we will see next, this approximation is also used later to establish the existence and the regularity of the solutions of \eqref{NSH}-\eqref{NSH initial condition} in the critical framework presented afterwards.

\paragraph{Well-posedness in the critical framework $\Tilde{H}^d$.} An important notion to select functional spaces for studying Cauchy problems of PDE is the so-called scaling invariance. It turns out that the Sub-Riemannian Navier-Stokes system \eqref{NSH} on the Heisenberg group also presents a scaling invariance, which is strongly related to the geometric structure of $\Hg^d$ (see Section \ref{sec:Heisenberg group}).\\
\textit{Scaling invariance.} The Sub-Riemannian Navier-Stokes system  \eqref{NSH} has the following scaling invariance property:  $u$ satisfies \eqref{NSH}-\eqref{NSH initial condition} with initial data $u_0$ if and only if for all $\mu>0$, the horizontal vector field $u_\mu$ given for any $t\in\R_{+}$ and $(Y,s)\in\R^{2d}\times\R$ by 
\begin{equation}
\label{scaling transformation}
u_{\mu}(t,Y,s):=\mu u(\mu^2 t,\mu Y,\mu^2s),
\end{equation}
satisfies \eqref{NSH}-\eqref{NSH initial condition} with data $\mu u_0(\mu \cdot,\mu^2 \cdot)$. 

We are thus looking for a Banach space of initial data whose norm is invariant by this scaling transformation. Such space is called a critical space for \eqref{NSH}.

Let us define the space 
\begin{equation*}
\Tilde{H}^d:=\enstq{f\in L^{2d+2}(\Hg^d)}{(-\Tilde{\Laplace}_{\Hg})^{\frac{d}{2}}f\in L^2(\Hg^d)},
\end{equation*}
where the operator $\Tilde{\Laplace}_{\Hg}$ is called the right-invariant sub-Laplacian on $\Hg^d$. This operator will be defined in \eqref{Definition sousLaplacien invariant à droite} and its fractional powers will be defined via the functional calculus of $-\Tilde{\Laplace}_\Hg$ in Section \ref{subsubsec:Pseudo-differential calculus on Hd}. The space $\Tilde{H}^d$ is a critical space for \eqref{NSH} (see \eqref{eq:calcul de l invariance d echelle des norme Sobolev}). 
Our second main result is the following Theorem on the Cauchy problem of \eqref{NSH}-\eqref{NSH initial condition} in $\Tilde{H}^d$:
\begin{theo}[Global well-posedness in $\Tilde{H}^d$]
\label{th main:Global well-posedness in TildeHd}
Let $d\geq 1$ be an integer.
\begin{enumerate}[topsep=0pt,parsep=0pt,leftmargin=*]
\item (\textit{Existence and uniqueness}) For any small enough horizontal vector field $u_0$ in $\Tilde{H}^{d}$ satisfying $\divergence_{\Hg}(u_0)=0$,
there exists a unique solution $u$ of \eqref{NSH}-\eqref{NSH initial condition} which satisfies $u\in\mathcal{C}_{b}(\R_{+};\Tilde{H}^d)$ and $\nabla_{\Hg}u\in L^2(\R_{+};\Tilde{H}^d)$.
\item (\textit{Stability estimate}) There exists a positive constant $C$ such that for any $T>0$ and for any solutions $u$ and $v$ of \eqref{NSH} such that $u$ and $v$ belong to $\mathcal{C}_{b}([0,T];\Tilde{H}^d)$ and $\nabla_{\Hg}u$ and $\nabla_{\Hg}v$ belong to $L^2([0,T];\Tilde{H}^{d})$, we have
\begin{align*}
\|u-v\|_{L^{\infty}_{T}(\Tilde{H}^d)}^{2}+\|\nabla_{\Hg}(&u-v)\|_{L^{2}_{T}(\Tilde{H}^{d})}^{2}\\
&\leq \|u(0)-v(0)\|_{\Tilde{H}^d}^{2}\exp{\left(C\|\nabla_{\Hg}v\|_{L^{2}_{T}(\Tilde{H}^{d})}^{2}+C\|\nabla_{\Hg}u\|_{L^{2}_{T}(\Tilde{H}^{d})}^{2}\right)}.
\end{align*}
\end{enumerate}
\end{theo}
Before going further, let us point out that the notion of solutions in Theorem \ref{th main:Global well-posedness in TildeHd} does not refer to the concept of weak solutions introduced in Definition \ref{def:Definition of weak solution for NSH}, since the solution $u$ constructed in  Theorem \ref{th main:Global well-posedness in TildeHd}, Item 1 does not a priori belong to the energy space $L^\infty(L^2)$. Therefore, the solutions constructed in Theorem \ref{th main:Global well-posedness in TildeHd} are to be understood as follows:  $u$ satisfies $u\in\mathcal{C}_{b}(\R_{+};\Tilde{H}^d)$ and $\nabla_{\Hg}u\in L^2(\R_{+};\Tilde{H}^d)$, $u(0) = u_0$ in $\Tilde{H}^d$, and $u$ satisfies \eqref{NSH} in the sense of distributions (i.e. Definition \ref{def:Definition of weak solution for NSH}, Items 3 and 4).

Theorem \ref{th main:Global well-posedness in TildeHd} will be proved in Section \ref{sec:Well-posedness in TildeHd} in two steps. We first show Theorem \ref{Global existence theorem in Tilde H1}, which states the existence of global solutions with small initial data in $\Tilde{H}^d$. Next, we establish Theorem \ref{Stabilite dans Tilde H}, which provides the stability estimate in Item 2. Of course, this stability estimate ensures the uniqueness of the solutions of \eqref{NSH}-\eqref{NSH initial condition} claimed in Theorem \ref{th main:Global well-posedness in TildeHd}, Item 1. 

Let us make some remarks about the regularity of the solutions provided by Theorem \ref{th main:Global well-posedness in TildeHd}. First, the main advantage of the operator $-\Tilde{\Laplace}_\Hg$ and its powers is that it commutes with the operators $\nabla_\Hg$, $\divergence_\Hg$ and $\Laplace_\Hg$, see Section \ref{sec:Heisenberg group} and Section \ref{sec:Main ideas}. Second, the diffusion in \eqref{NSH} corresponds to the left-invariant sub-Laplacian $\Laplace_{\Hg}$, and is thus distinct from the information that we propagate, that is the $\Tilde{H}^d$ regularity, which corresponds to the right-invariant sub-Laplacian $\Tilde{\Laplace}_{\Hg}$, see Remark \ref{Remarque sur les espaces de Sobolev gauche et droite}. Accordingly, to obtain strong solutions of \eqref{NSH}, we need to get more regularity on the solutions constructed in Theorem \ref{th main:Global well-posedness in TildeHd}.

\paragraph{Smoothing effects in the critical framework $\Tilde{H}^d$.} The problem of the regularity for the solution of the Navier-Stokes system in the Euclidean sitting was mainly investigated since the pioneer work \cite{Essaisurlemouvementdunliquidevisqeuxemplissantlespace}. The question of the regularity was treated firstly for its link with the uniqueness of the Leray solutions and the analytic smoothing effect was studied later (see \cite{GevreyclassregularityforthesolutionsoftheNavier-Stokesequations, grujicKukavica_space_1998, AnalyticityandDecayEstimatesoftheNavierStokesEquationsinCriticalBesovSpaces, AnalyticityuptotheboundaryfortheStokesandtheNavierStokessystems,AnalyticityestimatesfortheNavier-Stokesequations, Ontheradiusofanalyticityofsolutiontosemi-linearparabolicsystems}). The relevant measure of the analytic smoothing is the notion of the radius of analyticity. In particular the radius of analyticity for the solution of the Navier-Stokes system was studied for its link with the turbulence theory (see for instance \cite{SmallestscaleestimatesfortheNavierStokesequationsforincompressiblefluids}). More applications of the estimate of the radius of analyticity in space for solutions of the Navier-Stokes equation appear in other contexts, such as numerical analysis \cite{ExponentialConvergenceoftheGalerkinApproximationfortheGinzburgLandauEquation}, temporal decay rates of Sobolev norms \cite{RemarkontheRateofDecayofHigherOrderDerivativesforSolutionstotheNavierStokesEquationsinRn} and geometric regularity criteria for the Navier–Stokes equations \cite{Thegeometricstructureofthesuperlevelsetsandregularityfor3DNavierStokesequations}.
 In the case of the incompressible Navier-Stokes system in $\R^3$, the best estimate in large times on the radius of analyticity $\rad(u(t)):=\sup\enstq{R\geq 0}{e^{R(-\Laplace)^{1/2}}u(t)\in\dot{H}^{1/2}(\R^3)}$ of the solutions $u$ of the incompressible Navier-Stokes system in the critical framework $\dot{H}^{1/2}(\R^3)$ (see \cite{AnalyticityestimatesfortheNavier-Stokesequations}) is, to the best of our knowledge,
\begin{equation}
\label{eq:amelioration of the Navier-Stokes system on R3}
\liminf_{t\rightarrow+\infty}\frac{\rad(u(t))}{\sqrt{t\ln(t)}}>0.
\end{equation}

In the following, we investigate whether an analytic smoothing effect occurs for \eqref{NSH}. Before stating our results, let us first introduce the notion of the radius of analyticity that we use in this article.\\
\textit{Radius of analyticity in $\Tilde{H}^d$ in the vertical direction.} For any horizontal vector field $f$ that belongs to $\Tilde{H}^d$, we define the radius of analyticity of $f$ with respect to the variable $s$, denoted by $\rad_s(f)$, by setting
\begin{equation}
\label{eq:Definition rayon d analyticite de Hd}
\rad_s(f):=\sup\enstq{R\geq 0}{e^{R|D_s|}f\in \Tilde{H}^d}.
\end{equation}
We show in Theorem \ref{theo:interpretation rayon d'analyticite Hg}, that, similarly to the case of homogeneous Sobolev spaces on $\R^d$ (see \cite{AnalyticityestimatesfortheNavier-Stokesequations}), if $\rad_s(f)\geq R>0$ then $f=g+h$ where $g$ belongs to $\Tilde{H}^d$ is an entire function with respect to the variable $s$ and $h$  belongs to $L^2(\Hg^d)$, and can be extended to a holomorphic function with respect to the variable $s$ on the strip $\enstq{z\in\C}{|\Im(z)|< R}$.

We obtain the following regularity result:

\begin{theo}[Regularity of the solution in $\Tilde{H}^d$]
\label{th main:Regularity of the solution in TildeHd}
Let $\sigma\in(0,4d)$. Then, for any small enough initial data $u_0\in\Tilde{H}^d$, the associated solution $u$ of \eqref{NSH}-\eqref{NSH initial condition} given by Theorem \ref{th main:Global well-posedness in TildeHd}  belongs to $\mathcal{C}^{\infty}((0,+\infty)\times\R^{2d+1})$ and for any $t>0$ and $(\alpha,\beta)\in\N\times\N$, we have 
\begin{equation}
\label{eq:Theorem 3 global regularity}
\partial_{s}^{\alpha}(-\Laplace_\Hg)^{\beta}u(t)\in\Tilde{H}^d\ \ \text{ and }\ \ \rad_s(u(t))\geq \sigma t.
\end{equation}
Consequently, the solution $u$ of \eqref{NSH}-\eqref{NSH initial condition} then is a strong solution of \eqref{NSH}. Moreover, the pressure $p$ belongs to $\mathcal{C}^{\infty}((0,+\infty)\times\R^{2d+1})$.
\end{theo}

Theorem \ref{th main:Regularity of the solution in TildeHd} is the direct combination of Theorem \ref{Analytic smoothing}, which establishes the estimate of the radius of analyticity with respect to the vertical variable in \eqref{eq:Theorem 3 global regularity}, and Corollary \ref{strong solution}, which deduces from it the regularity with respect to the other variables. In fact, as we will see, the difficult point is to prove the smoothness of solutions in terms of global regularity, namely \eqref{eq:Theorem 3 global regularity} with respect to the space variables. 

It is surprising that the estimate of the radius of analyticity in the vertical direction obtained in Theorem \ref{th main:Regularity of the solution in TildeHd} is better in  large times than the one obtained in \eqref{eq:amelioration of the Navier-Stokes system on R3} for solutions of the incompressible Navier-Stokes system in $\R^3$. In fact, a linear estimate of the radius of analyticity in the context of the incompressible Navier-Stokes system on the torus, similar to the one in \eqref{eq:Theorem 3 global regularity}, was obtained in \cite{GevreyclassregularityforthesolutionsoftheNavier-Stokesequations}, related to the fact that there is a spectral gap for the Laplace operator on the torus. In fact, in some sense, the estimate \eqref{eq:Theorem 3 global regularity} also originates from some kind of spectral gap (see Proposition \ref{prop- Ds est d'ordre 2}), but not from the spectral gap of the Heisenberg sublaplacian, whose spectrum is $\R_+$.

Note that we do not claim any analyticity property of the solutions of \eqref{NSH} in the horizontal directions, and this is so far an open problem. This might come as a surprise in view of the dissipation law \eqref{eq:Dissipation}. In fact, the dissipation law \eqref{eq:Dissipation} implies 
$$
	\frac{1}{2} \frac{d}{dt} \| u \|_{L^2}^2 + 4d \| |D_s|^{1/2} u \|_{L^2}^2 \leq 0, 
$$
see Proposition \ref{prop- Ds est d'ordre 2}. This is the key estimate to understand that the energy is dissipated in the direction of the commutators, and generates some analytic smoothing effect in the vertical direction.
\subsection{Outline of the article}\label{subsec:Outline of the article}

The article is organized as follows. In Section \ref{sec:Heisenberg group}, we describe some basic notions related to the Heisenberg group and several tools developed in \cite{QuantizationonNilpotentLieGroups,AfrequencyspacefortheHeisenberggroup,AnalysisontheHeisenberggroup} as the Fourier transform, which will sustain our strategy in this article. We finished this section by showing some useful results on the pseudo-differential operators, in the spirit of \cite{TempereddistributionsandFouriertransformontheHeisenberggroup}, with the quantization given by the Fourier transform on the Heisenberg group. The main ideas and insights of the strategy are summarized in Section \ref{sec:Main ideas}. In Section \ref{sec:Derivation of a suitable approximate system}, we introduce Friedrichs multipliers, the Leray projector, and present their properties. They are used to study the nonstationary Stokes system and to get suitable approximate systems of \eqref{NSH}. In Section \ref{sec:Weak solutions}, we prove the global existence of weak solutions with finite energy for any initial data in $L^2(\Hg^d)$ (Theorem \ref{th main:Existence of weak solutions}). Section \ref{sec:Well-posedness in TildeHd} is devoted to establishing the global well-posedness of the solutions to \eqref{NSH}-\eqref{NSH initial condition} with small initial data in $\Tilde{H}^d$ (Theorem \ref{th main:Global well-posedness in TildeHd}). In Section \ref{sec:Smoothing effects}, we establish the smoothness of the solutions of \eqref{NSH} (Theorem \ref{th main:Regularity of the solution in TildeHd}), first its analyticity with respect to the vertical variable $s$, and second its $\mathcal{C}^{\infty}$ smoothness with respect to all the variables. Section \ref{sec:Long time existence in dotHd} discusses the existence of solutions to \eqref{NSH}-\eqref{NSH initial condition} with initial data belonging to $\dot{H}^d$ and which are analytic with respect to the vertical variable.
Section \ref{sec:Comments on the generalization and open problems} presents the derivation of sub-Riemannian Navier-Stokes equations on a general stratified Lie group, and give an existence result of a weak solution in this setting in the energy space. This section ends up with some open problems. Finally, the appendices present several technical results.
\paragraph{Acknowledgments.} The author would like to thank Jean-Yves Chemin and Sylvain Ervedoza for their comments on this work and Hajer Bahouri, Jean-Yves Chemin and Rapha\"el Danchin for providing their manuscript  \cite{AnalysisontheHeisenberggroup}.
\section{Heisenberg group}\label{sec:Heisenberg group}
\subsection{Lie group structure}
\label{subsec:Lie group structure}
Let $d$ be a positive integer. The Heisenberg group $\Hg^d$ is the set $\R^{2d+1}$ endowed with the following group law 
\begin{equation}
\label{24/08/2022 14h 1}
    (Y,s)\cdot (Y',s'):=(Y+Y',s+s'+\langle \mathfrak{S}Y,Y'\rangle_{\R^{2d}}),
\end{equation}
in which $Y$ and $Y'$ are in $\R^{2d}$, $s$ and $s'$ are in $\R$ and $\mathfrak{S}$ is the matrix introduced in \eqref{matrice mathfrak S}. The group $\Hg^d$ is a noncommutative Lie group, with $0$ as the unit element and for which the inverse of an element $w\in\Hg$ is given by $w^{-1}:=-w$. We denote the generic elements of $\Hg^d$, as a couple $(Y,s)$, in which $Y=(y,\eta)\in\R^d\times\R^d$ is called the horizontal variable and $s\in\R$ is called the vertical variable.

\paragraph{Lie algebra and the sub-Riemannian structure of $\Hg^d$.} The Lie algebra of left-invariant vector fields on $\Hg^d$, denoted by $\mathfrak{h}^d$ is the Lie algebra spanned by the following vector fields
\begin{equation}
\label{I 28/07/2023}
P_{j}=X_j:=\partial_{y_j}+2\eta_{j}\partial_s\ \quad  P_{j+d}=\Xi_j:=\partial_{\eta_j}-2y_{j}\partial_s,\ \text{ with } j\in[\![ 1,d]\!], 
\ \text{ and }\
S : = - 4 \partial_s.
\end{equation}
The Lie algebra $\mathfrak{h}^d$ has the following gradation
$\mathfrak{h}^d=\mathfrak{h}_{1}^d\bigoplus\mathfrak{h}_{2}^{d}$, where $\mathfrak{h}_{1}^{d}:=\text{Vect}(\lbrace P_j\rbrace_{j\in[\![1,2d]\!]})$ and $\mathfrak{h}_{2}^{d}:=\text{Vect}(\{S\})$ and we have the following identity \begin{equation}
    \label{24/08/2022 14h 2}
    [X_j,\Xi_j]=S, \quad 
    \text{for any} \ j \in [\![1, d]\!],
\end{equation}
in which  $[A,B]:=AB-BA$ denotes the commutator of the operators $A$ and $B$. Accordingly, the Heisenberg group is a 2-step stratified Lie group, and is endowed with a sub-Riemannian structure (see \cite[Section 7.4, p. 209]{AComprehensiveIntroductiontoSubRiemannianGeometry}), where the space of horizontal vector fields is the first stratum, namely $\mathfrak{h}_{1}^{d}$. 

Note that the vector field $\partial_s$ is a bi-invariant vector field and that we have the following identities
\begin{equation}
\label{commutateur XkYj S}
[X_j,X_k]=[\Xi_j,\Xi_k]=[X_j,\partial_s]=[\Xi_j,\partial_s]=0,
 \text{ and } 
  [X_j,\Xi_k]= {\bf 1}_{j=k} S
 \quad \text{for all $j, k$ in $[\![1,d]\!]$}.
\end{equation}
The Lie group $\Hg^d$ is equipped with a bi-invariant Haar measure, which is simply the Lebesgue measure on $\R^{2d+1}$. We can write $\nabla_{\Hg}$, $\Laplace_{\Hg}$ and $\divergence_{\Hg}$, defined in \eqref{Operators sR avec matrice R}, by using the family of left-invariant vector fields $(P_j)_{j\in[\![1,2d]\!]}$ as follows. For any smooth enough function $f$ and horizontal vector field $u=(u_1,\dots,u_{2d})$ on $\Hg^d$, we have 
\begin{equation}
\label{HeisenbergOp-InTermsOf-Pj}
\nabla_{\Hg}f=\!^t(P_1f,\dots,P_{2d}f),\ \ \ \divergence_{\Hg}(u)=\sum_{j=1}^{2d}P_{j}u_{j}\ \ \ \text{and}\ \ \  \Laplace_{\Hg}f=\sum_{j=1}^{2d}P_{j}^{2}f.
\end{equation}
Note that the formula $\Laplace_{\Hg}=\divergence_{\Hg}\circ\nabla_{\Hg}$ holds.

Let us finally emphasize that the family of vector fields $(P_j)_{j\in[\![1,2d]\!]}$ satisfies the H\"ormander condition, and thus that the sub-Riemannian Laplacian $\Laplace_{\Hg}$ is hypoelliptic (see \cite{Hypoellipticsecondorderdifferentialequations}).

In a similar way, we define the right-invariant vector fields $\Tilde{P}_{j} = \Tilde{X}_j$ and $\Tilde{P}_{j+d} =  \Tilde{\Xi}_j$ with $j\in [\![1,d]\!]$ as follows
\begin{equation}
\label{Definition des champs invariants a droite}
\Tilde{P}_{j} = \Tilde{X}_j:=\partial_{y_j}-2\eta_{j}\partial_s
\quad\text{ and }\quad
 \Tilde{P}_{j+d} = \Tilde{\Xi}_j:=\partial_{\eta_j}+2y_{j}\partial_s.
\end{equation}
The Lie algebra $\Tilde{\mathfrak{h}}^d$ of the right-invariant vector fields of $\Hg^d$ is the Lie algebra generated by the family $(\Tilde{P}_j)_{j\in[\![1,2d]\!]}$. In view of the following formula
$$
[\Tilde{X}_j,\Tilde{\Xi}_j]=4\partial_s,
$$
with $j\in[\![1,d]\!]$, the Lie algebra $\Tilde{\mathfrak{h}}^d$ is also $2$-step stratified. We define the right-invariant sub-Laplacian on $\Hg^d$
\begin{equation}
\label{Definition sousLaplacien invariant à droite}
\Tilde{\Laplace}_{\Hg}f:=\sum_{j=1}^{2d}\Tilde{P}_{j}^{2}f.
\end{equation}

The main interest of right-invariant vector fields is that they commute with the left-invariant vector fields:  
\begin{equation}
\label{P Tilde P cummute}
[P_{i},\Tilde{P}_{j}]=0, \quad \text{for any $i$ and $j$ in $[\![1,2d]\!]$}.
\end{equation}

\begin{notation}
	We will use the following notation: for all $j\in\N$ and for any multi-index $\alpha\in[\![1,2d]\!]^j$,
	\begin{equation}
	\label{eq:notation pour les puissances de P et TildeP}
	P^{\alpha}:=P_{\alpha_1}\cdots P_{\alpha_{j}}\ \ \text{and}\ \ \ \Tilde{P}^{\alpha}:=\Tilde{P}_{\alpha_1}\cdots \Tilde{P}_{\alpha_{j}}.
\end{equation}
\end{notation}

\paragraph{Homogeneous Lie group structure of $\Hg^d$.} We now describe the homogeneous structure of the Heisenberg group, which provides properties for which the main tools of our analysis (scaling symmetry, Sobolev spaces associated with the sub-Laplacian, Fourier analysis and pseudo-differential operators on $\Hg^d$) are well developed in the literature (see for instance \cite{QuantizationonNilpotentLieGroups,TempereddistributionsandFouriertransformontheHeisenberggroup, SpectralsummabilityforthequarticoscillatorwithapplicationstotheEngelgroup, AnalysisontheHeisenberggroup}). The homogeneous structure of $\Hg^d$ is omnipresent in this article. The family of dilations $(\delta_\mu)_{\mu>0}$ on $\Hg^d$ is defined for any $\mu>0$ and $(Y,s)\in\Hg^d$, by
$$
\delta_{\mu}(Y,s):=(\mu Y,\mu^2s).
$$
Let us remark that for every $\mu>0$, for all smooth enough $f:\Hg^d\rightarrow\C$ and for all $j\in[\![1,2d]\!]$, we have
\begin{equation}
\label{I 25/01/2023}
P_{j}(f\circ \delta_{\mu})=\mu (P_{j}f)\circ\delta_{\mu},\ \ \Tilde{P}_{j}(f\circ \delta_{\mu})=\mu (\Tilde{P}_{j}f)\circ\delta_{\mu}\ \ \text{and}\ \ \partial_s(f\circ \delta_{\mu})=\mu^2 (\partial_s f)\circ\delta_{\mu}.
\end{equation}
This is, of course, reflected by the scaling invariance \eqref{scaling transformation} for solutions of \eqref{NSH}, given by $\mu\mapsto\mu u(\mu^2\cdot,\delta_\mu(\cdot))$. 
For all $\mu>0$, the Jacobian of the dilation $\delta_{\mu}$ is $\mu^{Q}$, where 
\begin{equation}
\label{eq:Definition dimension homogene}
Q:=2d+2
\end{equation}
is called the \textit{homogeneous dimension} of $\Hg^d$ and for every $f\in L^1(\Hg^d)$ we have 
$$
\int_{\Hg^d}f(\delta_{\mu}(w))dw=\mu^{-Q}\int_{\Hg^d}f(w)dw.
$$
As we will see next (see Proposition \ref{Propriete Sobolev sur Heisenberg}), the homogeneous dimension $Q$ plays the same role in the exponent of Sobolev's embedding for Sobolev spaces on $\Hg^d$ as the algebraic dimension $2d+1$ on the usual Sobolev embedding on $\R^{2d+1}$. 
\subsection{Fourier transform on the Heisenberg group}\label{subsec:Fourier transform}
\subsubsection{Definition of the Fourier transform}\label{subsubsec:Definition of the Fourier transform}
For every $\lambda\in\R^{*}$ and $w=(y,\eta,s)\in \Hg^d$, we define the bounded operator $\mathrm{U}_{w}^{\lambda}$ acting on $L^2(\R^d)$ by setting 
\begin{equation}
	\label{Representation-Schrodinger-Operator}
\mathrm{U}_{w}^{\lambda}u(x):=e^{-is\lambda-2i\lambda\langle \eta,x-y\rangle}u(x-2y),
\end{equation}
for $u$ in $L^2(\R^d)$ and $x$ in $\R^d$.
For any $w\in\Hg^d$, $\mathrm{U}_{w}^{\lambda}$ is a unitary operator on $L^2(\R^d)$. 
The family $(\mathrm{U}^{\lambda},L^2(\R^d))_{\lambda\in\R^{*}}$
describes all the equivalence classes of the unitary dual of $\Hg^d$. According to the definition of the Fourier transform on locally compact Lie groups (see \cite{QuantizationonNilpotentLieGroups}), we define the Fourier transform of $f\in L^1(\Hg^d)$ evaluated at $\mathrm{U}^{\lambda}$ by the following formula
\begin{equation}
\label{eq:definition Fourier in the Heisenberg group version operateur}
\mathscr{F}_{\Hg}(f)(\mathrm{U}^\lambda):=\int_{\Hg^d}f(w)\mathrm{U}^{\lambda}_{w}dw, \quad \lambda \in \R^*.
\end{equation}
\subsubsection{Frequency space approach}\label{subsubsec:Frequency space approach}

In order to get a precise description of the spectrum of $\Laplace_{\Hg}$ and $\Tilde{\Laplace}_{\Hg}$, we also use a description of the Fourier transform $\mathscr{F}_\Hg$ in terms of frequency space. This recent approach is developed in \cite{TempereddistributionsandFouriertransformontheHeisenberggroup, AfrequencyspacefortheHeisenberggroup, AnalysisontheHeisenberggroup}, and gives an alternative definition of the Fourier transform on $\Hg^d$, in which the Fourier modes are complex numbers and the Fourier transform of a function is a function defined on the following subset of $\Hg^d$
$$
\Tilde{\Hg}^d:=\N^d\times\N^d\times\R^{*}.
$$ 
For every $\lambda\in\R^{*}$, we consider the orthonormal basis $(h_{n,\lambda})_{n\in\N^d}$ on $L^2(\R^d)$ given by the rescaled Hermite functions on $\R^d$, defined for any $n\in\N^{d}$ by
$$
h_{n,\lambda}:=|\lambda|^{\frac{d}{4}}h_{n}(|\lambda|^{\frac{1}{2}}\cdot),
$$
where
$$
h_{n}:=\frac{1}{(2^{|n|}n!)^{\frac{1}{2}}}\prod_{j=1}^{d}(-\partial_{j}+x_j)^{n_j}h_0\ \ \text{and}\ \ \ h_{0}:=\pi^{-\frac{d}{2}}e^{-\frac{|\cdot|^2}{2}}.
$$
In particular, we have $\|h_{n,\lambda}\|_{L^2(\R^d)}=1$. According to \eqref{eq:definition Fourier in the Heisenberg group version operateur}, for each $\lambda \in \R^*$, the Fourier mode $\mathscr{F}_\Hg(f)(\mathrm{U}^{\lambda})$ of a function $f\in L^1(\Hg^d)$ is completely determined by the following family of complex numbers $(\mathcal{F}_{\Hg}(f)(n,m,\lambda))_{(n,m)\in\N^{2d}}$ given by
\begin{equation}
    \label{Definition de Fourier comme fonction}
    \mathcal{F}_{\Hg}(f)(n,m,\lambda):=\langle \mathscr{F}_{\Hg}(f)(\mathrm{U}^{\lambda})h_{m,\lambda},h_{n,\lambda}\rangle_{L^2(\R^d)}, \quad \text{ for } (n,m)\in \N^{2d}.
\end{equation}
For every $f\in L^1(\Hg^d)$, we thus define the function $\mathcal{F}_{\Hg}(f):(n,m,\lambda)\in\Tilde{\Hg}^d\mapsto \mathcal{F}_{\Hg}(f)(n,m,\lambda)\in\C$. This function $\mathcal{F}_{\Hg}(f)$ provides a suitable alternative definition of the Fourier transform of $f$, sought as a complex valued function defined on $\Tilde{\Hg}^d$. 

Now, we recall classical properties of this Fourier transform that can be found in \cite{QuantizationonNilpotentLieGroups, AfrequencyspacefortheHeisenberggroup, TempereddistributionsandFouriertransformontheHeisenberggroup}. We define the measure $d\widehat{w}$ in $\Tilde{\Hg}^d$ by setting, for any measurable functions $f:\Tilde{\Hg}^d\rightarrow\C$,
\begin{equation}
\label{eq:mesure d widehat w}
\int_{\Tilde{\Hg}^d}f(\widehat{w})d\widehat{w}:=\sum_{(n,m)\in\N^d\times\N^d}\int_{\R^{*}}f(n,m,\lambda)|\lambda|^dd\lambda.
\end{equation}
The measure $|\lambda|^dd\lambda$ corresponds, up to  multiplication by a constant, to the classical Plancherel measure on the Heisenberg group (see \cite[Example 1.8.4, p. 45]{QuantizationonNilpotentLieGroups}). Let us point out that the Schwartz space $\mathcal{S}(\Hg^d)$ and its dual $\mathcal{S}'(\Hg^d)$ respectively coincide with $\mathcal{S}(\R^{2d+1})$ and  $\mathcal{S}'(\R^{2d+1})$.

By using the Plancherel theorem and inversion theorem for $\mathscr{F}_\Hg$ (see \cite[ Theorem 1.8.5 and Corollary 1.8.6, p. 46]{QuantizationonNilpotentLieGroups}), we can deduce the corresponding theorem for the Fourier transform $\mathcal{F}_{\Hg}$ (see  \cite[Theorem 1.3]{TempereddistributionsandFouriertransformontheHeisenberggroup} and \cite{AnalysisontheHeisenberggroup} for an elementary and self-contained approach):
\begin{prop}
\label{inversion et Plancherel sur Heisenberg}
If $f\in \mathcal{S}(\Hg^d)$,  then for any $\lambda\in\R^{*}$, the operator $\mathscr{F}_{\Hg}(f)(\mathrm{U}^{\lambda})$ is a trace class (thus Hilbert-Schmidt)  operator on $L^2(\R^d)$ and for any $w\in\Hg^d$, we have
\begin{align*}
f(w)&=\frac{2^{d-1}}{\pi^{d+1}}\int_{\R}\Tr\left( \mathrm{U}_{w^{-1}}^{\lambda}\mathscr{F}_{\Hg}
(f)(\mathrm{U}^{\lambda})\right)|\lambda|^dd\lambda.
\end{align*}
Moreover, $\mathcal{F}_{\Hg}$ can be extended as a bi-continuous isomorphism from $L^2(\Hg^d)$ to $L^2(\Tilde{\Hg}^d)$, and then for every $f$ and $g$ in $L^2(\Hg^d)$, we get
$$
\langle \mathcal{F}_{\Hg}(f),\mathcal{F}_{\Hg}(g)\rangle_{L^2(\Tilde{\Hg}^d)}=\int_{\R}\Tr\left(\mathscr{F}_{\Hg}(f)(\mathrm{U}^{\lambda})\circ\mathscr{F}_{\Hg}(g)(\mathrm{U}^{\lambda})^{*}\right)|\lambda|^dd\lambda=\frac{\pi^{d+1}}{2^{d-1}}\langle f,g\rangle_{L^2(\Hg^d)},
$$
where $L^2(\Tilde{\Hg}^d):=L^{2}(\Tilde{\Hg}^d,d\widehat{w})$ and the measure $d\widehat{w}$ is defined in \eqref{eq:mesure d widehat w}.
\end{prop}
\subsection{Preliminary results on the pseudo-differential calculus on $\Hg^d$}
\label{subsubsec:Pseudo-differential calculus on Hd}
The most useful interest of the Fourier transform $\mathcal{F}_\Hg$ in this article is that (see \cite[p.6]{TempereddistributionsandFouriertransformontheHeisenberggroup}), for any $f\in\mathcal{S}(\Hg^d)$ and $(n,m,\lambda)\in\Tilde{\Hg}^d$, we have
\begin{align}
\label{eq:diago sousLaplacian gauche}
&\mathcal{F}_\Hg(-\Laplace_\Hg f)(n,m,\lambda) =4|\lambda|(2|m|+d)\mathcal{F}_\Hg(f)(n,m,\lambda),
\\
\label{eq:diago sousLaplacian droite}
&\mathcal{F}_\Hg(-\Tilde{\Laplace}_\Hg f)(n,m,\lambda) =4|\lambda|(2|n|+d)\mathcal{F}_\Hg(f)(n,m,\lambda).
\end{align}
These two formulas are the analog of the fact that the symbol of the Laplacian on $\R^d$ is exactly $-|\xi|^2$. Following this analogy, the operators $\Laplace_\Hg$ and $\Tilde{\Laplace}_\Hg$ are Fourier multipliers with respect to the quantization on the Lie group $\Hg^d$. 

In this paper, we will use the symbols of the powers of $-\Laplace_\Hg$ and $-\Tilde{\Laplace}_\Hg$ with respect to the Fourier transform on the Heisenberg group. In this subsection, we begin by introducing the definition of the powers of the sub-Laplacian as unbounded operators on $L^2(\Hg^d)$ and then give their symbols with respect to the Fourier transform $\mathcal{F}_\Hg$. The case of the negative powers requires more work (see Section \ref{sec:Symbol of the negative powers of the sub-Laplacian}, Lemma \ref{theo:symbole homogene}) in order to identify their symbols on the intersection of their domains and the Schwartz classes. We will also give the expression of $P_j$, $\Tilde{P}_j$ and $\partial_s$, which are pseudo-differential operators, with respect to the Fourier transform on $\Hg^d$.
\paragraph{Power of the sub-Laplacians.} In order to define the fractional powers of the Heisenberg sub-Laplacians on $L^2(\Hg^d)$, we use the functional calculus for $-\Laplace_{\Hg}$ and $-\Tilde{\Laplace}_{\Hg}$. The operator $-\Laplace_{\Hg}$ is a positive self-adjoint unbounded operator on $L^{2}(\Hg^d)$ with $\mathcal{S}(\Hg^d)\subset\mathrm{Dom}(-\Laplace_{\Hg})$. According to the spectral theorem, denoting by $E$ the spectral measure of $-\Laplace_\Hg$, for any measurable function $\psi:\R_+\rightarrow \R$, we can define an unbounded operator $\psi(-\Laplace_{\Hg})$ on $L^2(\Hg^d)$, with domain
$$
\mathrm{Dom}(\psi(-\Laplace_{\Hg})):=\enstq{f\in L^2(\Hg^d)}{\int_{0}^{+\infty}|\psi(\mu)|^2 d\langle E(\mu)f,f\rangle_{L^2(\Hg^d)}<+\infty},
$$
by the formula
\begin{equation}
\label{eq:definition psi(-Laplace)}
\psi(-\Laplace_\Hg):=\int_{0}^{+\infty}\psi(\mu)dE(\mu).
\end{equation}
Similarly, for $\Tilde{E}$ the spectral measure of $-\Tilde{\Laplace}_{\Hg}$,  for any measurable function $\psi:\R_+\rightarrow\R$, we set 
\begin{equation}
\label{eq:definition psi(-TildeLaplace)}
\psi(-\Tilde{\Laplace}_\Hg):=\int_{0}^{+\infty}\psi(\mu)d\Tilde{E}(\mu),
\end{equation}
with
$$
\mathrm{Dom}(\psi(-\Tilde{\Laplace}_{\Hg})):=\enstq{f\in L^2(\Hg^d)}{\int_{0}^{+\infty}|\psi(\mu)|^2 d\langle \Tilde{E}(\mu)f,f\rangle_{L^2(\Hg^d)}<+\infty}.
$$
Furthermore, if $\ell>- Q/4$, the space $\mathcal{S}(\Hg^d)$ is contained in $\mathrm{Dom}((-\Laplace_{\Hg})^{\ell})$ and $\mathrm{Dom}((-\Tilde{\Laplace}_{\Hg})^\ell)$ (see \cite[Proposition 4.4.13, Item 2, p. 230]{QuantizationonNilpotentLieGroups}).

Inspired by \eqref{eq:diago sousLaplacian droite} and \eqref{eq:diago sousLaplacian gauche}, we give the following representation formula by using the Fourier transform on $\Hg^d$.
\begin{prop}
\label{Fourier diagonalise the sublaplacian}
Let $\ell'>0$, $\ell\in\R$, and $f\in \mathcal{S}(\Hg^d)$. We have, for $(n, m, \lambda) \in \Tilde{\Hg}^d$,
\begin{align*}
&\mathcal{F}_{\Hg}((\Id-\Laplace_{\Hg})^{\ell} f)(n,m,\lambda)=(1+4|\lambda|(2|m|+d))^{\ell} \mathcal{F}_{\Hg}(f)(n,m,\lambda),\\ &\mathcal{F}_{\Hg}((\Id-\Tilde{\Laplace}_{\Hg})^{\ell} f)(n,m,\lambda)=(1+4|\lambda|(2|n|+d))^{\ell}\mathcal{F}_{\Hg}(f)(n,m,\lambda),\\
&\mathcal{F}_{\Hg}(\partial_s f)(n,m,\lambda)=i\lambda \mathcal{F}_{\Hg}(f)(n,m,\lambda)\ \ \text{and}\ \ \ \mathcal{F}_{\Hg}(|D_s|^{\ell'}f)(n,m,\lambda)=|\lambda|^{\ell'}\mathcal{F}_{\Hg}(f)(n,m,\lambda).
\end{align*}
If $f\in\mathcal{S}(\Hg^d)\cap \mathrm{Dom}((-\Laplace_{\Hg})^{\ell})$, we have, for $(n, m, \lambda) \in \Tilde{\Hg}^d$,
$$
\mathcal{F}_{\Hg}((-\Laplace_{\Hg})^{\ell} f)(n,m,\lambda)=(4|\lambda|(2|m|+d))^{\ell} \mathcal{F}_{\Hg}(f)(n,m,\lambda),
$$
and if $f\in\mathcal{S}(\Hg^d)\cap \mathrm{Dom}((-\Tilde{\Laplace}_{\Hg})^{\ell})$, we have, for $(n, m, \lambda) \in \Tilde{\Hg}^d$,
$$
\mathcal{F}_{\Hg}((-\Tilde{\Laplace}_{\Hg})^{\ell} f)(n,m,\lambda)=(4|\lambda|(2|n|+d))^{\ell}\mathcal{F}_{\Hg}(f)(n,m,\lambda). 
$$
\end{prop}

\begin{rem}
	To our knowledge, the action of the Fourier transform $\mathscr{F}_{\Hg}$ on the sub-Laplacians $-\Laplace_{\Hg}$ and $-\Tilde{\Laplace}_{\Hg}$ is known in the literature. However, this is to our knowledge the first explicit mention of how the Fourier transform $\mathcal{F}_{\Hg}$ acts simultaneously on the powers of both sub-Laplacians $-\Laplace_{\Hg}$ and $-\Tilde{\Laplace}_{\Hg}$, especially for negative powers $\ell$.
\end{rem}

\begin{proof}
For $(\Id-\Laplace_\Hg)^\ell$ and $(\Id-\Tilde{\Laplace}_{\Hg})^{\ell}$, this follows from Proposition \ref{prop:Remark definition des puissance fractionnaire} and the definition of $\mathcal{F}_\Hg$. The formula for $\partial_s$ and $|D_s|^{\ell'}$ follows immediately from the formula (see for instance \cite[ Equation (1.9)]{TempereddistributionsandFouriertransformontheHeisenberggroup})
\begin{equation}
\label{eq:Formul 1.9}
\mathcal{F}_\Hg(f)(n,m,\lambda)=\int_{\Hg^d}e^{-is\lambda}W(n,m,\lambda,Y)f(Y,s)dYds,
\end{equation}
for $(n,m,\lambda)\in\Tilde{\Hg}^d$, where, writing $Y=(y,\eta)$ with $y$ and $\eta$  in $\R^d$, $W$ is defined by 
\begin{equation}
	\label{Def-W-n-m-lambda}
W(n,m,\lambda,Y):=e^{is\lambda}\langle  \mathrm{U}_{(Y,s)}^{\lambda}h_{m,\lambda},h_{n,\lambda}\rangle_{L^2(\R^d)}=\int_{\R^d}e^{-2i\lambda\langle\eta,x-y\rangle}h_{m,\lambda}(x-2y)h_{n,\lambda}(x)dx,
\end{equation}
which thus does not depend on the variable $s$. 
For $(-\Laplace_{\Hg})^{\ell}$ and $(-\Tilde{\Laplace}_{\Hg})^{\ell}$, the formula follows from Proposition \ref{prop:Remark definition des puissance fractionnaire} (for $\ell\geq0$) and Lemma \ref{theo:symbole homogene} (for $\ell<0$) combined with the definition of $\mathcal{F}_\Hg$.
\end{proof}
Let us now describe the action of the Fourier transform $\mathcal{F}_{\Hg}$ on the left-invariant and the right-invariant operators in the following proposition (see for instance \cite[Proposition A.3]{TempereddistributionsandFouriertransformontheHeisenberggroup}).
\begin{prop}
\label{Prop- Forier X and Y}
Let $f\in\mathcal{S}(\Hg^d)$. Then, for any $j\in[\![1,d]\!]$, we have
\begin{equation}
    \mathcal{F}_{\Hg}(X_jf)=-\mathcal{M}^{+}_{j} \mathcal{F}_{\Hg}(f)\text{ and }\mathcal{F}_{\Hg}(\Xi_jf)=-\mathcal{M}^{-}_{j}\mathcal{F}_{\Hg}(f),
\end{equation}
where, for any $(n,m,\lambda)\in\Tilde{\Hg}^d$,
$$
\mathcal{M}^{+}_{j} \mathcal{F}_{\Hg}(f)(n,m,\lambda)=\sqrt{2|\lambda|}\begin{cases}
\sqrt{m_j+1}\mathcal{F}_{\Hg}(f)(n,m+e_j,\lambda)-\sqrt{m_j}\mathcal{F}_{\Hg}(f)(n,m-e_j,\lambda)\text{ if } m_j\neq 0,\\
\mathcal{F}_{\Hg}(f)(n,m+e_j,\lambda)\text{ if } m_j=0,\\
\end{cases}
$$
and $$
\mathcal{M}^{-}_{j} \mathcal{F}_{\Hg}(f)(n,m,\lambda)=i\frac{\sqrt{2}\lambda}{\sqrt{|\lambda|}}\begin{cases}
\sqrt{m_j+1}\mathcal{F}_{\Hg}(f)(n,m+e_j,\lambda)+\sqrt{m_j}\mathcal{F}_{\Hg}(f)(n,m-e_j,\lambda)\text{ if } m_j\neq 0,\\
\mathcal{F}_{\Hg}(f)(n,m+e_j,\lambda)\text{ if }m_j=0,\\
\end{cases}
$$
where $e_j=(e^{k}_{j})_{k\in [\![1,d]\!]}$ belongs to $\N^d$ and satisfies $e^{k}_{j}=1$ if $k=j$ and $e^{k}_{j}=0$ if $k\neq j$. Similarly, for any $j\in[\![1,d]\!]$, we have
\begin{equation}
    \mathcal{F}_{\Hg}(\Tilde{X}_jf)=\Tilde{\mathcal{M}}^{+}_{j} \mathcal{F}_{\Hg}(f)\text{ and }\mathcal{F}_{\Hg}(\Tilde{\Xi}_jf)=\Tilde{\mathcal{M}}^{-}_{j}\mathcal{F}_{\Hg}(f),
\end{equation}
where, for any $(n,m,\lambda)\in\Tilde{\Hg}^d$,
$$
\Tilde{\mathcal{M}}^{+}_{j} \mathcal{F}_{\Hg}(f)(n,m,\lambda)=\sqrt{2|\lambda|}\begin{cases}
\sqrt{n_j+1}\mathcal{F}_{\Hg}(f)(n+e_j,m,\lambda)-\sqrt{n_j}\mathcal{F}_{\Hg}(f)(n-e_j,m,\lambda)\text{ if }n_j\neq 0,\\
\mathcal{F}_{\Hg}(f)(n+e_j,m,\lambda)\text{ if } n_j=0,\\
\end{cases}
$$
and $$
\Tilde{\mathcal{M}}^{-}_{j} \mathcal{F}_{\Hg}(f)(n,m,\lambda)=i\frac{\sqrt{2}\lambda}{\sqrt{|\lambda|}}\begin{cases}
\sqrt{n_j+1}\mathcal{F}_{\Hg}(f)(n+e_j,m,\lambda)+\sqrt{n_j}\mathcal{F}_{\Hg}(f)(n-e_j,m,\lambda)\text{ if }n_j\neq 0,\\
\mathcal{F}_{\Hg}(f)(n+e_j,m,\lambda)\text{ if }n_j=0.\\
\end{cases}
$$
\end{prop}

\begin{proof}
For the left-invariant vector fields, the formulas can be found in \cite[Proposition A.3]{TempereddistributionsandFouriertransformontheHeisenberggroup}. In order to show the formulas for the right-invariant vector fields, we will use another approach, which allows to establish the case of the left-invariant vector fields and the right-invariant vector fields simultaneously using the pseudo-differential calculus on $\Hg^d$. Let us recall that we saw in the proof of Proposition \ref{prop:Remark definition des puissance fractionnaire} that, for all $\lambda \in \R^*$, $\mathrm{U}^{\lambda}(X_j)=-2\partial_{x_j}$ and $\mathrm{U}^{\lambda}(\Xi_j)=-2i\lambda x_j$. 

Besides, for any $f\in\mathcal{S}(\Hg^d)$, we have
$$
\mathscr{F}_{\Hg}(X_jf)(\mathrm{U}^\lambda)=\mathscr{F}_{\Hg}(f)(\mathrm{U}^\lambda)\circ \mathrm{U}^{\lambda}(X_j)\text{ and }\mathscr{F}_{\Hg}(\Tilde{X}_jf)(\mathrm{U}^\lambda)=\mathrm{U}^{\lambda}(X_j)\circ\mathscr{F}_{\Hg}(f)(\mathrm{U}^\lambda), 
$$
see \cite[Proposition 1.7.6, p. 40]{QuantizationonNilpotentLieGroups}. 

Then we can recover the formula for $X_j$ and $\Xi_j$ from the definition of $\mathcal{F}_\Hg$ (see \eqref{Definition de Fourier comme fonction}) and the 
following classical recurrence relation for the Hermite functions: for any $j\in[\![1,d]\!]$ and $m=(m_1,\dots,m_d)\in\N^d$ with $m_j\neq 0$, we have
\begin{align*}
& \partial_jh_{m,\lambda}=\frac{|\lambda|^{\frac 12}}{2}\left(\sqrt{2m_j}h_{m-e_j,\lambda}-\sqrt{2m_j+2}h_{m+e_j,\lambda}\right),
\\
&x_jh_{m,\lambda}=\frac{1}{2|\lambda|^{\frac 12}}\left(\sqrt{2m_j}h_{m-e_j,\lambda}+\sqrt{2m_j+2}h_{m+e_j,\lambda}\right).
\end{align*}
The formulas for $\Tilde{X}_j$ and $\Tilde{\Xi}_j$ follow similarly.
\end{proof}

Let us note that the maps $\mathcal{M}^{\pm}_{j}$ and $\Tilde{\mathcal{M}}^{\pm}_{j}$, with $j\in[\![1,d]\!]$, are not multiplicative operators. This reflects the fact that $X_j$ is not a Fourier multiplier on $\Hg^d$. In fact, it is a pseudo-differential operator with respect to the quantization on the Lie group $\Hg^d$.

\subsection{Sobolev spaces}
\label{sec:Sobolev spaces}
Let us now introduce the Sobolev-type spaces associated with $\Laplace_{\Hg}$ and $\Tilde{\Laplace}_{\Hg}$. In particular, all the results on these Sobolev spaces in this article are listed in \cite[Theorem 4.4.28, p. 246 and Theorem 4.4.29, p. 248]{QuantizationonNilpotentLieGroups}. For the definitions of Sobolev spaces, we refer to \cite[Definition 4.4.2, p. 219]{QuantizationonNilpotentLieGroups} in the inhomogeneous case and \cite[Definition 4.4.12, p. 230]{QuantizationonNilpotentLieGroups} in the  homogeneous case. The fractional powers of the sub-Laplacians can be defined in the $L^p(\Hg^d)$ framework by an abstract way (see \cite[Theorem 4.3.6, p. 203]{QuantizationonNilpotentLieGroups}). 
In the case of $L^2(\Hg^d)$, this definition coincides with the definition that uses the functional calculus, see \eqref{eq:definition psi(-Laplace)} and \eqref{eq:definition psi(-TildeLaplace)}, with $\psi(\mu)=\mu^\ell$ where $\ell\in\R$.
\begin{defi}
Let $p\in(1,+\infty)$ and $\ell\in\R$.
\begin{enumerate}[topsep=0pt,parsep=0pt,leftmargin=*]
    \item The (inhomogeneous) Sobolev space $W^{\ell,p}_{\Hg}(\Hg^d)$ is the set of tempered  distributions obtained by the completion of $\mathcal{S}(\Hg^d)$ with respect to the norm given for any $f\in\mathcal{S}(\Hg^d)$ by
    $$
    \|f\|_{W_{\Hg}^{\ell,p}}:=\|(\Id-\Laplace_{\Hg})^{\frac{\ell}{2}}f\|_{L^p}.
    $$
    \item We define the (homogeneous) Sobolev space $\dot{W}^{\ell,p}_{\Hg}(\Hg^d)$ as the set of tempered distributions obtained by the completion of $\mathcal{S}(\Hg^d)\cap\mathrm{Dom}((-\Laplace_{\Hg})_{p}^{\frac{\ell}{2}})$ with respect to the norm given for any $f\in\mathcal{S}(\Hg^d)\cap \mathrm{Dom}((-\Laplace_{\Hg})_{p}^{\frac{\ell}{2}})$ by
    $$
     \|f\|_{\dot{W}_{\Hg}^{\ell,p}}:=\|(-\Laplace_{\Hg})^{\frac{\ell}{2}}f\|_{L^p},
    $$
where $\mathrm{Dom}((-\Laplace_{\Hg})_{p}^{\frac{\ell}{2}})$ is the domain of $(-\Laplace_{\Hg})^{\frac{\ell}{2}}$ on $L^p(\Hg^d)$ (see \cite[Theorem 4.3.6, p. 203]{QuantizationonNilpotentLieGroups}).
\item If $p=2$, we set 
$
H^\ell(\Hg^d):=W^{\ell,2}_{\Hg}(\Hg^d)\ \text{ and }\ \dot{H}^{\ell}(\Hg^d):=\dot{W}^{\ell,2}_{\Hg}(\Hg^d),
$
which are endowed respectively with the following scalar products 
$$
\langle\cdot,\cdot\rangle_{H^\ell}:=\langle(\Id-\Laplace_{\Hg})^{\frac{\ell}{2}}\cdot,(\Id-\Laplace_{\Hg})^{\frac{\ell}{2}}\cdot\rangle_{L^2}\ \text{ and }\  \langle\cdot,\cdot\rangle_{\dot{H}^\ell}:=\langle(-\Laplace_{\Hg})^{\frac{\ell}{2}}\cdot,(-\Laplace_{\Hg})^{\frac{\ell}{2}}\cdot\rangle_{L^2}.
$$
\end{enumerate}
\end{defi}
By definition, the spaces $W_{\Hg}^{\ell,p}$ and $\dot{W}_{\Hg}^{\ell,p}$ are Banach spaces and $H^\ell(\Hg^d)$ and $\dot{H}^\ell(\Hg^d)$ are Hilbert spaces. Let us note that for all $p\in(1,+\infty)$, if $\ell>-Q/p$, then $
\mathcal{S}(\Hg^d)\subset \mathrm{Dom}((-\Laplace_{\Hg})^{\ell/2}_{p})$ (see \cite[Proposition 4.4.13, Item 2, p. 230]{QuantizationonNilpotentLieGroups}), so that $\mathcal{S}(\Hg^d)\cap \mathrm{Dom}((-\Laplace_{\Hg})_{p}^{{\ell}/{2}}) = \mathcal{S}(\Hg^d)$.

The following proposition is continuously used in this article.
\begin{prop}[{\cite[Theorem 4.4.16, p. 233]{QuantizationonNilpotentLieGroups}}]\label{Continuite des champs horizonteaux}
For any $\ell\in\N$, $\alpha\in[\![1,2d]\!]^\ell$, $\ell'\in\R$ and $p\in(1,+\infty)$, the operators $P^{\alpha}$ (recall the notation \eqref{eq:notation pour les puissances de P et TildeP}) are bounded from $W_{\Hg}^{\ell',p}(\Hg^d)$ to $W_{\Hg}^{\ell'-\ell,p}(\Hg^d)$ and from $\dot{W}_{\Hg}^{\ell',p}(\Hg^d)$ to $\dot{W}_{\Hg}^{\ell'-\ell,p}(\Hg^d)$.
\end{prop}

In the following, we will use the regularity properties corresponding to the left-invariant sub-Laplacian $\Laplace_{\Hg}$ and to the right-invariant sub-Laplacian $\Tilde{\Laplace}_{\Hg}$. Therefore, since the Sobolev spaces $H^\ell(\Hg^d)$ and $\dot{H}^\ell(\Hg^d)$ measure only the regularity with respect to $\Laplace_{\Hg}$,  we also introduce the following Sobolev spaces corresponding to the regularity properties with respect to $\Tilde{\Laplace}_{\Hg}$.
\begin{defi}
Let $\ell\in\R$ and $p\in(1,+\infty)$. We define the Sobolev space $\Tilde{W}_{\Hg}^{\ell,p}(\Hg^d)$ as the subspace of $\mathcal{S}'(\Hg^d)$ obtained as the completion of $\mathcal{S}(\Hg^d)\cap \mathrm{Dom}((-\Tilde{\Laplace}_{\Hg})_{p}^{\frac{\ell}{2}})$ for the norm defined for any $f\in\mathcal{S}(\Hg^d)\cap \mathrm{Dom}((-\Tilde{\Laplace}_{\Hg})_{p}^{\frac{\ell}{2}})$ by 
$\|f\|_{\Tilde{W}_{\Hg}^{\ell,p}}:=\|(-\Tilde{\Laplace}_{\Hg})^{\frac{\ell}{2}}f\|_{L^p},$
where $\mathrm{Dom}((-\Tilde{\Laplace}_{\Hg})^{\frac{\ell}{2}}_{p})$ denotes the domain of $(-\Tilde{\Laplace}_{\Hg})^{\frac \ell2}$ in $L^p(\Hg^d)$ (see \cite[Theorem 4.3.6, p. 203]{QuantizationonNilpotentLieGroups}). We also set 
$
\Tilde{H}^\ell(\Hg^d):=\Tilde{W}_{\Hg}^{\ell,2}(\Hg^d),
$
which we equip with the following scalar product
$$
\langle\cdot,\cdot\rangle_{\Tilde{H}^\ell}:=\langle (-\Tilde{\Laplace}_{\Hg})^{\frac{\ell}{2}}\cdot,(-\Tilde{\Laplace}_{\Hg})^{\frac{\ell}{2}}\cdot\rangle_{L^2}.
$$
\end{defi}
Accordingly, the space $\Tilde{W}_{\Hg}^{\ell,p}(\Hg^d)$ is a Banach space and $\Tilde{H}^\ell(\Hg^d)$ is a Hilbert space.\\ Moreover, using the formulas of Proposition \ref{Fourier diagonalise the sublaplacian}, we deduce the following homogeneity properties for the homogeneous Sobolev norms: for any $\mu>0$ and $\ell\in\R$, we have
\begin{equation}
\label{eq:calcul de l invariance d echelle des norme Sobolev}
\forall u \in \dot{H}^\ell, \quad \|u\circ\delta_\mu\|_{\dot{H}^\ell}=\mu^{\ell-Q/2}\|u\|_{\dot{H}^\ell}\ 
\text{ and }\
\forall v \in \Tilde{H}^\ell, \quad  \|v\circ\delta_\mu\|_{\Tilde{H}^\ell}=\mu^{\ell-Q/2}\|v\|_{\Tilde{H}^\ell}.
\end{equation}
\begin{rem}
\label{Remarque sur les espaces de Sobolev gauche et droite}
The Sobolev spaces $\dot{H}^\ell(\Hg^d)$ and $\Tilde{H}^\ell(\Hg^d)$ are not comparable in general. In particular, for $d=1$ one can find a function $f\in\mathcal{C}^{\infty}(\R^3)$ such that $X_1 f$ and $\Xi_1 f$ belong to $L^2(\Hg^1)$ and $\Tilde{X}_1f\notin L^2(\Hg^1)$ (see \cite[Example 4.4.32, p. 250]{QuantizationonNilpotentLieGroups}).
\end{rem}
We have the following properties (see \cite[Theorem 4.4.28, p. 246]{QuantizationonNilpotentLieGroups}).
\begin{prop}
\label{Propriete Sobolev sur Heisenberg}
\begin{enumerate}[topsep=0pt,parsep=0pt,leftmargin=*]
\item Let $\ell\in\N$ and $p\in (1,+\infty)$. Then we have the following norm equivalence
$$
 \|\cdot\|_{\dot{W}^{\ell,p}_{\Hg}}\sim\sum_{\alpha\in[\![1,2d]\!]^{\ell}}\|P^{\alpha}\cdot\|_{L^p},
$$
where the operators $P^\alpha$ with $\alpha\in\N^d$ are defined in \eqref{eq:notation pour les puissances de P et TildeP}.
\item The space $\mathcal{S}(\Hg^d)$ is dense in $W^{\ell,p}_{\Hg}(\Hg^d)$ and in $\dot{W}^{\ell,p}_{\Hg}(\Hg^d)$ if $\ell>-Q/p$ and $p\in(1,+\infty)$.
\item If $1<p<q<+\infty$, and $\ell$ and $\ell'$ are real numbers satisfying $\ell'-\ell=Q(1/p-1/q)$, then $\dot{W}^{\ell',p}_{\Hg}(\Hg^d)\hookrightarrow \dot{W}^{\ell,q}_{\Hg}(\Hg^d)$ and $W^{\ell',p}_{\Hg}(\Hg^d)\hookrightarrow W^{\ell,q}_{\Hg}(\Hg^d)$. Moreover, if $p\in(1,+\infty)$ and $\ell>Q/p$, then $W_{\Hg}^{\ell,p}(\Hg^d)\hookrightarrow \mathcal{C}(\Hg^d)\cap L^{\infty}(\Hg^d)$.
\item If $\ell\in\R$, $p$ and $q$ belong to $(1,+\infty)$ and satisfy $1/p+1/q=1$, then the dual of $W^{\ell,p}_{\Hg}(\Hg^d)$ is $W^{-\ell,q}_{\Hg}$ and the dual of $\dot{W}^{\ell,p}_{\Hg}(\Hg^d)$ is $\dot{W}^{-\ell,q}_{\Hg}$.
\end{enumerate} 
\end{prop}
\begin{rem}
\label{Remark- Gauche à droite}
If we replace $\dot{W}^{\ell,p}(\Hg^d)$ by $\Tilde{W}_{\Hg}^{\ell,p}(\Hg^d)$, $\Laplace_{\Hg}$ by $\Tilde{\Laplace}_{\Hg}$ and $P_j$ by $\Tilde{P}_j$, then Proposition \ref{Propriete Sobolev sur Heisenberg} and \ref{proprietes des Sobolev H} holds for $\Tilde{W}_{\Hg}^{\ell,p}(\Hg^d)$. It follows from Proposition \ref{Continuite des champs horizonteaux}
 that for any $\ell\in\N$, $\alpha\in[\![1,2d]\!]^\ell$ and $\ell'\in\R$, the operators $\Tilde{P}^{\alpha}$ are bounded from $\Tilde{H}^{\ell'}(\Hg^d)$ to $\Tilde{H}^{\ell'-\ell}(\Hg^d)$.
\end{rem}
The following results concern the product estimates (see \cite{AnalysisontheHeisenberggroup}).
\begin{prop}
\label{proprietes des Sobolev H} 
 If $\ell_1$ and $\ell_2$ belong to $(-Q/2,Q/2)$ are such that $\ell_1+\ell_2>0$, there is a constant $C_{\ell_1,\ell_2}$ such that for any $f\in\dot{H}^{\ell_1}(\Hg^d)\cap L^{2}(\Hg^d)$ and $g\in\dot{H}^{\ell_2}(\Hg^d)\cap L^{2}(\Hg^d)$, the product $fg$ belongs to $\dot{H}^{\ell_1+\ell_2-Q/2}(\Hg^d)$ and 
\begin{equation}
\label{Moser estimate}
\|fg\|_{\dot{H}^{\ell_1+\ell_2-Q/2}}\leq C_{\ell_1,\ell_2} \|f\|_{\dot{H}^{\ell_1}}\|g\|_{\dot{H}^{\ell_2}}.
\end{equation}
\end{prop}
Let us finish with the following proposition which is crucial in this article.
\begin{prop}
\label{prop- Ds est d'ordre 2}
For $\ell \geq 0$, for every $f\in\dot{H}^{2\ell}(\Hg^d)$ (respectively $f\in\Tilde{H}^{2\ell}(\Hg^d)$), we have
    $$
    \||D_s|^{\ell}f\|_{L^2}\leq \frac{1}{(4d)^{\ell}}\|f\|_{\dot{H}^{2\ell}}\ \ \ \left(\text{respectively }\ \||D_s|^{\ell}f\|_{L^2}\leq \frac{1}{(4d)^{\ell}}\|f\|_{\Tilde{H}^{2\ell}}\right),    
    $$
where $|D_s|^{\ell}$ is the Fourier multiplier on $\R^{2d+1}$ of symbol $|\xi_{2d+1}|^{\ell}$.
\end{prop}
\begin{proof}
The proof is a direct consequence of Proposition \ref{Fourier diagonalise the sublaplacian}, since $4|\lambda|(2|m|+d)\geq 4d|\lambda|$ and $4|\lambda|(2|n|+d)\geq 4d|\lambda|$ respectively.
\end{proof}
Note that according to the homogeneity properties of $|D_s|$ with respect to the dilation $\delta$ and of the homogeneous Sobolev norms (see \eqref{eq:calcul de l invariance d echelle des norme Sobolev}), the Sobolev indexes in Proposition \ref{prop- Ds est d'ordre 2} are optimal.

\subsection{Order of commutators on $\Hg^d$}\label{subsec:Order of commutators on}
The first difficulty that appears in this article is the following: \textit{the commutator of two pseudo-differential operators on $\Hg^d$ is in general the sum of the orders of the two pseudo-differential operators}. In particular, we cannot expect to gain in regularity by estimating the commutator of two pseudo-differential operators, compared to the estimate of the product (in the Sobolev scales ${H}^\ell(\Hg^d)$ and / or $\Tilde{H}^\ell(\Hg^d)$). In this text we refer to the order of an operator on $\Hg^d$ with respect to a sub-Laplacian, defined as follows:
\begin{defi}
\label{def: Ordre operateur}
     Let $\ell$ be a real number and $\mathrm{T}:\mathcal{S}(\Hg^d)\rightarrow\mathcal{S}'(\Hg^d)$. We say that $\mathrm{T}$ is of an operator of order $\ell$ with respect to the left-invariant sub-Laplacian (respectively to the right-invariant sub-Laplacian) if
    \begin{enumerate}[topsep=0pt,parsep=0pt,leftmargin=*]
       \item $\mathrm{T}$ is \textit{homogeneous of degree $\ell$}: $\mathrm{T}(\vphi\circ\delta_{\mu})=\mu^\ell(\mathrm{T}\vphi)\circ\delta_\mu$ for any $\mu>0$ and $\vphi\in\mathcal{S}(\Hg^d)$;
       \item the operators $\mathrm{T}\circ(\Id-\Laplace_{\Hg})^{-\ell}$ and $(\Id-\Laplace_{\Hg})^{-\ell}\circ \mathrm{T}$ (respectively $\mathrm{T}\circ(\Id-\Tilde{\Laplace}_{\Hg})^{-\ell}$ and $(\Id-\Tilde{\Laplace}_{\Hg})^{-\ell}\circ \mathrm{T}$) belong to $\mathcal{L}(L^2)$.  
    \end{enumerate}
If $h$ and $h'$ are two positive integers and $\mathrm{T}:\mathcal{S}(\Hg^d)^h\rightarrow\mathcal{S}'(\Hg^d)^{h'}$ is a bounded operator, then the same terminology occurs with the obvious modifications on the definition.
\end{defi}
It follows from Proposition \ref{prop- Ds est d'ordre 2} that $|D_s|^{\frac{1}{2}}$ and $\partial_s=-\frac{1}{4}[X_i,\Xi_i]=\frac{1}{4}[\Tilde{X}_j,\Tilde{\Xi}_j]$ for $i$ and $j$ in $[\![1, d]\!]$ are respectively operators of order $1$ and $2$ with respect to both the left-invariant sub-Laplacian and the right-invariant sub-Laplacian.

Note that, using the homogeneity of the item $1$ of Definition \ref{def: Ordre operateur}, the commutator of two operators $\mathrm{T}_1$ of order $\ell_1$ and $\mathrm{T}_2$ of order $\ell_2$ is of order $\ell_1 + \ell_2$ except if $[\mathrm{T}_1, \mathrm{T}_2] = 0$, i.e. when $\mathrm{T}_1$ and $\mathrm{T}_2$ commute.
\section{Main ideas}
\label{sec:Main ideas}
The main consequence of the noncommutativity properties of operators $P_j$ (see \eqref{24/08/2022 14h 2}) in the sub-Riemannian Navier-Stokes system on $\Hg^d$ is that the computation of the pressure from the equation \eqref{NSH} involves a term depending linearly on $u$, so that $\nabla_{\Hg} p$ involves an operator of order $2$ in $u$ and consequently is part of the diffusive operator.
\paragraph{Pressure and loss of derivatives.} Let us explain this more precisely. In order to compute formally the pressure, we apply the divergence on the first line of \eqref{NSH}. According to the free divergence condition, it follows that
\begin{equation}
	\label{p-function-of-u}
p=(-\Laplace_{\Hg})^{-1}\divergence_{\Hg}(-\Laplace_{\Hg}u)+(-\Laplace_{\Hg})^{-1}\divergence_{\Hg}(u\cdot\nabla_{\Hg}u).
\end{equation}
We then define the horizontal Leray projector $\mathbb{P}$ on $\Hg^d$ by
\begin{equation}
\label{def: Leray projector}
\mathbb{P}:=\Id+\nabla_{\Hg}\circ(-\Laplace_{\Hg})^{-1}\circ\divergence_{\Hg}.
\end{equation}
Thus System \eqref{NSH} is formally equivalent to 
\begin{equation*}
\begin{cases}
\partial_tu-\Laplace_{\Hg}u+(\Id-\mathbb{P})\circ \Laplace_{\Hg}u+\mathbb{P}(u\cdot\nabla_{\Hg}u)=0\ &\text{in}\ \R_+\times\Hg^d,\\
\divergence_{\Hg}(u)=0\ &\text{in}\ \R_+\times\Hg^d,
\end{cases}
\end{equation*}
since
$$
-\nabla_\Hg p=(\Id-\mathbb{P})\left(-\Laplace_{\Hg}u+u\cdot\nabla_{\Hg}u\right).
$$
The difficulty here comes from the term 
$$
(\Id-\mathbb{P})\circ\Laplace_{\Hg}u,
$$
which is of order $2$ in $u$ with respect to the left-invariant sub-Laplacian (see Definition \ref{def: Ordre operateur}), even when we restrict the operator to the space of divergence free vector fields, see Lemma \ref{I 28/08/2023}. 

Indeed, when $\divergence_{\Hg}(u)=0$, we obviously have $(\Id-\mathbb{P})u=0$, so one might expect that 
$$
	(\Id-\mathbb{P})\circ\Laplace_{\Hg}u=[(\Id-\mathbb{P}),\Laplace_{\Hg}]u
$$
would provide some gain of regularity (which is the usual property of commutators when considering the usual Sobolev scales $H^{\ell}(\R^{2d+1})$, but these spaces are not adapted to our case since $\Laplace_\Hg$ is not elliptic). As we mentioned in Subsection \ref{subsec:Order of commutators on}, this is not the case in general when considering pseudo-differential operators on $\Hg^d$, and we will indeed prove in Lemma \ref{I 28/08/2023} that there is no gain there for the operators $(\Id-\mathbb{P})$ and $\Laplace_{\Hg}$.\footnote{The presence of the above term contrasts with the system introduced in \cite{AnexistenceanduniquenessresultfortheNavierStokestypeequationsontheHeisenberggroup} and the Navier-Stokes system on the Heisenberg group $\Hg^d$, where the corresponding Leray projector commutes with the diffusive operator.}
\paragraph{$L^2$ energy estimates and weak solutions.} Despite this fact, the pressure term can be canceled using the following property: for $u$ satisfying $\divergence_{\Hg}(u) = 0$, 
\begin{equation}
\label{I 2 11 2022}
\langle\nabla_{\Hg}p,u\rangle_{L^2}=-\langle p,\divergence_{\Hg}(u)\rangle_{L^2}=0,
\end{equation}
which is valid if $u$ is sufficiently smooth. Note that this is the cancellation that appears in the derivation of the dissipation law of the kinetic energy \eqref{eq:Dissipation}. This is useful when dealing with weak solutions lying in the energy space as in Definition \ref{def:Definition of weak solution for NSH} and will be of primary importance in the proof of Theorem \ref{th main:Existence of weak solutions}.
\paragraph{Critical framework and right-invariant vector fields.} But this $L^2(\Hg^d)$ space is not scaling invariant for \eqref{NSH}, in the sense that the transform \eqref{scaling transformation} does not preserve the $L^\infty(L^2)$-norm. A natural critical (scaling invariant) space for \eqref{NSH}--\eqref{NSH initial condition} would be the space $\dot{H}^d(\Hg^d)$, 
 but we cannot cancel the pressure term as in \eqref{I 2 11 2022} for the $\dot{H}^d(\Hg^d)$ scalar product because $\nabla_{\Hg}$ does not commute with left-invariant operators $P_i$ with $i\in[\![1,2d]\!]$. Indeed, in this space the computation of the energy corresponding to the linear part of \eqref{NSH} yields 
$$
\frac 12\frac{d}{dt}\|u\|_{\dot{H}^d}^{2}+\|\nabla_\Hg u\|_{\dot{H}^d}^{2}+\langle (\Id-\mathbb{P})\circ\Laplace_\Hg u,u\rangle_{\dot{H}^d}.
$$
Yet, the third term $\langle (\Id-\mathbb{P})\circ\Laplace_\Hg u,u\rangle_{L^{2}}$ has no sign and is of main order. Thus if this term is too large, the dissipation fails.
Nevertheless, if $\mathrm{T}$ is an operator that commutes with all vector fields $P_i$ with $i\in[\![1,2d]\!]$, we have
$$
\langle \mathrm{T}\nabla_{\Hg}p,\mathrm{T}u\rangle_{L^2}=-\langle \mathrm{T}p, \mathrm{T}\divergence_{\Hg}(u)\rangle_{L^2}=0.
$$
This is the main strategy that we use to make the energy methods work for existence results. In this article, we will in particular use $\mathrm{T}:=(-\Tilde{\Laplace}_{\Hg})^{\frac{d}{2}}$, which will eventually provide the existence part of Theorem \ref{th main:Global well-posedness in TildeHd}. In a nutshell, the regularity with respect to the right-invariant vector fields is propagated in the energy estimate.

\paragraph{Loss of derivatives in nonlinearity and right-invariant vector fields.} In this paper, this choice is also used in order to gain one derivative in the commutator estimates involving the convection operator $u\cdot\nabla_\Hg$. Let us explain this in more detail in the case $d=1$. Let $u$ and $v$ be two smooth enough horizontal vector fields on $\Hg^1$. If $\Tilde{\mathcal{Z}}\in\{\Tilde{X}_1,\Tilde{\Xi}_1\}$, then, since $[\Tilde{\mathcal{Z}},\nabla_\Hg]=0$, we have 
\begin{equation}
\label{eq:bon commutateurs}
[\Tilde{\mathcal{Z}},u\cdot\nabla_{\Hg}]v=\Tilde{\mathcal{Z}}u\cdot\nabla_\Hg v,
\end{equation}
while, if $\mathcal{Z}\in\{X_1,\Xi_1\}$, we have
\begin{equation}
\label{eq:mauvais commutateurs}
[\mathcal{Z},u\cdot\nabla_{\Hg}]v=\mathcal{Z}u\cdot\nabla_\Hg v+u\cdot[\mathcal{Z},\nabla_\Hg]v.
\end{equation}
Since $[\mathcal{Z},\nabla_\Hg]$ is of order $2$ (see Proposition \ref{prop- Ds est d'ordre 2}), we lose one derivative in \eqref{eq:mauvais commutateurs} compared to \eqref{eq:bon commutateurs}.
This is the main point in the proof of the stability part of Theorem \ref{th main:Global well-posedness in TildeHd}.
\paragraph{Vertical smoothing effects.} The crucial idea to obtain the smoothing effects in the vertical direction is to use the dissipation of the energy for the sub-elliptic heat equation on the Heisenberg group. Let us explain this. Let $u\in \mathcal{C}_{b}(\R_+;L^{2})\cap L^2(\R_+;\dot{H}^1(\Hg^d))$ be a smooth enough (so that the following calculus makes sense) solution of 
$$
\partial_t u-\Laplace_\Hg u=0\ \ \ \ \text{ in }(0,+\infty)\times\Hg^d.
$$
We have
$$
\frac{1}{2}\frac{d}{dt}\|u\|^{2}_{L^2}+\|\nabla_\Hg u\|^{2}_{L^2}=0.
$$
Let $\sigma>0$. If we set $U(t):=e^{\sigma t|D_s|}u(t)$ for any $t>0$, then 
$$
\partial_tU-\Laplace_\Hg U=\sigma|D_s|U\ \ \ \ \text{ in }(0,+\infty)\times\Hg^d.
$$
In view of Proposition \ref{prop- Ds est d'ordre 2}, we have
$$
\sigma\langle |D_s|U,U\rangle_{L^2} = \sigma\||D_s|^{\frac{1}{2}}U\|^{2}_{L^2} \leq \frac{\sigma}{4d}\|\nabla_\Hg U\|^{2}_{L^2}.
$$
It follows that
\begin{equation}
	\label{Diss-Law+Smoothness-heat-Eq}
\frac{1}{2}\frac{d}{dt}\|U\|^{2}_{L^2}+\left(1-\frac{\sigma}{4d}\right)\|\nabla_\Hg U\|^{2}_{L^2}\leq 0.
\end{equation}
Accordingly, for $\sigma<4d$, the function $t\mapsto\|e^{\sigma t|D_s|}u(t)\|_{L^2}^{2}$ is decreasing and we get
$$
\forall t>0,\ \ \ \|e^{\sigma t|D_s|}u(t)\|_{L^2}^{2}\leq \|u(0)\|_{L^2}^{2},
$$
that is, $u(t)$ is analytic with respect to the variable $s$ for any $t>0$. This is the underlying idea for the proof of the analytic regularity in the vertical variable stated in Theorem \ref{th main:Regularity of the solution in TildeHd}.

\paragraph{Horizontal smoothing effects.} In this paragraph, we give the flavor of the arguments to derive  estimates on the solution of \eqref{NSH} in $\dot H^k(\Hg^d) $ from the vertical analytic regularizing properties derived above. We do that with $k = 1$, as the general case follows similarly.

Let $u$ be a smooth solution of
\begin{equation*}
\begin{cases}
\partial_tu-\Laplace_{\Hg}u+(\Id-\mathbb{P})\circ \Laplace_{\Hg}u=0\ &\text{in}\ \R_+\times\Hg^d,\\
\divergence_{\Hg}(u)=0\ &\text{in}\ \R_+\times\Hg^d,\\
u_{|_{t=0}}=u_0\ &\text{in}\ \Hg^d,
\end{cases}
\end{equation*}
with $u_0\in L^2(\Hg^d)$. Let $\sigma\in(0,4d)$. Then, similarly as in \eqref{Diss-Law+Smoothness-heat-Eq}, we can obtain
$$
\|e^{\sigma t|D_s|}u\|_{L^{\infty}(L^2)}^{2}+\|\nabla_\Hg e^{\sigma t|D_s|}u\|_{L^2(L^2)}\leq \frac{\|u_0\|_{L^2}^{2}}{\min\{1,2-\sigma/(2d)\}}.
$$
Then there exists $t_0>0$ such that $u(t_0)\in\dot{H}^1(\Hg^d)$. Thus, we have
$$
\frac{1}{2}\frac{d}{dt}\|u\|_{\dot{H}^1}^{2}+\|\nabla_\Hg u\|_{\dot{H}^1}^{2}\leq|\langle (\Id-\mbb{P})\circ (-\Laplace_\Hg)u,u\rangle_{\dot{H}^1}|\ \ \text{on } (t_0,+\infty).
$$
This implies 
\begin{align}
\|u\|_{L^{\infty}((t_0,+\infty);\dot{H}^1)}^{2}+2\|\nabla_{\Hg}u&\|_{L^2((t_0,+\infty);\dot{H}^1)}^{2}\nonumber\\
&\leq 2\|u(t_0)\|_{\dot{H}^1}^{2}+\int_{t_0}^{+\infty}|\langle (\Id-\mathbb{P})\circ(-\Laplace_\Hg) u,u\rangle_{\dot{H}^1}dt.\label{eq:energy Stokes lineare H1}
\end{align}
The main point now is the following identity
$$
(\Id-\mathbb{P})\circ(-\Laplace_\Hg)u=\Pi_\Hg\partial_s u,
$$
proved in Lemma \ref{I 28/08/2023}, with $\Pi_\Hg\in\mathcal{L}(\dot{H}^1)$ commuting with $\partial_s$. Then, we have
$$
\forall t>0,\ \ |\langle (\Id-\mathbb{P})\circ(-\Laplace_\Hg)u,u\rangle_{\dot{H}^1}|\leq\|\Pi_\Hg\|_{\mathcal{L}(\dot{H}^1)}\||D_s|^{\frac 12}u\|_{\dot{H}^1}^{2}.
$$
Since $t_0>0$ and $e^{\sigma t|D_s|}u\in L^2(\dot{H}^1)$, we have $|D_s|^{\frac 12}u\in L^2((t_0,+\infty);\dot{H}^1)$. Thus, the right-hand side of \eqref{eq:energy Stokes lineare H1} is finite and we deduce that $u \in L^\infty((t_0, +\infty); \dot H^1 (\Hg^d))\cap L^2((t_0,+\infty);\dot{H}^2(\Hg^d))$.
\section{Derivation of a suitable approximate system} 
\label{sec:Derivation of a suitable approximate system}
\subsection{Homogeneous Friedrichs multipliers on $\Hg^d$}
In this section we define an approximate version of System \eqref{NSH} by performing a Friedrichs-type method in our context. 

In the context of the Navier-Stokes system on the Euclidean setting, the Friedrichs method aims to construct approximate systems of the original system by truncation in the spectrum of the Stokes operators or using the Fourier transform in the case of the torus or the whole space, which provides a more flexible procedure in the last two cases. Due to the hypoellipticity of $\Laplace_{\Hg}$, the strategy that consists of using a spectral decomposition associated with the Stokes operators appears to be trickier than in the Euclidean setting. We choose to use a Fourier analysis approach based on the Fourier transform on the Heisenberg group $\Hg^d$. This approach provides a highly flexible and unified framework to manipulate all the regularity and differential operators that appear in our strategy.

In order to work on the $\Tilde{H}^d(\Hg^d)$ framework for System \eqref{NSH}, our analysis must involve two different regularities, with respect to the space variable $(Y,s)$, that are: the required regularity in order to give a ``strong sense'' of all terms of System \eqref{NSH}, which is $\dot{H}^\ell(\Hg^d)$ with $\ell\geq 2$, and the required regularity in order to work on the Sobolev-type scale $\Tilde{H}^\ell(\Hg^d)$ with $\ell\in \R$. The most conceptual obstruction of this idea is that these two regularities (the regularity given by $\Laplace_{\Hg}$ and  $\Tilde{\Laplace}_{\Hg}$) do not coincide in general (see Remark \ref{Remarque sur les espaces de Sobolev gauche et droite}). In order to overcome this difficulty, we construct spectral multipliers that regularize simultaneously with respect to the $H^\ell(\Hg^d)$ and $\Tilde{H}^\ell(\Hg^d)$ regularity and that commutes with the two sub-Laplacians $\Laplace_{\Hg}$ and $\Tilde{\Laplace}_{\Hg}$. This is possible by taking advantage of the Fourier transform on the Heisenberg group $\Hg^d$ as follows. For any $k\in \N$, we define the \textit{bi-stratified Friedrichs multiplier $\Jk$} by setting for all smooth enough complex value function $f$ on $\Hg^d$
$$
\mathcal{F}_{\Hg}(\Jk f)(n,m,\lambda):=\mathbf{1}_{\{\frac{1}{2^{k+1}}\leq 4|\lambda|(2|m|+d)\leq 2^k\}}\mathbf{1}_{\{\frac{1}{2^{k+1}}\leq 4|\lambda|(2|n|+d)\leq 2^k\}}\mathcal{F}_{\Hg}(f)(n,m,\lambda),
$$
with $(n,m,\lambda)\in\Tilde{\Hg}^d$. We also define the Friedrichs multiplier $\TJk$ associated with the right-invariant Laplacian $\Tilde{\Laplace}_{\Hg}$ by setting 
$$
\mathcal{F}_{\Hg}(\TJk f)(n,m,\lambda):=\mathbf{1}_{\{\frac{1}{2^{k+1}}\leq 4|\lambda|(2|n|+d)\leq 2^k\}}\mathcal{F}_{\Hg}(f)(n,m,\lambda).
$$
\begin{rem}\label{IV 4/01/2023}
Let us point out that we cannot expect the $\Tilde{H}^\ell$ regularity if we localize in Fourier space only by using $\mathbf{1}_{\{\frac{1}{2^{k+1}}\leq 4|\lambda|(2|m|+d)\leq 2^{k}\}}$ since the $\dot{H}^\ell$ and the $\Tilde{H}^{\ell'}$ regularities are not comparable in general (see \cite[Example 4.4.32, p. 250]{QuantizationonNilpotentLieGroups}) as explained in Remark \ref{Remarque sur les espaces de Sobolev gauche et droite}. This fact can be seen by using the Fourier transform on the Heisenberg group according to Proposition \ref{Fourier diagonalise the sublaplacian}: among the parameters $(n,m,\lambda)\in\N^d\times\N^d\times\R^{*}$, the action of a left-invariant operator involves only $(m,\lambda)$ while the action of a right-invariant operator involves only $(n,\lambda)$.
\end{rem}
\begin{notation}
If $E$ and $F$ are two topological vector spaces, we denote by $\mathcal{L}(E,F)$ the space of continuous linear maps from $E$ to $F$ and we denote $\mathcal{L}(E):=\mathcal{L}(E,E)$. 
\end{notation}
The following proposition summarizes the useful properties of the bi-stratified Friedrichs multipliers $\Jk$.
\begin{prop}
\label{Prop- I 25/11/2023}
Let $\ell$ and $\ell'$ two real numbers and $k$ a nonnegative integer. Then 
\begin{enumerate}[topsep=0pt,parsep=0pt,leftmargin=*]
    \item \label{I 04/01/2023} $\Jk$ belongs to $\mathcal{L}(\dot{H}^\ell,\dot{H}^{\ell'})\cap\mathcal{L}(\Tilde{H}^\ell,\Tilde{H}^{\ell'})$, and we have
$$
    \|\Jk\|_{\mathcal{L}(\dot{H}^\ell,\dot{H}^{\ell'})}\leq 2^{\frac{(k+1)|\ell-\ell'|}{2}}\ \ \ \ \text{and}\ \ \ \ \|\Jk\|_{\mathcal{L}(\Tilde{H}^{\ell},\Tilde{H}^{\ell'})}\leq 2^{\frac{(k+1)|\ell-\ell'|}{2}},
$$
\item $\Jk$ belongs to $\mathcal{L}(\dot{H}^\ell,\Tilde{H}^{\ell'})\cap\mathcal{L}(\Tilde{H}^\ell,\dot{H}^{\ell'})$, and we have
$$
\|\Jk\|_{\mathcal{L}(\dot{H}^\ell,\Tilde{H}^{\ell'})}\leq 2^{\frac{(k+1)(|\ell|+|\ell'|)}{2}}\ \ \ \ \text{and}\ \ \ \ \|\Jk\|_{\mathcal{L}(\Tilde{H}^\ell,\dot{H}^{\ell'})}\leq 2^{\frac{(k+1)(|\ell|+|\ell'|)}{2}},
$$
\item \label{III 4/01/2023} $\left[\Jk,(-\Laplace_{\Hg})^{\ell}\right]=\left[\Jk,(\Id-\Laplace_{\Hg})^\ell\right]=\left[\Jk,(-\Tilde{\Laplace}_{\Hg})^{\ell}\right]=\left[\Jk,(\Id-\Tilde{\Laplace}_{\Hg})^\ell\right]=0,$
\item \label{Commutation Jk et |Ds|}for any $\zeta\in\R$, $\left[\Jk,e^{\zeta|D_s|}\right]=\left[\Jk,|D_s|^\ell\right]=\left[\Jk,\partial_s\right]=0,$
\item $\Jk$ is a bounded self-adjoint operator on $\dot{H}^\ell(\Hg^d)$ and on $\Tilde{H}^\ell(\Hg^d)$,
\item if $\ell\in\R$, then for any $f\in \Tilde{H}^\ell(\Hg^d)$ (respectively $f\in\dot{H}^\ell(\Hg^d)$), the sequence $(\Jk f)$ converges to $f$ in $\Tilde{H}^\ell(\Hg^d)$ (respectively in $\dot{H}^\ell(\Hg^d)$),
\item $\Jk^2=\Jk$.
\end{enumerate}
\end{prop}

We emphasize that, in view of the first point, we can control the norm of $\Jk$ in $\mathcal{L}(\dot{H}^\ell)$ or in $\mathcal{L}(\Tilde{H}^\ell)$, for any $\ell\in\R$ independently of $k\in\N$, that is $\|\Jk\|_{\mathcal{L}(H^\ell)}\leq 1$ and $\|\Jk\|_{\mathcal{L}(\Tilde{H}^\ell)}\leq 1$. 

\begin{proof}
We give only the proof of the first two points. The proofs of the other points are left to the reader, and follow from Lemma \ref{Fourier diagonalise the sublaplacian}, the Plancherel formula on $\Hg^d$ and the definition of $\Jk$. For the first and the second points, we prove only the first inequality because the proof of the second ones are similar.

Let $f$ be in $\dot{H}^\ell(\Hg^d)$. According to the Plancherel formula on $\Hg^d$, we have
\begin{align}
   &\|\Jk f\|_{\dot{H}^\ell}^{2}
  =\frac{2^{d-1}}{\pi^{d+1}}\!\sum_{(n,m)\in\N^d\times\N^d}\int_{\R}(4|\lambda|(2|m|+d))^{\ell}\mathbf{1}_{\{\frac{1}{2^{k+1}}\leq 4|\lambda|(2|m|+d)\leq 2^k\}}\nonumber\\
  & 
  \hspace{5.8cm}
  \times\mathbf{1}_{\{\frac{1}{2^{k+1}}\leq 4|\lambda|(2|n|+d)\leq 2^k\}}|\mathcal{F}_{\Hg}(f)(n,m,\lambda)|^2|\lambda|^dd\lambda\label{I 16/02/2024}.
\end{align}
\textit{1.} For any $(n,m,\lambda)\in\N^d\times\N^d\times\R^{*}$, we have
$$
(4|\lambda|(2|m|+d))^{\ell'-\ell}\mathbf{1}_{\{\frac{1}{2^{k+1}}\leq 4|\lambda|(2|m|+d)\leq 2^k\}}\leq (4|\lambda|(2|m|+d))^{\ell'}2^{(k+1)|\ell-\ell'|}.
$$
Thus, by bounding the indicator function of the set $\{2^{-(k+1)}\leq 4|\lambda|(2|n|+d)\leq 2^k\}$ by $1$, we deduce from \eqref{I 16/02/2024} and the Plancherel formula that
$$
\|\Jk f\|^{2}_{\dot{H}^\ell}\leq 2^{(k+1)|\ell-\ell'|}\|f\|^{2}_{\dot{H}^{\ell'}}.
$$
\textit{2.} Let $(n,m,\lambda)\in\N^d\times\N^d\times\R^{*}$. Separating the cases $\ell\geq 0$ and $\ell< 0$, we get
\begin{equation}
\label{II 16/02/2024}
(4|\lambda|(2|m|+d))^{\ell}\mathbf{1}_{\{\frac{1}{2^{k+1}}\leq 4|\lambda|(2|m|+d)\leq 2^k\}}\leq 2^{(k+1)|\ell|}.
\end{equation}
Besides, also by reasoning on the sign of $\ell'$, we obtain
\begin{equation}
\label{III 16/02/2024}
\mathbf{1}_{\{\frac{1}{2^{k+1}}\leq 4|\lambda|(2|n|+d)\leq 2^k\}}\leq 2^{(k+1)|\ell'|}(4|\lambda|(2|n|+d))^{\ell'}.
\end{equation}
Then, thanks to \eqref{II 16/02/2024} and \eqref{III 16/02/2024}, it follows from \eqref{I 16/02/2024} and the Plancherel formula on $\Hg^d$ that
\begin{align*}
    \|\Jk f\|_{\dot{H}^\ell}^{2}&\leq 2^{(k+1)(|\ell|+|\ell'|)}\frac{2^{d-1}}{\pi^{d+1}}\sum_{(n,m)\in\N^d\times\N^d}\int_{\R}(4|\lambda|(2|n|+d))^{\ell'}|\mathcal{F}_{\Hg}(f)(n,m,\lambda)|^2|\lambda|^dd\lambda\\
&=2^{(k+1)(|\ell|+|\ell'|)}\|f\|_{\Tilde{H}^{\ell'}}^{2}.
\end{align*}
\end{proof}
The operators $\Jk$ regularize with respect to the regularities generated by the left-invariant and the right-invariant fields (see the first two points of Proposition \ref{Prop- I 25/11/2023}). Unfortunately, they do not commute with $\nabla_{\Hg}$ and then we cannot propagate $\Jk$ in \eqref{NSH}. However, the Fourier multiplier $\TJk$ commutes with $\nabla_{\Hg}$ and can thus be propagated in \eqref{NSH}. We now give the properties of $\TJk$ in the following proposition.
\begin{prop}
 \label{Prop-Tilde Jk}  
Let $k$ be a nonnegative integer. Then 
\begin{enumerate}[topsep=0pt,parsep=0pt,leftmargin=*]
    \item \label{continuite Tilde Jk} for any real numbers $\ell$ and $\ell'$, the operator $\TJk$ belongs to $\mathcal{L}(\Tilde{H}^\ell,\Tilde{H}^{\ell'})$, and we have
$$
\|\TJk\|_{\mathcal{L}(\Tilde{H}^\ell,\Tilde{H}^{\ell'})}\leq 2^{\frac{(k+1)|\ell-\ell'|}{2}},
$$
\item \label{commutation Tilde Jk souslaplacien} for any $\ell\in\R$, $\left[\TJk,(-\Laplace_{\Hg})^{\ell}\right]=\left[\TJk,(\Id-\Laplace_{\Hg})^\ell\right]=\left[\TJk,(-\Tilde{\Laplace}_{\Hg})^{\ell}\right]=\left[\TJk,(\Id-\Tilde{\Laplace}_{\Hg})^\ell\right]=0,$
\item \label{commutation Tilde Jk champs invariant a droite} for any $\ell\in\N$ and $\alpha\in[\![1,2d]\!]^\ell$, we have $\left[\TJk,P^\alpha\right]=0$, in particular we have $\left[\TJk,\nabla_\Hg\right]=0$, 
\item \label{Commutation Tilde Jk Ds}for any $\zeta\in\R$, we have $\left[\TJk,e^{\zeta|D_s|}\right]=\left[\TJk,|D_s|^\ell\right]=\left[\TJk,\partial_s\right]=0$,
\item for any $\ell\in\R$, the operator $\TJk$ is a bounded self-adjoint operator on $\Tilde{H}^\ell(\Hg^d)$ and $\dot{H}^\ell(\Hg^d)$,
\item for any $\ell\in\R$, if $f\in \Tilde{H}^\ell(\Hg^d)$ (respectively $f\in \dot{H}^\ell(\Hg^d)$), the sequence $(\TJk f)$ converges to $f$ in $\Tilde{H}^\ell(\Hg^d)$ (respectively in $ \dot{H}^\ell(\Hg^d)$),
\item  $\TJk^2=\TJk$ and $\Jk\TJk=\TJk\Jk=\Jk$.
\end{enumerate} 
\end{prop}
The proof is similar to the proof of Proposition \ref{Prop- I 25/11/2023} and is left to the reader.

\begin{rem}
\label{Remark- Regularistaion par troncature en Fourier}
Let $k\in\N$ and $f\in H^{\ell}(\Hg^d)\cup\Tilde{H}^\ell(\Hg^d)$ with $\ell\in\R$. According to Proposition \ref{Prop- I 25/11/2023}, Items 1 and 2, and the local Sobolev embedding $H^{2j}(\Hg^d)\hookrightarrow H^{j}_{loc}(\R^{2d+1})$ with $j\in \N$ (see for instance \cite{SubellipticestimatesandfunctionspacesonnilpotentLiegroups}, Theorem 4.16), we deduce that $\Jk f$ belongs to $ \mathcal{C}^{\infty}(\R^{2d+1})$. Similarly, if $f\in \Tilde{H}^{\ell}(\Hg^d)$ with $\ell\in\R$, then $\TJk f$ belongs to $ \mathcal{C}^{\infty}(\R^{2d+1})$.
\end{rem}
\subsection{The Leray projector on $\Hg^d$}
We now give some properties concerning the Leray projector on $\Hg^d$, given in \eqref{def: Leray projector}, that we recall for convenience
\begin{equation*}
\mbb{P}:=\Id+\nabla_{\Hg}\circ(-\Laplace_{\Hg})^{-1}\circ\divergence_{\Hg}.
\end{equation*}
\begin{prop} The Leray projector $\mbb{P}$ satisfies the following properties:
\label{Prop- Leray projector}
\begin{enumerate}[topsep=0pt,parsep=0pt,leftmargin=*]
    \item $\divergence_{\Hg}\circ\mbb{P}=0$,
    \item if $u\in \mc{S}'(\Hg^d)^{2d}$ is a horizontal vector field such that $\divergence_{\Hg}(u)=0$, then $\mbb{P}u=u$,
    \item for any $p\in(1,+\infty)$, we have $\mbb{P}\in\mc{L}(L^p(\Hg^d))$,
    \item for any $\ell\in\R$, we have $[\mbb{P},(-\Tilde{\Laplace}_{\Hg})^{\frac{\ell}{2}}]=[\mbb{P},(\Id-\Tilde{\Laplace}_{\Hg})^{\frac{\ell}{2}}]=0$,
    \item for any $\ell\in\N$ and $\alpha\in[\![1,2d]\!]^\ell$, we have $[\mbb{P},\Tilde{P}^{\alpha}]=0$,
    \item for any positive real number $\zeta$ and $\ell$, we have $[\mbb{P},e^{\zeta|D_s|}]=[\mbb{P},|D_s|^\ell]=0$,
    \item for any $k\in\N$, we have $[\mbb{P},\TJk]=0$,
    \item $\mbb{P}$ is self-adjoint on $\Tilde{H}^\ell(\Hg^d)$ for any $\ell\in\R$,
    \item for any $\ell\in\R$, the operator $\mbb{P}$ is bounded on $\dot{H}^{\ell}(\Hg^d)$ and on $\Tilde{H}^\ell(\Hg^d)$.
\end{enumerate}
\end{prop}
\begin{proof}
The first two points follow from direct computation. To prove Item 3, let us begin by remarking that $\Id-\mbb{P}$ is a $2d\times 2d$-matrix operators whose components are of the form $-P_i\circ(-\Laplace_{\Hg})^{-1}\circ P_j$. Then the continuity follows from the continuity of the Riesz transforms $P_i\circ(-\Laplace_{\Hg})^{-\frac 12}$ and $(-\Laplace_{\Hg})^{-\frac 12}\circ P_j$ on $L^p(\Hg^d)$ (see for instance \cite{RieszTransformsonGroupsofHeisenbergType}). Items $4$, $5$, $6$ and $7$ follow from Propositions \ref{Fourier diagonalise the sublaplacian} and \ref{Prop- Forier X and Y}. Item 8 is a consequence of Items $3$ (with $p=2$) and $4$. Finally, the last item follows from Proposition \ref{Continuite des champs horizonteaux} and the continuity of $(-\Laplace_{\Hg})^{-1}$ from $\dot{H}^{\ell+1}(\Hg^d)$ to $\dot{H}^{\ell-1}(\Hg^d)$.
\end{proof}
\begin{rem}
\label{rem:remark sur les interaction entre Jk et Pj et P}
Let us note the $[\mbb{P},\Jk]\neq 0$ and $[\mathbb{P},P_j]\neq0$ for any $j\in[\![1,2d]\!]$. This contrasts with Items $5$ and $7$ of Proposition \ref{Prop- Leray projector}.
\end{rem}
We introduce the following crucial identity.
\begin{lem}
    \label{I 28/08/2023}
Let $v$ be a smooth enough horizontal vector field. If $\divergence_{\Hg}(v)=0$, then we have
\begin{equation}
\label{I 27/08/2023}
(\Id-\mathbb{P})\circ(-\Laplace_{\Hg})v=\Pi_{\Hg}\circ\partial_s v,
\end{equation}
where
$$
\Pi_{\Hg}:=4(\Id-\mathbb{P})\circ\mathfrak{S},
$$
and $\mathfrak{S}$ is the matrix defined in \eqref{matrice mathfrak S}. Moreover, $\Pi_\Hg$ is an operator of order $0$ with respect to both left-invariant and right-invariant sub-Laplacians (see Definition \ref{def: Ordre operateur}).
\end{lem}
\begin{proof}[Proof of Lemma \ref{I 28/08/2023}]
Let us begin by remarking that for any smooth enough horizontal vector fields $v$, we have
\begin{align*}
\divergence_{\Hg}(\Laplace_{\Hg}v)&=\sum_{1\leq i,j\leq 2d}P_iP_{j}^2v_i=2\sum_{1\leq i\neq j\leq 2d}[P_i,P_j]P_j v_i+\Laplace_{\Hg}\divergence_{\Hg}(v).
\end{align*}
However, in view of \eqref{24/08/2022 14h 2} and since $\partial_s$ commute with $X_i$ and $\Xi_i$ for any $i\in[\![1,d]\!]$ (see \eqref{commutateur XkYj S}), we have
\begin{align*}
\sum_{1\leq i\neq j\leq 2d}[P_i,P_j]P_j v_i&=\sum_{1\leq i\leq d}[X_i,\Xi_i]\Xi_iv_i+\sum_{1\leq j\leq d}[\Xi_j,X_j]X_jv_{j+d}
\\
&=-4\sum_{1\leq i\leq d}\partial_s(\Xi_i v_i-X_i v_{i+d})\\
&= -2  \left(\sum_{1\leq i\leq d} 2(\Xi_i \partial_s v_i-X_i \partial_s v_{i+d}) \right)\\
&= -2 \divergence_{\Hg}(\mathfrak{S}\partial_sv).
\end{align*}
Thus, for $v$ such that $\divergence_{\Hg}(v)=0$, 
$$
	\divergence_{\Hg}(-\Laplace_{\Hg}v)=4\divergence_{\Hg}(\mathfrak{S} \partial_s v).
$$
Then, \eqref{I 27/08/2023} follows from the definition of $\mathbb{P}$. Because $(\Id-\mathbb{P})$ is an operator of order $0$ (see Proposition \ref{Prop- Leray projector}, Items $8$ and $9$), the operator $\Pi_\Hg$ is also of order $0$.
\end{proof}
This Lemma will be crucial to prove the convergence of the solution of approximate systems and to propagate the regularity with respect to the left-invariant vector fields in Lemma \ref{Lemme 2 08/08/2023}.

\subsection{Stokes system}
In this subsection, using the properties of the Leray projector on $\Hg^d$ and the operators $\Jk$ and $\TJk$, we investigate the well-posedness of the linear Stokes system in the following lemma.
\begin{lem}
\label{Stokes linaire}
Let $f$ be in $L^2(\R_{+};\dot{H}^{-1}(\Hg^d))$ and $u_0\in L^2(\Hg^d)$ such that $\divergence_{\Hg}(u_0)=0$. Then, there exists a unique solution $u\in\mathcal{C}_{b}(\R_{+},L^{2}(\Hg^d))\cap L^2(\R_{+},\dot{H}^{1}(\Hg^d))$ of the following initial value problem
\begin{equation}
    \label{NSHL}
    \begin{cases}
\partial_tu-\mathbb{P}\Laplace_{\Hg}u=\mathbb{P}f \ &\text{in}\ \ \R_{+}\times\Hg^d,\\
\divergence_{\Hg}(u)=0 \ &\text{in}\ \R_{+}\times\Hg^d,\\
u_{|_{t=0}}=u_0\ &\text{in}\ \ \Hg^d.
\end{cases}
\end{equation}
Moreover, this solution satisfies, for any $T>0$,
\begin{equation}
    \label{Energy linear}
    \|u\|_{L^{\infty}_{T}(L^2)}^{2}+2\|\nabla_{\Hg}u\|_{L^{2}_{T}(L^2)}^{2}\leq \|u_0\|_{L^2}^{2}+2\langle f,u\rangle_{L^{2}_{T}(L^2)}.
\end{equation}
\end{lem}
\begin{proof}
Let us take $k\in\N$. We consider the following regularized system
\begin{equation}
    \label{NSHLk}
    \begin{cases}
\partial_tu_k-\mathbb{P}\Jk\Laplace_{\Hg}u_k=\mathbb{P}\Jk(f\star_t\eta_k),\\
\divergence_{\Hg}(u_k)=0,\\
u_k{|_{t=0}}=u_0,
\end{cases}
\end{equation}
where, for any $k\in\N$, the function $\eta_{k}$ is defined as follows: we choose $\eta\in\mathcal{C}_{c}^{\infty}(\R_{-})$ such that $\int_{\R}\eta=1$ and we set $\eta_k:=\frac{1}{k+1}\eta(\frac{\cdot}{k+1})$ and where we have extended $f$ by $0$ on $(-\infty,0)$. In view of the properties of $\mbb{P}$ (see Proposition \ref{Prop- Leray projector}, Items 1 and 2) and $\Jk$ (see Proposition \ref{Prop- I 25/11/2023}, Item 1), and the Young inequality in order to ensure that $f\star_t\eta_k$ belongs to $L^{\infty}(\dot{H}^{-1})$, we deduce that \eqref{NSHLk} is the Cauchy problem for an ordinary differential equation on the Banach space 
$$
L^{2}_{\divergence_{\Hg}}:=\enstq{u\in L^2(\Hg^d)}{\divergence_{\Hg}(u)=0},
$$ 
equipped with the $L^2(\Hg^d)$ topology. Then, for any $k\in\N$, there is $T_k>0$ and a unique solution $u_k:[0,T_k)\times\Hg^d\rightarrow\R^{2d}$ of \eqref{NSHLk} which belongs to $\mathcal{C}^1([0,T_k), L^2(\Hg^d))$. Moreover, we have 
$$
\frac{1}{2}\frac{d}{dt}\|u_k\|_{L^2}^{2}-\langle \mathbb{P}\Laplace_{\Hg}\Jk u_k,u_k\rangle_{L^2}=\langle\mathbb{P}\Jk (f\star_t\eta_k),u_k\rangle_{L^2}.
$$
Since $\divergence_{\Hg}(u)=0$, we have $\mathbb{P}u_k=u_k$ and by using that $\mathbb{P}$ is a self-adjoint operator on $L^2(\Hg^d)$ (see Proposition \ref{Prop- Leray projector}, Item 8) and the properties of $\Jk$ (see Proposition \ref{Prop- I 25/11/2023}, Items $4$, $5$ and $7$), we get
$$
-\langle\mathbb{P}\Laplace_{\Hg}\Jk u_k,u_k\rangle_{L^2}=-\langle\Laplace_{\Hg}\Jk u_k,\Jk u_k\rangle_{L^2}=\|\nabla_{\Hg}\Jk u_k\|_{L^2}^{2},
$$
and
$$
\langle\mathbb{P}\Jk (f\star_t\eta_k),u_k\rangle_{L^2}=\langle f\star_t\eta_k, \Jk u_k\rangle_{L^2}\leq \frac{1}{2}\|f\star_t\eta_k\|_{\dot{H}^{-1}}^{2}+\frac{1}{2}\|\nabla_{\Hg}\Jk u_k\|_{L^2}^{2}.
$$
Accordingly, we deduce that for any $t\in(0,T_k)$, we have
\begin{equation}
    \label{borne sur NSHLk}
\|u_k(t)\|_{L^2}^{2}+\int_{0}^{t}\|\nabla_{\Hg}\Jk u_k(\tau)\|_{L^2}^{2}d\tau\leq \|u_0\|_{L^2}^{2}+\int_{0}^{t}\|(f\star_t\eta_k)(\tau)\|_{\dot{H}^{-1}}^{2}d\tau.
\end{equation}
Then $u_k\in L^{\infty}((0,T_k),L^2(\Hg^d))$ and according to the properties of $\Jk$, we have for any $\tau\in (0,T_k)$
\begin{align*}
\|\mathbb{P}\Laplace_{\Hg}\Jk u(\tau)+\mathbb{P}&\Jk(f\star_t\eta_k)(\tau)\|_{L^2}\\
&\leq \|\mathbb{P}\Laplace_{\Hg}\Jk\|_{\mathcal{L}(L^2)}\|u_k\|_{L^{\infty}((0,T_k);L^2)}+\|\mathbb{P}\Jk\|_{\mathcal{L}(\dot{H}^{-1},L^2)}\|f\star_t\eta_k\|_{L^{\infty}(\dot{H}^{-1})}.
\end{align*}
Hence, according to the blow-up criteria for ordinary differential equations, if $(0,T_k)$ is the maximal existence interval, this implies that $T_k=+\infty$. Since, for any $k\in\N$
$$
\int_{0}^{t}\|(f\star_t\eta_k)(\tau)\|_{\dot{H}^{-1}}^{2}d\tau\leq\|\eta_k\|_{L^1(\R)}^{2}\int_{0}^{t}\|f(\tau)\|_{\dot{H}^{-1}}^{2}d\tau=\int_{0}^{t}\|f(\tau)\|_{\dot{H}^{-1}}^{2}d\tau,
$$
then the right-hand side of \eqref{borne sur NSHLk} is bounded independently on $k$. Thus, we deduce that, up to extract a subsequence, $(u_k)$ converges weakly-$\star$ in $L^{\infty}(L^2)$ to a function $u$ and $(\Jk u_k)$ converges weakly in $L^2(\dot{H}^1)$ to a function $v$. For any $k\in \N$, since $u_k$ and $\Jk u_k$ belong to $L^{\infty}(L^2)$, using that $\Jk$ is self-adjoint in $L^2(\Hg^d)$ (see Proposition \ref{Prop- I 25/11/2023}, Item 5), for any $\vphi\in\mathcal{D}(\R_+\times\Hg^d)$,
$$
\langle\Jk u_k,\vphi\rangle_{\mc{D}',\mc{D}}=\int_{\R_+}\langle u_k(t),\vphi(t)\rangle_{L^2} dt+\int_{\R_+}\langle u_k(t),(\Jk-\Id)\vphi(t)\rangle_{L^2} dt.
$$
Because $(u_k)$ converges weakly-$\star$ to $u$ in $L^{\infty}(L^2)$ and $((\Jk-\Id)\vphi)$ converges strongly in $L^1(L^2)$ to $0$ (using Proposition \ref{Prop- I 25/11/2023}, which ensure that $\|\Jk-\Id\|_{\mc{L}(L^2)}\leq 2$, the convergence of $((\Jk-\Id)\vphi(t))$ in $L^2(\Hg^d)$ for any $t\in \R_{+}$, and the dominated convergence theorem), we deduce that $(\Jk u_k)$ converges to $u$ in $\mathcal{D}'(\R_+\times\Hg^d)$. This ensures that $u=v$. By the same way, since $(f\star_t\eta_k)$ converges to $f$ in $L^2(\dot{H}^{-1})$ and using that $\mbb{P}$ is self-adjoint in $L^2(\Hg^d)$, we deduce that for any $\vphi\in\mc{D}(\R_+\times\Hg^d)$,
$$
\lim_{k\rightarrow+\infty}\langle\mbb{P}\Jk(f\star_t\eta_k),\vphi\rangle_{\mc{D}',\mc{D}}=\langle f,\mbb{P}\vphi\rangle_{L^2(L^2)}=\langle \mbb{P}f,\vphi\rangle_{L^2(L^2)}.
$$
This ensures that $u$ is a solution of the Cauchy problem \eqref{NSHL} in $\mathcal{D}'((0,+\infty)\times\Hg^d)$. Moreover, in view of the properties of the weak-$\star$ and weak convergences, we deduce the estimate \eqref{Energy linear} holds for $u$ by passing to the limit in \eqref{borne sur NSHLk}. Finally, the continuity in time follows by interpolation since $\partial_tu$ belongs to $L^2(\dot{H}^{-1})$ according to the first line of \eqref{NSHL}. To prove the uniqueness, let us pick two solutions $u^1$ and $u^2$ of \eqref{NSHL} that belong to $\mathcal{C}_b(L^2)\cap L^2(\dot{H}^1)$. Then, we have
$$
\|u^1-u^2\|_{L^{\infty}(L^2)}^{2}+2\|\nabla_{\Hg}(u^1-u^2)\|_{L^2(L^2)}^{2}\leq 0,
$$
that is $u^1=u^2$.
\end{proof}
\subsection{Construction of the approximate problem}
We are next interested in the well-posedness of the following Cauchy problem in $\mathcal{C}_b(L^2)\cap L^2(\dot{H}^1)$
\begin{equation}
\label{NSHk}
\begin{cases}
\partial_tu_k-\mathbb{P}\Laplace_{\Hg}u_k+\mathbb{P}\Jk (u_k\cdot\nabla_{\Hg}\Jk u_k)=0\  &\text{in}\ \ \R_+\times\Hg^d,\\
\divergence_{\Hg}(u_k)=0\  &\text{in}\ \ \R_+\times\Hg^d,\\
\Tilde{\mathsf{J}}_k u_k=u_k\  &\text{in}\ \ \R_+\times\Hg^d,\\
\end{cases}
\end{equation}
and 
\begin{equation}
    \label{NSHk initial condition}
    u_{k|_{t=0}}=\Tilde{\mathsf{J}}_k u_0\ \text{in}\ \ \Hg^d,
\end{equation}
where $u_0$ is a horizontal vector field belonging to $L^2(\Hg)$ or $\Tilde{H}^d(\Hg)$ and $k\in\N$. This is the object of the following lemma.
\begin{lem}
\label{II 03/12/2024}
Let $k\in\N$. Let $u_0$ be a horizontal vector field which belongs to $L^2(\Hg^d)\cup\Tilde{H}^d(\Hg^d)$ and satisfies $\divergence_{\Hg}(u_0)=0$. Then there exists a unique solution $u_k$ of \eqref{NSHk}-\eqref{NSHk initial condition}, which belongs to $\mathcal{C}_b(\R_{+},L^2(\Hg^d))\cap L^2(\R_{+},\dot{H}^1(\Hg^d))$ and satisfies
\begin{equation}
\label{energy L2 pour tout le monde}
\|u_k\|_{L^{\infty}(L^2)}^{2}+2\|\nabla_{\Hg}u_k\|_{L^2(L^2)}^{2}\leq \|\TJk u_0\|_{L^2}^{2}.
\end{equation}
Moreover, we have $u_k\in \mathcal{C}_b(\Tilde{H}^\ell)$ and $\nabla_\Hg u_k\in L^2(\Tilde{H}^\ell)$ for any $\ell\in\R$. In particular, $u_k\in\mathcal{C}^{\infty}_{b}(\Hg^d)$.
\end{lem}
The proof is an adaptation of the classical strategy: first we show local existence and second we show that all solutions are global by establishing a blow-up criteria. However, some difficulties appear due to the structure of the system and the fact that we consider two different regularities.
\begin{proof} In this proof, since $k$ is fixed, we dropped the subscript $k$ for $u$ to simplify notations.

For any $T\in(0,+\infty)$, we define the space
   $$
   E_{k}^{T}:=\enstq{u\in\mathcal{C}_b([0,T],L^2(\Hg^d))\cap L^2(0,T;\dot{H}^1(\Hg^d))}{\divergence_{\Hg}(u)=0\ \ \text{and}\ \ \Tilde{\mathsf{J}}_ku=u},
   $$
   which is a closed sub-space of the Banach spaces $\mathcal{C}_b([0,T],L^2(\Hg^d))\cap L^2(0,T;\dot{H}^1(\Hg^d))$. Let us consider $T>0$ which will be chosen later. Let us define the map $\Phi_{u_0}$, which from $v\in E_{k}^{T}$ gives the unique solution $u$ of the following Cauchy problem 
   \begin{equation}
\label{NSHkv}
\begin{cases}
\partial_tu-\mathbb{P}\Laplace_{\Hg}u+\mathbb{P}\Jk (v\cdot\nabla_{\Hg}\Jk v)=0\  &\text{in}\ \ (0,T)\times\Hg^d,\\
\divergence_{\Hg}(u)=0,\  &\text{in}\ \ (0,T)\times\Hg^d,
\end{cases}
\end{equation}
and 
\begin{equation}
    \label{NSHkv initial condition}
    u_{|_{t=0}}=\Tilde{\mathsf{J}}_k u_0\  \text{in}\ \ \Hg^d.
\end{equation}
This map is well-defined according to Lemma \ref{Stokes linaire}, since $\Jk(v\cdot\nabla\Jk v)$ belongs to $L^2(\dot{H}^{-1})$ when $v\in E_{k}^{T}$.
\paragraph{Local existence} \textit{1) $E_{k}^{T}$ is stable by $\Phi_{u_0}$.} Let $v$ be in $E_{k}^{T}$, and set $u:=\Phi_{u_0}(v)$. We have $u\in\mathcal{C}_b(L^2)\cap L^2(\dot{H}^1)$ according to Lemma \ref{Stokes linaire} and $\divergence_{\Hg}(u)=0$. Then, in order to show that $u$ belongs to $E_{k}^{T}$, it is enough to show that
\begin{equation}
    \label{propafation of Tilde Jk}
    \TJk u=u.
\end{equation}
Since $[\TJk,\mbb{P}]=[\TJk,\Laplace_{\Hg}]=0$ and $\TJk\Jk=\Jk$ (see Proposition \ref{Prop-Tilde Jk} Items 2 and 7 and Proposition \ref{Prop- Leray projector} Item 7), we deduce that $\TJk u$ satisfies the first line of \eqref{NSHkv}. Using that $[\divergence_{\Hg},\TJk]=0$ (see Proposition \ref{Prop-Tilde Jk}, Item 3), we obtain that $\divergence_{\Hg}(\TJk u)=0$ on $[0,T)$. Finally, we have $\TJk u_{|_{t=0}}=\TJk^2u_0=\TJk u_0$. Thus, $\TJk u$ is also a solution of \eqref{NSHkv}-\eqref{NSHkv initial condition}. Since $\TJk$ belongs to $\mathcal{L}(L^2)\cap\mathcal{L}(\dot{H}^1)$, we have also $\Tilde{u}\in \mathcal{C}_b(L^2)\cap L^2(\dot{H}^1)$. Then \eqref{propafation of Tilde Jk} follows from the uniqueness provided by Lemma \ref{Stokes linaire} and then $u$ belongs to $E_{k}^{T}$. We conclude that $\Phi_{u_0}$ maps $E_{k}^{T}$ to itself.

\textit{2) For $T>0$ small enough, $\Phi_{u_0}$ is a strict contraction on a subset of $E^{T}_{k}$.} Let $v$ be in $E_{k}^{T}$. We have
$$
\langle \mbb{P}\Jk (v\cdot\nabla_\Hg\Jk v), v\rangle_{L^{2}_{T}(L^2)}\leq \frac{1}{2}\|\mbb{P}\Jk (v\cdot\nabla_\Hg\Jk v)\|_{L^{2}_{T}(\dot{H}^{-1})}^{2}+\frac 12 \|\nabla_{\Hg}v\|_{L^{2}_{T}(L^2)}^{2},
$$
and then, in view of \eqref{Energy linear}, it follows that
\begin{multline}
\|\Phi_{u_0}(v)\|_{E_{k}^{T}}^{2}:=\|\Phi_{u_0}(v)\|_{L^{\infty}_{T}(L^2)}^{2}+\|\nabla_{\Hg}\Phi_{u_0}(v)\|_{L^{2}_{T}(L^2)}^{2}
\\
\leq \|\TJk u_0\|_{L^2}^{2}+ \|\mbb{P}\Jk (v\cdot\nabla_\Hg\Jk v)\|_{L^{2}_{T}(\dot{H}^{-1})}^{2}.\label{I 03/01/2024}
\end{multline}
Furthermore, according to Proposition \ref{Prop- I 25/11/2023}, Item 1 and Proposition \ref{proprietes des Sobolev H}, we deduce that there exists a constant $C_k$, which does not depend on $T$ such that
$$
\|\mbb{P}\Jk(v\cdot\nabla_\Hg\Jk v)\|_{L^{2}_{T}(\dot{H}^{-1})}^{2}\leq C_k T\|v\|_{L^{\infty}_{T}(L^2)}^{2}.
$$
Then combining this inequality with \eqref{I 03/01/2024}, we deduce that
$$
\|\Phi_{u_0}(v)\|_{E_{k}^{T}}^{2}\leq \|\TJk u_0\|_{L^2}^{2}+C_k T\|v\|_{E_{k}^{T}}^{2}.
$$
If $v_1$ and $v_2$ are two elements of $E_{k}^{T}$, then, by choosing $C_k$ large enough, we deduce by the same way that
$$
\|\Phi_{u_0}(v_1)-\Phi_{u_0}(v_2)\|_{E_{k}^{T}}^{2}\leq C_k T\|v_1-v_2\|_{E_{k}^{T}}(\|v_1\|_{E_{k}^{T}}+\|v_2\|_{E_{k}^{T}}).
$$
It follows that, if we set
\begin{equation}
    \label{I 10/01/2024}
    T(u_0):=\frac{1}{4C_k\|\TJk u_0\|_{L^2}^{2}},
\end{equation}
then the map $\Phi_{u_0}$ is a strict contraction on $B_{u_0}:=\enstq{u\in E_{k}^{T(u_0)}}{\|u\|_{E_{k}^{T(u_0)}}^{2}\leq 2\|\TJk u_0\|_{L^2}^{2}}$.\\

\textit{3) Fixed-point argument}. According to the Banach fixed-point argument, the map $\Phi_{u_0}$ admits a unique fixed-point on $B_{u_0}$ which is a solution of \eqref{NSHk}-\eqref{NSHk initial condition} on $[0,T(u_0)]$.
\paragraph{Global existence.} \textit{1) Energy estimate.} Let $T_\star>0$. We consider a solution $u$ of \eqref{NSHk}-\eqref{NSHk initial condition} belonging to $E_{k}^{T_{\star}}$. Then, we have
$$
\frac{1}{2}\frac{d}{dt}\|u\|_{L^2}^{2}+\|\nabla_{\Hg} u\|_{L^2}^{2}+\langle \mbb{P}\Jk (u\cdot \nabla_{\Hg}\Jk u),u\rangle_{L^2}=\langle(\Id-\mbb{P})\Laplace_{\Hg} u,u\rangle_{L^2}.
$$
However, thanks to Proposition \ref{Prop- Leray projector}, Items $2$ and $8$, we deduce that the right-hand side of the above inequality vanishes. Additionally, since $\Jk$ is self-adjoint on $L^2(\Hg^d)$ (see Proposition \ref{Prop- I 25/11/2023}, Item $5$) we have
$$
\langle \mbb{P}\Jk (u\cdot \nabla_{\Hg}\Jk u),u\rangle_{L^2} =\langle\Jk (u\cdot \nabla_{\Hg}\Jk u),u\rangle_{L^2} =-\frac{1}{2}\langle \divergence_{\Hg}(u),|\Jk u|^2\rangle_{L^2} =0.
$$
Thus, it follows that
$$
\frac{1}{2}\frac{d}{dt}\|u\|_{L^2}^{2}+\|\nabla_{\Hg}u\|_{L^2}^{2}=0.
$$
We deduce that
\begin{equation}
    \label{II 10/01/2024}
    \|u\|_{L^{\infty}_{T_\star}(L^2)}^{2}+2\|\nabla_{\Hg}u\|_{L^{2}_{T_\star}(L^2)}^{2}=\|\Tilde{\mathsf{J}}_k u_0\|_{L^2}^{2}.
\end{equation}

\textit{2) Uniqueness.} Let $u^1$ and $u^2$ be two solutions of \eqref{NSHk}-\eqref{NSHk initial condition} belonging to $\mathcal{C}_{b}([0,T(u_0)]; L^2)\cap L^2((0,T(u_0));\dot{H}^1)$. According to \eqref{II 10/01/2024}, $u^1$ and $u^2$ belong to $B_{u_0}$, and then coincide on $[0,T(u_0)]$ by uniqueness of the fixed-point of $\Phi_{u_0}$. Let us denote by $T_0$ the supremum of the time $t$ so that $u^1=u^2$ on $[0,t)$. We have $T_0\geq T(u_0)>0$. Suppose that $T_0<T_{\star}$. By continuity, we have $u^1(T_0)=u^2(T_0)$. Thanks to the time translation invariance of \eqref{NSHk}, we deduce that $u^1(T_0+\cdot)$ and $u^2(T_0+\cdot)$ are two solutions of \eqref{NSHk} with the same initial data (remark that $\TJk u^1(T_0)=u^1(T_0)$) and belong to $E_{k}^{t}$ for any $t\in(0,T_{\star}-T_0)$. Then, performing the same energy inequality as \eqref{II 10/01/2024} with $u^1(T_0+\cdot)$ and $u^2(T_0+\cdot)$, we deduce that $u^1(T_0+\cdot)$ and $u^2(T_0+\cdot)$ are two fixed-points of $\Phi_{u^1(T_0)}$ that belong to $B_{u^1(T_0)}$. In view of the uniqueness of the fixed-point of $\Phi_{u^1(T_0)}$ in $B_{u^1(T_0)}$, we finally get $u^1=u^2$ on $[0,T_0+T_1]$ with $T_1:=\min\{T(u_0),T_\star-T_0\}>0$. This is in contradiction with the definition of $T_0$, and then $T_0=T_\star$.\\

\textit{3) Blow-up argument}. It follows from the uniqueness of the solution of \eqref{NSHk}-\eqref{NSHk initial condition} that there exists a maximal existence time denoted by $T_{\star}>0$. Suppose that $T_{\star}$ is finite. If $t<T_{\star}$, then the solution $u(t+\cdot)$ exists at least on $[0,T(u(t))]$ (see \eqref{I 10/01/2024}). Then, for any $t\in(0,T_\star)$, we have $T(u(t))\leq T_\star-t$, that is 
$$
\|u(t)\|_{L^2}\geq \frac{1}{2\sqrt{C_k(T_{\star}-t)}}.
$$
This implies that $\lim_{t\rightarrow T_{\star}}\|u(t)\|_{L^2}=+\infty $, which contradicts \eqref{II 10/01/2024}. We conclude that $T_{\star}=+\infty$.\\
The fact that $u_k\in L^{\infty}(\Tilde{H}^\ell)$ and $\nabla_\Hg u_k\in L^2(\Tilde{H}^\ell)$ for any $\ell\in\R$, follows from Proposition \ref{Prop-Tilde Jk}, Item 1, since $\TJk u_k=u_k$ and $u_k\in L^{\infty}(L^2)\cap L^2(\dot{H}^1)$. According to Remark \ref{Remark- Regularistaion par troncature en Fourier}, we also deduce that $u(t)$ belongs to $\mathcal{C}^{\infty}_{b}(\Hg^d)$ for any $t>0$.
\end{proof}
\begin{rem}
In the proof, in order to show \eqref{propafation of Tilde Jk}, we use that $\TJk$ commutes with $\mbb{P}$ and $\divergence_{\Hg}$. These two properties are not satisfied by $\Jk$ (see Remark \ref{rem:remark sur les interaction entre Jk et Pj et P}). This is the spirit of our work: \textit{the regularity with respect to $\Tilde{\Laplace}_{\Hg}$ is propagated in the equation.}
\end{rem}
\section{Existence of global weak solutions: Proof of Theorem \ref{th main:Existence of weak solutions}}
\label{sec:Weak solutions} As the incompressible Navier-Stokes equations, the $L^2$-energy estimate for \eqref{NSHk}-\eqref{NSHk initial condition} is the key argument to obtain Leray-type theorem, namely the existence of global weak solutions of finite energy with initial data in $L^2(\Hg^d)$. Accordingly, we begin by establishing the following lemma that ensures that the sequence of solutions to the approximate systems remains bounded in the energy space if the initial data belongs to $L^2$.
\begin{lem}
\label{Lemme Energy equality} Let $u_0\in L^2(\Hg^d)$  be a horizontal vector field such that $\divergence_{\Hg}(u_0)=0$. For any $k\in\N$, we denote by $u_k$ the associate solution of \eqref{NSHk}-\eqref{NSHk initial condition}. Then we have 
\begin{equation}
\label{Energy equality}
\| u_{k}\|_{L^{\infty}(L^2)}^{2}+2\|\nabla_{\Hg}u_k\|_{L^2(L^2)}^{2}\leq\|u_0\|_{L^2}^{2}.
\end{equation}
\end{lem} 
\begin{proof}[Proof of Lemma \ref{Lemme Energy equality}.]
Estimate $\eqref{Energy equality}$ follows immediately from \eqref{energy L2 pour tout le monde} since $\|\TJk\|_{\mc{L}(L^2)}\leq 1$. 
\end{proof}
We are now able to prove Theorem \ref{th main:Existence of weak solutions}. 
\begin{proof}[Proof of Theorem \ref{th main:Existence of weak solutions}.]
In this proof the constant implied by $\lesssim$ is independent of $k$. From Lemma \ref{Lemme Energy equality}, without loss of generality we can assume that $(u_k)$ converges weakly-$\star$ in $L^{\infty}(L^2)$ and weakly in $L^2(\dot{H}^1)$ to $u$. Obviously, we have $\divergence_{\Hg}(u)=0$ and the sequence $(\mbb{P}\Laplace_{\Hg}u_k)$ converges to $\mbb{P}\Laplace_{\Hg}u$  in $\mathcal{D}'((0,+\infty)\times\Hg^d)$. We also deduce that $(\partial_t u_k)$ converges to $\partial_t  u$ in $\mathcal{D}'((0,+\infty)\times\Hg^d)$.

It remains to show that $(\mbb{P}\Jk(u_k\cdot\nabla_{\Hg}\Jk u_k))$ converges to $\mbb{P}(u\cdot\nabla_{\Hg}u)$ in $\mathcal{D}'((0,+\infty)\times\Hg^d)$. In view of the equation satisfied by $u_k$, that is \eqref{NSHk}-\eqref{NSHk initial condition}, and Lemma \ref{Lemme Energy equality}, we get
$$
\|\partial_{t}u_k\|_{L^2(H^{-Q/2})}\lesssim \|u_0\|_{L^2}+\|\mbb{P}\Jk(u_k\cdot\nabla_{\Hg}\Jk u_k)\|_{L^2(H^{-Q/2})}.
$$
Since $\divergence_{\Hg}(u_k)=0$, we have $\divergence_{\Hg}(\Jk u_k\otimes u_k)=u_k\cdot\nabla_{\Hg}\Jk u_k$. Then, according to Proposition \ref{Prop- I 25/11/2023}, Item 1 and Proposition \ref{Prop- Leray projector}, Item 9, we deduce that
\begin{align*}
\|\mbb{P}\Jk(u_k\cdot\nabla_{\Hg}\Jk u_k)\|_{L^2(H^{-\frac{Q}{2}})}&\leq \|\mbb{P}\|_{\mc{L}(H^{-Q/2})}\|\divergence_{\Hg}(\Jk u_k\otimes u_k)\|_{L^2(H^{-Q/2})}\\
&\lesssim\|\Jk u_k\otimes u_k\|_{L^2(H^{1-Q/2})}.
\end{align*}
Using Proposition \ref{Prop- I 25/11/2023}, Item 1 and Proposition \ref{proprietes des Sobolev H} in $\dot{H}^{1-Q/2}$, with $\Jk u_k\in L^{2}(\Hg^d)$ and $u_k\in L^2(\Hg^d)\cap\dot{H}^1(\Hg^d)$, with
$$
\|\Jk u_k\otimes u_k\|_{H^{1-Q/2}}
\lesssim
\|\Jk u_k\otimes u_k\|_{\dot H^{1-Q/2}}
\lesssim\|\Jk u_k\|_{L^2}\|\nabla_{\Hg}u_k\|_{L^2}.
$$
Hence, thanks to Lemma \ref{Lemme Energy equality} and using that $\|\Jk\|_{\mc{L}(L^2)}\leq 1$, we have
$$
\|\partial_t u_k\|_{L^2(H^{-Q/2})}\lesssim \|u_0\|_{L^2}+\|u_k\|_{L^{\infty}(L^2)}\|\nabla_{\Hg}u_k\|_{L^{2}(L^2)}\lesssim \|u_0\|_{L^2}+\|u_0\|^{2}_{L^2}.
$$
Thus, according to the above bound for $(\partial_t u_k)$, Lemma \ref{Lemme Energy equality} and the embedding $L^2(\Hg^d)\hookrightarrow H^{-Q/2}(\Hg^d)$, it follows from the Aubin-Lions theorem (see \cite{theoremdeSimon}) and the Cantor diagonal argument that $(u_k)$ converges strongly in $L^{2}_{loc}(\R_{+}\times\Hg^d)$ to $u$. Then, in view of the weak convergence of $(\nabla_{\Hg}\Jk u_k)$ to $\nabla_{\Hg}u$ in $L^2(L^2)$, we deduce that $(\mbb{P}\Jk(u_k\cdot\nabla_{\Hg} \Jk u_k))$ converges to $\mbb{P}(u\cdot\nabla_{\Hg}u)$ in $\mathcal{D}'((0,+\infty)\times\Hg^d)$. This shows that $u$ satisfies the momentum and the continuity equations in the sense of Definition \ref{def:Definition of weak solution for NSH}, Items $3$ and $4$ respectively. Let us pick $t$ and $t'$ in $[0,+\infty)$. Let $v$ be a horizontal vector field belonging to $\mathcal{D}(\Hg^d)$ and satisfying $\divergence_\Hg(v)=0$. Thanks to the momentum equation in the sense of Definition \ref{def:Definition of weak solution for NSH}, Item $3$, we deduce that
\begin{equation}
    \label{eq:weak continuity on dense subspace}
    |\langle u(t)-u(t'),v\rangle_{L^2}|\leq |t-t'|^{\frac 12}\left(\|u\|_{L^2(\dot{H}^1)}\|\nabla_\Hg v\|_{L^2}+\|u\otimes u\|_{L^2(H^{1-Q/2})}\|\nabla_\Hg v\|_{L^2}\right).
\end{equation}
Therefore, the left-hand side of \eqref{eq:weak continuity on dense subspace} converges to $0$ when $t$ goes to $t'$.  Using that $\mathbb{P}$ is a self-adjoint operator on $L^2$ and that $\mathbb{P}u=u$, we deduce that the left-hand side of \eqref{eq:weak continuity on dense subspace} converges to $0$, when $t$ goes to $t'$ even when $v$ simply is a (possibly non-divergence free) horizontal vector field  in $\mathcal{D}(\Hg^d)$. Moreover, since $u \in L^\infty( L^2)$, we get the same convergence result when $v$ simply belongs to $L^2$ by density.

This shows that $u$ belongs to $\mathcal{C}_w([0,+\infty),L^2)$ and satisfies the initial condition in the sense of Definition \ref{def:Definition of weak solution for NSH}, Item $2$. This concludes the proof.
\end{proof}

\section{Well-posedness in $\Tilde{H}^d$: proof of Theorem \ref{th main:Global well-posedness in TildeHd}}
\label{sec:Well-posedness in TildeHd}
In this section, we turn our attention to the case in which the initial data belong to $\Tilde{H}^d(\Hg^d)$.

\subsection{Existence of global solutions in $\Tilde{H}^d$}
In this subsection, we will show the following theorem.
\begin{theo}
\label{Global existence theorem in Tilde H1}
There exists a positive real number $\varepsilon$ such that for any horizontal vector field $u_0\in\Tilde{H}^{d}(\Hg^d)$ satisfying $\divergence_{\Hg}(u_0)=0$ and 
\begin{equation}
	\label{Smallness-Cond-H-tilde-d}
\|u_0\|_{\Tilde{H}^d}<\varepsilon,
\end{equation}
there exists a solution $u\in\mathcal{C}_{b}(\R_{+};\Tilde{H}^d)$ of \eqref{NSH}-\eqref{NSH initial condition} satisfying
$$
\|u\|_{L^{\infty}(\Tilde{H}^d)}^{2}+\|\nabla_{\Hg}u\|_{L^2(\Tilde{H}^d)}^{2}\leq 2\|u_0\|_{\Tilde{H}^d}^{2}.
$$
\end{theo}
\subsubsection{$\Tilde{H}^{d}$ energy estimates}
This subsection is devoted to the proof of energy estimates for the solutions of System \eqref{NSHk} in $\Tilde{H}^d(\Hg^d)$. The main idea is to propagate the $\Tilde{H}^d(\Hg^d)$ regularity in the energy in order to cancel the pressure, to use the Sobolev embedding for the two Sobolev-type scales $\Tilde{H}^\ell(\Hg^d)$ and $\dot{H}^\ell(\Hg^d)$ alternatively, and to close energy estimates from classical bootstrap argument (which works because the linear term in the pressure does not appear during the process).
\begin{lem}
\label{Energy estimate Hd right for k}
	If $\varepsilon>0$ small enough and $u_0\in\Tilde{H}^d(\Hg^d)$ satisfies \eqref{Smallness-Cond-H-tilde-d}, then, for any $k\in\N$, the following energy estimate
$$
\|u_k\|_{L^{\infty}(\Tilde{H}^d)}^{2}+\|\nabla_{\Hg}u_k\|_{L^{2}(\Tilde{H}^d)}^{2}\leq 2\|u_0\|_{\Tilde{H}^{d}}^{2}
$$
holds, where $u_k$ denotes the solution of \eqref{NSHk}-\eqref{NSHk initial condition}.
\end{lem}
\begin{proof}
Let $T>0$. In this proof we skip the index $k$ on $u_k$ and the constant implied by $\lesssim$ is independent of $k$ and $T$. Let us recall that $\TJk u=u$. Furthermore, in view of Proposition \ref{Prop-Tilde Jk}, Item 3, we have also $\TJk \nabla_{\Hg}u=\nabla_{\Hg}u$. Thus, in view of Proposition \ref{Prop-Tilde Jk} Item 1, $u$ and $\nabla_{\Hg}u$ respectively belong to $\mathcal{C}_{b}(\R_{+};\Tilde{H}^d(\Hg^d))$ and $L^2(\R_{+};\Tilde{H}^{d}(\Hg^d))$. Moreover, for any $t\in [0,+\infty)$, we have $u(t)\in\mathcal{C}^{\infty}(\R^{2d+1})$ (see Remark \ref{Remark- Regularistaion par troncature en Fourier}). By taking the $\Tilde{H}^d(\Hg^d)$ scalar product of the first line of \eqref{NSHk} and $u$, since $u$ is smooth in $\R_{+}\times\Hg^d$, we deduce that, for any $t\in[0,T]$, we have
\begin{align*}
\frac{1}{2}\frac{d}{dt}\|u(t)\|_{\Tilde{H}^d}^{2}-\langle\mbb{P}\Laplace_{\Hg}u(t),u(t)\rangle_{\Tilde{H}^d}=-\langle \mbb{P}\Jk (u(t)\cdot\nabla_{\Hg}\Jk u(t)),u(t)\rangle_{\Tilde{H}^d}.
\end{align*}
By using that $\mbb{P}$ is a self-adjoint operator on $\Tilde{H}^d(\Hg^d)$ and that $\mbb{P}u=u$ (see Proposition \ref{Prop- Leray projector}, Items $2$ and $8$), we get 
$$
-\langle\mbb{P}\Laplace_{\Hg}u(t),u(t)\rangle_{\Tilde{H}^d}=\|\nabla_{\Hg}u(t)\|_{L^2(\Tilde{H}^d)}^{2}.
$$
Besides, by using that $\mbb{P}u=u$ and $\Jk$ is self-adjoint on $\Tilde{H}^{d}$ (see Proposition \ref{Prop- I 25/11/2023}, Item $5$), we get
\begin{align*}
-\langle \mbb{P}\Jk (u(t)\cdot\nabla_{\Hg}\Jk u(t)),u(t)\rangle_{\Tilde{H}^d}&=-\langle \Jk (u(t)\cdot\nabla_{\Hg}\Jk u(t)),u(t)\rangle_{\Tilde{H}^d}\\
&=-\langle (u(t)\cdot\nabla_{\Hg}\Jk u(t)),\Jk u(t)\rangle_{\Tilde{H}^d}.
\end{align*}
Therefore, for any $t'$ and $t$ in $(0,T)$ such that $t'<t$, we have
$$
\frac{1}{2}\frac{d}{dt}\|u(t')\|_{\Tilde{H}^d}^{2}+\|\nabla_{\Hg}u(t')\|_{\Tilde{H}^d}^{2}=-\langle u(t')\cdot\nabla_{\Hg}\Jk u(t'),\Jk u(t')\rangle_{\Tilde{H}^d}.
$$
By integrating the above identity between $0$ and $t$ and taking the supremum on $[0,T]$, it follows that
\begin{equation}
\label{I 21 12 2022}
\|u\|_{L^{\infty}_{T}(\Tilde{H}^d)}^{2}+2\|\nabla_{\Hg}u\|_{L^{2}_{T}(\Tilde{H}^d)}^{2}
\leq 
 \|\TJk u_0\|_{\Tilde{H}^d}^{2}
+ 2 \int_{0}^{T}|\langle 
u_k(\tau)\cdot\nabla_{\Hg}\Jk u_k(\tau),\Jk u_k(\tau)\rangle_{\Tilde{H}^d}|d\tau.
\end{equation}
We shall estimate the nonlinear term by using the following lemma that we will prove later.
\begin{lem}
\label{II 06 10 2022}
There exists a constant $C>0$ such that for any function $a$ and $b$ in $\Tilde{H}^d(\Hg^d)$ and $c$ such that $\nabla_{\Hg}c$ belongs to $\Tilde{H}^{d}(\Hg^d)$, we have
$$
|\langle ab,c\rangle_{\Tilde{H}^d}|\leq C\|a\|_{\Tilde{H}^{d}}\|b\|_{\Tilde{H}^d}\|\nabla_{\Hg}c\|_{\Tilde{H}^{d}}.
$$
\end{lem}
By using Lemma \ref{II 06 10 2022} and Proposition \ref{Prop- I 25/11/2023}, Items 1 and 3, we deduce from \eqref{I 21 12 2022} that there is a constant $C_{\star}>0$, independent of the index $k\in\N$ such that for any $T>0$, we have
\begin{align*}
\|u\|_{L^{\infty}_{T}(\Tilde{H}^d)}^{2}+2\|\nabla_{\Hg}u\|_{L^{2}_{T}(\Tilde{H}^d)}^{2}&\leq \|u_0\|_{\Tilde{H}^d}^{2}+C_{\star}\|u\|_{L^{\infty}_{T}(\Tilde{H}^d)}\|\nabla_{\Hg}\Jk u\|_{L^{2}_{T}(\Tilde{H}^d)}^{2}\\
&\leq \|u_0\|_{\Tilde{H}^d}^{2}+C_{\star}\|u\|_{L^{\infty}_{T}(\Tilde{H}^d)}\|\nabla_{\Hg} u\|_{L^{2}_{T}(\Tilde{H}^d)}^{2}.
\end{align*}
Let us define 
$$
T_{\star}:=\sup\enstq{T\geq 0}{\|u\|_{L^{\infty}_{T}(\Tilde{H}^d)}^{2}\leq 2\|u_0\|_{\Tilde{H}^d}^{2}}.
$$
According to the continuity of $t\in[0,+\infty)\mapsto\|u(t)\|_{\Tilde{H}^d}$ (see Lemma \ref{II 03/12/2024}) and the properties of $\TJk$, we have $T_\star>0$. Supposed that $T_\star<+\infty$. If $T\in(0,T_\star)$, then it follows from the definition of $T_\star$ and the smallness condition on $\|u_0\|_{\Tilde{H}^d}$, that 
$$
\|u_k\|_{L^{\infty}_{T}(\Tilde{H}^d)}^{2}+(2-\sqrt{2}C_\star\varepsilon)\|\nabla_\Hg u\|_{L^{2}_{T}(\Tilde{H}^d)}\leq \|u_0\|_{\Tilde{H}^d}^{2}.
$$
Then, by choosing $\varepsilon\in(0,\sqrt{2}/C_\star)$, we deduce by a standard bootstrap argument that $T_\star=+\infty$. This concludes the proof of Lemma \ref{Energy estimate Hd right for k}.
\end{proof}

\begin{proof}[Proof of Lemma \ref{II 06 10 2022}]
By density in $\Tilde{H}^d$, without loss of generality, we can assume that $a$ and $b$ belong to $\mathcal{S}(\Hg^d)$. First by using the H\"older estimate and the Sobolev embedding $\dot{H}^{1}(\Hg^d)\hookrightarrow L^{\frac{2Q}{Q-2}}(\Hg^d)$ (recall Proposition \ref{Propriete Sobolev sur Heisenberg} Item 3 and Remark \ref{Remark- Gauche à droite}) , we get
\begin{align*}
|\langle ab,c\rangle_{\Tilde{H}^d}|&\lesssim \|(-\Tilde{\Laplace}_{\Hg})^{\frac{d}{2}}(ab)\|_{L^{\frac{2Q}{Q+2}}}\|(-\Tilde{\Laplace}_{\Hg})^{\frac{d}{2}}c\|_{L^{\frac{2Q}{Q-2}}}
\lesssim \sum_{\gamma\in[\![1,2d]\!]^d}\|\Tilde{P}^{\gamma}(ab)\|_{L^{\frac{2Q}{Q+2}}}\|\nabla_{\Hg}c\|_{\Tilde{H}^d}.
\end{align*}
Let $\gamma\in[\![1,2d]\!]^{d}$. Thanks to the Leibniz formula and the H\"older estimates, we get
$$
\|\Tilde{P}^{\gamma}(ab)\|_{L^{\frac{2Q}{Q+2}}}\lesssim\sum_{\ell=0}^{d}\sum_{\substack{\alpha\in[\![1,2d]\!]^\ell\\ \beta\in[\![1,2d]\!]^{d-\ell}}}\|\Tilde{P}^{\alpha}a\|_{L^{\frac{2Q}{Q-2(d-\ell)}}}\|\Tilde{P}^{\beta}b\|_{L^{\frac{2Q}{Q-2\ell}}}.
$$
Then, for all $\ell\in[\![0,d]\!]$, applying the Sobolev embedding  $\Tilde{H}^{d-\ell}(\Hg^d)\hookrightarrow L^{\frac{2Q}{Q-2(d-\ell)}}(\Hg^d) $ and $\Tilde{H}^\ell(\Hg^d)\hookrightarrow L^{\frac{2Q}{Q-2\ell}}(\Hg^d)$ (recall Proposition \ref{Propriete Sobolev sur Heisenberg} Item 3 and Remark \ref{Remark- Gauche à droite}), it follows that for any $\alpha\in[\![1,2d]\!]^{\ell}$ and $\beta\in[\![1,2d]\!]^{d-\ell}$,
$$
\|\Tilde{P}^{\alpha}a\|_{L^{\frac{2Q}{Q-2(d-\ell)}}}\lesssim\|\Tilde{P}^{\alpha}a\|_{\Tilde{H}^{d-\ell}}\lesssim\|a\|_{\Tilde{H}^d}\ \ \ \text{and}\ \ \ \|\Tilde{P}^{\beta}b\|_{L^{\frac{2Q}{Q-2\ell}}}\lesssim\|\Tilde{P}^{\beta}b\|_{\Tilde{H}^\ell}\lesssim\|b\|_{\Tilde{H}^d}.
$$
This concludes the proof of  Lemma \ref{II 06 10 2022}.
\end{proof}
\subsubsection{Convergence}\label{sec-convergence}
In this subsection we complete the proof of Theorem \ref{Global existence theorem in Tilde H1} by showing that the sequence $(u_k)$ converges in $\mathcal{D}'((0,+\infty)\times\Hg^d)$ to a solution of \eqref{NSH}-\eqref{NSH initial condition} satisfying the suitable energy estimates.

The first step is to get the convergence of the linear part. According to the Lemma \ref{Energy estimate Hd right for k}, using the weak compactness of spaces $L^{\infty}(\Tilde{H}^d)$ and $L^2(\Tilde{H}^d)$ and identifying the limits in $\mathcal{D}'(\R_+\times\Hg^d)$, we deduce that, up to extract a subsequence, $(u_k)$ converges weakly-$\star$ to a function $u$ and $(\nabla_\Hg u_k)$ converges weakly to $\nabla_\Hg u$ in $L^2(\Tilde{H}^d)$. Moreover, thanks to the properties of the weak and the weak-$\star$ convergence and the energy estimate from Lemma \ref{Energy estimate Hd right for k}, for any $T>0$, we have
\begin{equation*}
\|u\|_{L^{\infty}_{T}(\Tilde{H}^d)}^{2}+\|\nabla_{\Hg}u\|_{L_{T}^2(\Tilde{H}^{d})}^{2}\leq 2\|u_0\|_{\Tilde{H}^d}^{2}.
\end{equation*}
Thanks to the weak convergence of $(u_k)$ to $u$ and the continuity properties of $\mathbb{P}$ (see Proposition \ref{Prop- Leray projector}, Item 9), we have the following convergence in $\mathcal{D}'((0,+\infty)\times\Hg^d)$ 
\begin{align*}
\lim_{k\rightarrow+\infty}\partial_tu_k=\partial_tu,\ \, \ \ 
\lim_{k\rightarrow+\infty}\mbb{P}\Laplace_{\Hg}u_k=\mbb{P}\Laplace_{\Hg}u\ \ \text{and}\ \ \divergence_{\Hg}(u)=0.
\end{align*}
Moreover, the sequence $\TJk u_{0}$ converges to $u_0$ in $\Tilde{H}^d$ (see Proposition \ref{Prop-Tilde Jk}, Item 6). To conclude, it is sufficient to show that $(\mbb{P}\Jk (u_k\cdot\nabla_{\Hg}\Jk u_k))$ converges to $\mbb{P}(u\cdot\nabla_{\Hg}u)$ in $\mathcal{D}'((0,+\infty)\times\Hg^d)$. To this aim, according to the classical strategy, we need to have a strong convergence on the sequence $(u_k)$, that we can get by bounding $(\partial_{t}u_k)$ in a suitable space, using an Aubin-Lions type lemma and a locally compact embedding. \begin{proof}[End of the proof of Theorem \ref{Global existence theorem in Tilde H1}]
We shall show the local (strong) compactness of the sequence $(u_k)$ in a suitable space. In this proof the constant implied by the symbol $\lesssim$ is assumed to be independent of the parameter $k\in\N$. 

\textit{\underline{1) If $\varphi$ belongs to $\mathcal{D}((0,+\infty)\times\Hg^d)$, then $(\varphi u_k)$ is bounded in $H^1(\R_{+};W^{-1,Q}_{\Hg})$.}} By using the embedding $L^{\frac{Q}{2}}(\Hg^d)\hookrightarrow W^{-1,Q}_{\Hg}(\Hg^d)$ (see Proposition \ref{Propriete Sobolev sur Heisenberg}, Item 3), we deduce that
\begin{equation}
\label{eq:borne avec injection Sobolev et integrabilite de phi}
\|\varphi u_k\|_{L^2(W^{-1,Q}_{\Hg})}\lesssim \|\varphi u_k\|_{L^2(L^{\frac{Q}{2}})}\lesssim \|\varphi\|_{L^2(L^{Q})}\|u_k\|_{L^{\infty}(L^{Q})}.
\end{equation}
Hence, according to the Sobolev embedding $\Tilde{H}^{d}(\Hg^d)\hookrightarrow L^Q(\Hg^d)$, we obtain
$$
\|\varphi u_k\|_{L^2(W^{-1,Q}_{\Hg})}\lesssim\|u_k\|_{L^{\infty}(\Tilde{H}^d)}.
$$
According to Lemma \ref{Energy estimate Hd right for k}, it follows that
\begin{equation}
\label{eq:borne sur phi uk}
\|\varphi  u_k\|_{L^2(W^{-1,Q}_{\Hg})}\lesssim\|u_0\|_{\Tilde{H}^d}.
\end{equation}
Thus, $(\varphi u_k)$ is bounded in $L^2(W^{-1,Q}_{\Hg})$. 
Let us begin by remarking that, according to Lemma \ref{I 28/08/2023}, we have
\begin{align*}
\partial_t(\varphi u_k)&=(\partial_t\vphi) u_k+\vphi\mbb{P}\Laplace_{\Hg}u_k-\vphi\mbb{P}\Jk(u_k\cdot\nabla_{\Hg}\Jk u_k)\\
&=(\partial_t\vphi) u_k+\vphi\Laplace_{\Hg}u_k+\vphi\Pi_{\Hg}\partial_s u_k-\vphi\mbb{P}\Jk(u_k\cdot\nabla_{\Hg}\Jk u_k).
\end{align*}
Let us now estimate all the terms in $L^2(W^{-1,Q}_{\Hg})$ .\\
$\bullet$ \textit{Estimate on $(\partial_t\vphi) u_k$.} Since $\partial_t\vphi$ belongs to $L^{2}(L^{Q})$, as for \eqref{eq:borne sur phi uk}, we obtain
$$
\|(\partial_t\varphi) u_k\|_{L^2(W^{-1,Q}_{\Hg})}\lesssim\|u_0\|_{\Tilde{H}^d}.
$$
$\bullet$ \textit{Estimate on $\varphi \mathbb{P}\Jk(u_k\cdot\nabla_{\Hg}\Jk u_k)$.} Using the embedding $L^{\frac{Q}{2}}(\Hg^d)\hookrightarrow W^{-1,Q}_{\Hg}(\Hg^d)$, that $\vphi$ belongs to $L^{\infty}(\R_{+}\times\Hg^d)$ and the embedding $\Tilde{H}^{d-1}(\Hg^d)\hookrightarrow L^{\frac{Q}{2}}(\Hg^d)$, we get
\begin{equation}
\label{Est-W-1-Q-by-H-tilde-d-1}
\|\varphi \mathbb{P}\Jk(u_k\cdot\nabla_{\Hg}\Jk u_k)\|_{L^2(W^{-1,Q}_{\Hg})}\lesssim\|\mathbb{P}\Jk (u_k\cdot\nabla_{\Hg}\Jk u_k)\|_{L^2(\Tilde{H}^{d-1})}.
\end{equation}
Since $\mathbb{P}$ belongs to $\mathcal{L}(\Tilde{H}^{d-1})$ (see Proposition \ref{Prop- Leray projector}, Item 8) and $\Jk$ is bounded by $1$ in this space (see Proposition \ref{Prop- I 25/11/2023}, Item 1), we have
$$
\|\mathbb{P}\Jk (u_k\cdot\nabla_{\Hg}\Jk u_k)\|_{L^2(\Tilde{H}^{d-1})}
\lesssim\|u_k\cdot\nabla_{\Hg}\Jk u_k\|_{L^2(\Tilde{H}^{d-1})}.
$$ 
Then, by using the tame estimates of Proposition \ref{proprietes des Sobolev H} for $u_k\cdot\nabla_{\Hg}\Jk u_k$ on $\Tilde{H}^{d-1}(\Hg^d)$ with $u_k$ and $\nabla_{\Hg}\Jk u_k$ in $\Tilde{H^{d}}(\Hg^d)\cap L^2(\Hg^d)$ (the fact that these two functions belong to $L^2(\Hg^d)$ follows from the continuity of the operator $\TJk$ from $\Tilde{H}^d(\Hg^d)$ to $L^2(\Hg^d)$ from Proposition \ref{Prop-Tilde Jk}, Item 1), since $\Jk$ is bounded by $1$ in $\mc{L}(\Tilde{H}^d)$, the energy estimates from Lemma \ref{Energy estimate Hd right for k} yield
$$
\|\varphi \mathbb{P}\Jk(u_k\cdot\nabla_{\Hg}\Jk u_k)\|_{L^2(W^{-1,Q}_{\Hg})}\lesssim\|u_0\|_{\Tilde{H}^d}^{2}.
$$
$\bullet$ \textit{Estimate on $\varphi \Pi_{\Hg}\partial_s u_k$.} Similarly to \eqref{Est-W-1-Q-by-H-tilde-d-1}, we have 
$$
\|\varphi \Pi_{\Hg}\partial_su_k\|_{L^2(W^{-1,Q}_{\Hg})}\lesssim\|\Pi_{\Hg}\partial_s u_k\|_{L^2(\Tilde{H}^{d-1})}. 
$$
Let us recall that $\Pi_{\Hg}:=8(\Id-\mbb{P})\circ\mathfrak{S}$ (see Lemma \ref{I 28/08/2023}). Thus, since $\mbb{P}$ belongs to $\mathcal{L}(\Tilde{H}^{d-1})$, we deduce that $\Pi_{\Hg}$ belongs to $\mathcal{L}(\Tilde{H}^{d-1})$. Accordingly, we have
$$
\|\Pi_{\Hg}\partial_s u_k\|_{L^2(\Tilde{H}^{d-1})}\lesssim\|\partial_s u_k\|_{L^2(\Tilde{H}^{d-1})}.
$$
Besides, according to Proposition \ref{prop- Ds est d'ordre 2} applied alternatively with $\dot{H}^1(\Hg^d)$ and $\Tilde{H}^1(\Hg^d)$, we have
$$
\|\partial_s u_k\|_{L^2(\Tilde{H}^{d-1})}=\||D_s|^{\frac 12}|D_s|^{\frac 12}u_k\|_{L^2(\Tilde{H}^{d-1})} \lesssim\||D_s|^{\frac 12}u_k\|_{L^2(\Tilde{H}^d)}\lesssim\|\nabla_{\Hg}u_k\|_{L^2(\Tilde{H}^{d})}.
$$
Then, it follows from the energy estimates in Lemma \ref{Energy estimate Hd right for k} that
$$
\|\varphi \Pi_{\Hg}\partial_s u_k\|_{L^2(W^{-1,Q}_{\Hg})}\lesssim\|u_0\|_{\Tilde{H}^d}.
$$
$\bullet$ \textit{Estimate on $\varphi\Laplace_{\Hg}u_k$.} Since $\varphi$ and $u_k$ are smooth (this follows from Remark \ref{Remark- Regularistaion par troncature en Fourier} and the third line of \eqref{NSHk}, namely $\TJk u_k=u_k$), we can write 
\begin{align}
    \|\varphi\Laplace_{\Hg}u_k&\|_{L^2(W^{-1,Q}_{\Hg})}\nonumber\\
    &\lesssim
    \|\Laplace_{\Hg}(\varphi u_k)\|_{L^2(W^{-1,Q}_{\Hg})}
    + 
    \|(\Laplace_{\Hg}\varphi)u_k\|_{L^2(W^{-1,Q}_{\Hg})}+\|\nabla_{\Hg}\varphi\cdot\nabla_{\Hg}u_k\|_{L^2(W^{-1,Q}_{\Hg})}.\label{III 28 10 2022}
\end{align}
By using the embedding $L^{\frac{Q}{2}}(\Hg^d)\hookrightarrow W^{-1,Q}_{\Hg}(\Hg^d)$, the Hölder estimate and the embedding $\Tilde{H}^d(\Hg^d)\hookrightarrow L^Q(\Hg^d)$, we deduce from Lemma \ref{Energy estimate Hd right for k} that
\begin{equation}
    \label{IV 28 10 2022}
    \|(\Laplace_{\Hg}\varphi)u_k\|_{L^2(W^{-1,Q}_{\Hg})}+\|\nabla_{\Hg}\varphi\cdot\nabla_{\Hg}u_k\|_{L^2(W^{-1,Q}_{\Hg})}\lesssim\|u_0\|_{\Tilde{H}^d}.
\end{equation}
Besides, using the Leibniz rules, the Sobolev embedding and the energy estimate of Lemma \ref{Energy estimate Hd right for k}, it follows that
\begin{align*}
    \|\Laplace_{\Hg}(\varphi u_k)\|_{L^2(W^{-1,Q}_{\Hg})}&\lesssim \|\nabla_{\Hg}(\varphi u_k)\|_{L^2(L^Q)}\lesssim\|u_0\|_{\Tilde{H}^d}.
\end{align*}
Combining this estimate with \eqref{III 28 10 2022} and \eqref{IV 28 10 2022}, we deduce that
$$
\|\varphi\Laplace_{\Hg}u_k\|_{L^2(W^{-1,Q}_{\Hg})}\lesssim\|u_0\|_{\Tilde{H}^d}.
$$
We conclude that
$$
\|\partial_t(\varphi u_k)\|_{L^2(W^{-1,Q}_{\Hg})}\lesssim\|u_0\|_{\Tilde{H}^d}+\|u_0\|_{\Tilde{H}^d}^{2}.
$$
We conclude that $(\varphi u_k)$ is bounded in $H^1(W^{-1,Q}_{\Hg})$.\\

\textit{2) \underline{There exists a subsequence $(u_{k_l})$ of $(u_{k})$ such that for any $\varphi\in\mathcal{D}(\R_+\times\Hg^d)$, the}}\\
\textit{\underline{ sequence $(\varphi u_{k_l})$ converges to $\varphi u$ in $L^2(\R_{+};L^{Q})$.}} According to the Cantor diagonal extraction, we are reduced to show that if $\varphi\in \mathcal{D}(\R_+\times\Hg^d)$, there is a subsequence $(u_{k_l})$ of $(u_k)$ such that $(\varphi u_{k_l})$ converge to $\varphi u$ in $L^2(L^Q)$. Let $\varphi\in\mathcal{D}((0,+\infty)\times\Hg^d)$. Since $(\varphi u_k)$ is smooth in $\R_{+}\times \Hg^d$ and bounded in $H^1(\R_{+};W^{-1,Q}_{\Hg})$, for any $\tau$ and $\tau'$ in $\R_{+}$, we have
$$
    \|(\varphi u_k)(\tau)-(\varphi u_k)(\tau')\|_{W^{-1,Q}_{\Hg}}
    \lesssim |\tau-\tau'|^{\frac 12}\|\partial_{t}(\varphi u_k)\|_{L^2(W^{-1,Q}_{\Hg})}
    \lesssim_{\varphi, u_0}|\tau -\tau'|^{\frac 12}.
$$
We deduce that $(\varphi u_k)$ is equicontinuous in $\R_{+}$ with value in $W^{-1,Q}_{\Hg}$. From the energy estimate, the Hölder estimate and the embedding $\Tilde{H}^d\hookrightarrow L^Q$, it follows that $(\varphi u_k)$ is bounded in $L^{\infty}(L^Q)$. Applying Proposition \ref{Sobolev compacity on homogeneous Lie groups} to $(\varphi u_k(t))$ for $t\in\R_{+}$, we deduce that $(\varphi u_{k}(t))$ is relatively compact in $W^{-1,Q}_{\Hg}$ for any $t\in\R_{+}$. Thanks to the Arzela-Ascoli theorem, we conclude that, up to extraction, that $(\varphi u_k)$ converge to $\varphi u$ in $\mathcal{C}(\R_{+},W^{-1,Q}_{\Hg})$. The next step is to show that $(\varphi u_k)$ converges in fact to $\varphi u$ in $L^2(\R_{+};L^Q)$. We will need the following lemma (the proof is an adaptation of the proof \cite[(A.17) p. 195]{TheMathematicalAnalysisoftheIncompressibleEulerandNavierStokesEquationsAnIntroduction} and is left to the reader).
\begin{lem}
\label{lemme compacite tranquil classique}
Let $E$, $F\subset F'$ three Banach spaces, $A\in \mathcal{L}(E,F)$ and $B\in\mathcal{L}(F,F')$ such that $ R(A)\subset N(\Id-B)$. Assume that $A$ is a compact operator. Then, for all $\eta>0$, there exists $C_{\eta}>0$ such that for all $u\in E$, we have
$$
\|Au\|_{F}\leq \eta\|u\|_{E}+C_{\eta}\|B\circ A u\|_{F'}.
$$
\end{lem}

Choose a bounded open subset $\Omega$ of $\Hg^d$ such that $supp(\varphi)\subset \R_{+}\times\Omega$.
We now introduce $\chi_0\in\mathcal{D}(\Hg^d)$ and $\chi_1\in\mathcal{D}(\Hg^d)$ such that $\chi_0=1$ on $\Omega$, and $\chi_1 = 1$ on the support of $\chi_0$, so that $\chi_1 \chi_0 = \chi_0$.  
 From Proposition \ref{Sobolev compacity on homogeneous Lie groups}, the multiplication by $\chi_0$ is a compact operator from $W^{1,Q}_{\Hg}$ to $L^Q(\Hg^d)$. Lemma \ref{lemme compacite tranquil classique} then provides, for any $\eta>0$, a constant $C_{\eta}>0$ such that for any $v\in W^{1,Q}_{\Hg}$,
$$
	\|\chi_0 v\|_{L^Q}\leq \eta\|v\|_{W^{1,Q}_{\Hg}}+C_{\eta}\|\chi_0 v\|_{W^{-1,Q}_{\Hg}}.
$$
It follows that for all $t>0$, 
$$
\|\varphi(t,\cdot) v\|_{L^Q}\leq \eta\|v\|_{W^{1,Q}_{\Hg}}+C_{\eta}\|\varphi (t, \cdot) v\|_{W^{-1,Q}_{\Hg}}.
$$
Let us consider $T>0$ so that $supp(\varphi)\subset [0,T]\times\Hg^d$. Let $\eta>0$. Since $(\varphi u_k)$ converge in $L^{\infty}(\R_{+};W^{-1,Q}_{\Hg})$, there is $N_{\eta, T}\in \N$ such that for any $l,k\geq N_{\eta,T}$ we have 
$$
\|\varphi u_k-\varphi u_l\|_{L^{\infty}(W^{-1,Q}_{\Hg})}<\frac{\eta}{C_{\eta}T^{\frac 12}}.
$$
Then, using the Sobolev embedding and the energy estimate of Lemma \ref{Energy estimate Hd right for k}, we have, for any $k$ and $l$ larger than $N_{\eta,T}$,
$$
\|\varphi u_k-\varphi u_l\|_{L^{2}(L^{Q})}\leq \eta\|\varphi u_k-\varphi u_l\|_{L^2(W^{1,Q}_{\Hg})}+\eta\lesssim_{\varphi,u_0}\eta.
$$
This shows that $(\varphi u_k)$ converges to $\varphi u$ in $L^2(\R_{+};L^Q)$.\\

\textit{3) \underline{Convergence of the nonlinear term.}} Let us begin by showing that, up to extract a subsequence, $(u_k\cdot\nabla_{\Hg}\Jk u_k)$ converges to $u\cdot\nabla_{\Hg} u$ in $\mathcal{D}'(\R_+\times\Hg^d)$. Let $\varphi\in\mathcal{D}
(\R_+\times\Hg^d)$. We pick $\chi\in\mathcal{D}(\R_+\times\Hg^d)$ so that $\chi=1$ on $supp(\varphi)$. Since $u_k\cdot\nabla_{\Hg}\Jk u_k$ belongs to $L^2(L^{\frac{Q}{2}})$, we can write
\begin{equation}
    \label{II 23 10 2022}
    \langle u_k\cdot\nabla_{\Hg} \Jk u_k,\varphi\rangle_{\mathcal{D}',\mathcal{D}}=\int_{\R_{+}\times\Hg^d}(\varphi u_k)\cdot\nabla_{\Hg}\Jk u_k\, dtdx.
\end{equation}
Using the H\"older estimate and the Sobolev embedding $\Tilde{H}^d\hookrightarrow L^Q$, we deduce that $(\varphi\chi u_k)$ is strongly convergent to $\varphi u=\varphi\chi u$ in $L^2(\R_{+};L^{\frac{Q}{Q-1}})=L^{2}(\R_{+};L^{Q})^{'}$ (because $\varphi\in L^{\infty}(\R_{+};L^{Q/2})$) and $(\nabla_{\Hg}\Jk u_k)$ is weakly convergent to $\nabla_{\Hg}u$ in $L^{2}(\R_{+};L^Q)$. We deduce from \eqref{II 23 10 2022} and from the properties of the weak convergence (see \cite[Proposition 3.5, Item 4, p. 58]{FunctionalAnalysisSobolevSpacesandPartialDifferentialEquations}) that
$$
\lim_{k\rightarrow+\infty}\langle u_k\cdot\nabla_{\Hg} \Jk u_k,\varphi\rangle_{\mathcal{D}',\mathcal{D}}=\int_{\R_{+}\times\Hg^d} \varphi\left(u\cdot\nabla_{\Hg} u\right)dtdx.
$$
We have proved that $(u_k\cdot\nabla_{\Hg}\Jk u_k)$ converges to $u\cdot\nabla_{\Hg}u$ in $\mathcal{D}'(\R_+\times\Hg^d)$.\\
Since the sequence $(u_k\cdot\nabla_{\Hg}\Jk u_k)$ is bounded in $L^2(\Tilde{H}^{d-1})$, we deduce from Proposition \ref{Prop- I 25/11/2023}, Item $6$, that $(\Jk-\Id)(u_k\cdot\nabla_{\Hg}\Jk u_k)$ converges weakly to $0$ in $L^2(\Tilde{H}^{d-1})$ and, up to extract a subsequence, that $(u_k\cdot\nabla_\Hg\Jk u_k)$ converges weakly to $u\cdot\nabla_\Hg u$  in $L^2(\Tilde{H}^{d-1})$. According to the continuity of $\mathbb{P}$ on $\Tilde{H}^{d-1}(\Hg^d)$, this ensures that $u$ satisfies \eqref{NSH}.\\

\textit{4) \underline{Continuity in time.}} We finish the proof by showing the continuity in time for the solution, that is $u\in\mathcal{C}(\R_{+};\Tilde{H}^d)$. Since $(-\Tilde{\Laplace}_{\Hg})^{\frac d2}u$ belongs to $L^{2}_{loc}(\R_{+};H^{1})$ and $\partial_t(-\Tilde{\Laplace}_{\Hg})^{\frac d2}u\in L^{2}_{loc}(\R_{+};H^{-1})$, we deduce by interpolation that $u$ belongs to $\mathcal{C}(\R_{+};\Tilde{H}^d)$.
\end{proof}

\subsection{Stability and uniqueness}
In this section our goal is to establish the stability of solutions of \eqref{NSH}.
\begin{theo}
\label{Stabilite dans Tilde H}
There exists a positive constant $C$ such that for any $T>0$ and for any solutions $u$ and $v$ of \eqref{NSH} such that $u$ and $v$ belong to $\mathcal{C}_{b}([0,T];\Tilde{H}^d)$ and $\nabla_{\Hg}u$ and $\nabla_{\Hg}v$ belong to $L^2([0,T];\Tilde{H}^{d})$, we have
\begin{align*}
\|u-v\|_{L^{\infty}_{T}(\Tilde{H}^d)}^{2}+\|\nabla_{\Hg}(u-&v)\|_{L^{2}_{T}(\Tilde{H}^{d})}^{2}\nonumber\\
&\leq \|u(0)-v(0)\|_{\Tilde{H}^d}^{2}\exp{\left(C\|\nabla_{\Hg}v\|_{L^{2}_{T}(\Tilde{H}^{d})}^{2}+C\|\nabla_{\Hg}u\|_{L^{2}_{T}(\Tilde{H}^{d})}^{2}\right)}.
\end{align*}
\end{theo}
From this result we deduce immediately the uniqueness of the solutions constructed in Theorem \ref{Global existence theorem in Tilde H1}. This thus concludes the proof of Theorem \ref{th main:Global well-posedness in TildeHd}, Items 1 and 2. 
\begin{proof}
In this proof the constant implied by $\lesssim$ is independent of $T$. Without loss of generality, up to regularize $u$ and $v$ with respect to the time and space variable, we can assume that $u$ and $v$ are smooth. Let us set $w:=v-u$ and, for any $t\in [0,T]$,
$$
e(t):=\|w(t)\|_{\Tilde{H}^d}^{2}+2\|\nabla_{\Hg}w\|_{L^{2}_{t}(\Tilde{H}^d)}^{2}.
$$
Let $t\in[0,T]$. By developing the expression of $e(t)$, we obtain
\begin{equation}
    \label{I 6 12 2022}
    e(t)=\|u(t)\|_{\Tilde{H}^d}^{2}+2\|\nabla_{\Hg}u\|_{L^{2}_{t}(\Tilde{H}^d)}^{2}+\|v(t)\|_{\Tilde{H}^d}^{2}+2\|\nabla_{\Hg}v\|_{L^{2}_{t}(\Tilde{H}^d)}^{2}-2E(t),
\end{equation}
where 
$$
E(t):=\langle v(t),u(t)\rangle_{\Tilde{H}^d}+2\int_{0}^{t}\langle \nabla_{\Hg}v(\tau),\nabla_{\Hg}u(\tau)\rangle_{\Tilde{H}^d}d\tau.
$$
Besides, by a direct calculus, using the equation satisfied by $u$ and $v$, and keeping in mind that $[(-\Tilde{\Laplace}_{\Hg})^{\frac{d}{2}},\nabla_{\Hg}]=0$ in order to cancel the pressure, we get
\begin{align*}
\int_{0}^{t}\langle\nabla_{\Hg}v(\tau),\nabla_{\Hg}u(\tau)\rangle_{\Tilde{H}^d}d\tau=&\langle v(0),u(0)\rangle_{\Tilde{H}^d}-\langle v(t),u(t)\rangle_{\Tilde{H}^d}-\int_{0}^{t}\langle \nabla_{\Hg}v(\tau),\nabla_{\Hg}u(\tau)\rangle_{\Tilde{H}^d}d\tau
\\
&-\int_{0}^{t}\langle v(\tau)\cdot\nabla_{\Hg}v(\tau),u(\tau)\rangle_{\Tilde{H}^d}d\tau
-\int_{0}^{t}\langle v(\tau), u(\tau)\cdot\nabla_{\Hg}u(\tau)\rangle_{\Tilde{H}^d}d\tau,
\end{align*}
that is 
\begin{equation*}
E(t)=\langle v(0),u(0)\rangle_{\Tilde{H}^d}
-\int_{0}^{t}\langle v(\tau)\cdot\nabla_{\Hg}v(\tau),u(\tau)\rangle_{\Tilde{H}^d}d\tau-\int_{0}^{t}\langle v(\tau), u(\tau)\cdot\nabla_{\Hg}u(\tau)\rangle_{\Tilde{H}^d}d\tau.
\end{equation*}
In view of \eqref{I 6 12 2022} and by computing the energy for $u$ and $v$ in $\Tilde{H}^d(\Hg^d)$, it follows that
\begin{align}
e(t)&=\|u(0)\|_{\Tilde{H}^d}^{2}+\|v(0)\|_{\Tilde{H}^d}^{2}-2\langle v(0),u(0)\rangle_{\Tilde{H}^d}\nonumber\\
&\ \ -2\int_{0}^{t}\left( \langle u(\tau)\cdot\nabla_{\Hg}u(\tau),u(\tau)\rangle_{\Tilde{H}^d}+\langle v(\tau)\cdot\nabla_{\Hg}v(\tau),v(\tau)\rangle_{\Tilde{H}^d}\right)d\tau\nonumber\\
&\ \ +2\int_{0}^{t}\left(\langle v(\tau)\cdot\nabla_{\Hg}v(\tau),u(\tau)\rangle_{\Tilde{H}^d}+\langle u(\tau)\cdot\nabla_{\Hg}u(\tau),v(\tau)\rangle_{\Tilde{H}^d}\right)d\tau\nonumber\\
&=\|w(0)\|_{\Tilde{H}^d}^{2}+2\int_{0}^{t}\left(\langle w(\tau)\cdot\nabla_{\Hg}u(\tau),w(\tau)\rangle_{\Tilde{H}^d}+\langle v(\tau)\cdot\nabla_{\Hg}w(\tau),w(\tau)\rangle_{\Tilde{H}^d}\right)d\tau\label{I 7/01/2023}.
\end{align}
We have to estimate
\begin{equation}
    \label{I 6/01/2023}
    \int_{0}^{t}\langle w(\tau)\cdot\nabla_{\Hg}u(\tau),w(\tau)\rangle_{\Tilde{H}^d}d\tau+\int_{0}^{t}\langle v(\tau)\cdot\nabla_{\Hg}w(\tau),w(\tau)\rangle_{\Tilde{H}^d}d\tau.
\end{equation}
By applying Lemma \ref{II 06 10 2022}, we have
$$
\left|\int_{0}^{t}\langle w(\tau)\cdot\nabla_{\Hg}u(\tau),w(\tau)\rangle_{\Tilde{H}^d}d\tau\right|\lesssim \int_{0}^{t}\|w(\tau)\|_{\Tilde{H}^d}\|\nabla_{\Hg}w(\tau)\|_{\Tilde{H}^d}\|\nabla_{\Hg}u(\tau)\|_{\Tilde{H}^d}d\tau.
$$
 In order to recover the estimate on the second term of \eqref{I 6/01/2023}, let us remark that, if $d$ is even, then
 $$
 (-\Tilde{\Laplace}_{\Hg})^{d}=\mathcal{Z}^2,\ \text{ where } \mathcal{Z}:=(\Tilde{\Laplace}_{\Hg})^{\frac{d}{2}}=
 \sum_{i\in[\![1,2d]\!]^{\frac{d}{2}}}\prod_{j=1}^{\frac{d}{2}}\Tilde{P}_{i_j}^{2},
 $$
and  if $d$ is odd, then we have
 $$
 (-\Tilde{\Laplace}_{\Hg})^{d}=-\sum_{h=1}^{2d}(\mathcal{M}\Tilde{P}_h)(\Tilde{P}_{h}\mathcal{M}),\ \text{ where } \mathcal{M}:=(\Tilde{\Laplace}_{\Hg})^{\frac
{d-1}{2}}=\sum_{i\in[\![1,2d]\!]^{\frac{d-1}{2}}}\prod_{j=1}^{\frac{d-1}{2}}\Tilde{P}_{i_j}^{2}.
 $$
Note that $\mathcal{Z}$ and $\mathcal{M}$ are (unbounded) self-adjoint operators in $L^2(\Hg^d)$. At first, let us assume that $d$ is even. Then, we have
$$
\langle v\cdot\nabla_{\Hg}w,w\rangle_{\Tilde{H}^d}=\langle \mathcal{Z}(v\cdot\nabla_{\Hg}w),\mathcal{Z}w\rangle_{L^2}=\langle v\cdot\nabla_{\Hg}\mathcal{Z}w,\mathcal{Z}w\rangle_{L^2}+\langle [\mathcal{Z},v\cdot\nabla_{\Hg}]w,\mathcal{Z}w\rangle_{L^2}.
$$
However, if $a$ and $b$ are two smooth horizontal vector fields, such that $\divergence_{\Hg}(a)=0$, then 
$$
\langle a\cdot\nabla_{\Hg}b,b \rangle_{L^2}=0.
$$
Thus, we have
$$
\langle v\cdot\nabla_{\Hg}w,w\rangle_{\Tilde{H}^d}=\langle [\mathcal{Z},v\cdot\nabla_{\Hg}]w,\mathcal{Z}w\rangle_{L^2}.
$$
Let us now assume that $d$ is odd. Similarly to the case where $d$ is even, we get
$$
\langle v\cdot\nabla_{\Hg}w,w\rangle_{\Tilde{H}^d}=\sum_{h=1}^{2d}\langle [\Tilde{P}_h\mathcal{M},v\cdot\nabla_{\Hg}]w,\Tilde{P}_{h}\mathcal{M}w\rangle_{L^2}.
$$
Furthermore, let us note that $\mathcal{Z}$ and $\Tilde{P}_h\mathcal{M}$ with $h\in[\![1,2d]\!]$ commute with $P_j$, for any $j\in[\![1,2d]\!]$, and are both a sum of terms of the form $\Tilde{P}^{\gamma}$ where $\gamma$ belongs to $[\![1,2d]\!]^{d}$. Then, in order to get an estimate on $\langle v\cdot\nabla_{\Hg}w,w\rangle_{\Tilde{H}^d}$, we shall use the following lemma.
\begin{lem}
\label{II 06 10 2022 version gain de derive r et l sur a (dans le produit)}
There exists a constant $C>0$ such that for any smooth enough horizontal vector fields $a$, $b$ and $c$ with $\divergence_{\Hg}(a)=0$ and for all $\gamma\in[\![1,2d]\!]^{d}$, we have
$$
|\langle [\Tilde{P}^{\gamma},a\cdot\nabla_{\Hg}]b,c\rangle_{L^2}|\leq C\|\nabla_{\Hg}a\|_{\Tilde{H}^{d}}\|b\|_{\Tilde{H}^d}\|\nabla_{\Hg}c\|_{L^2}.
$$
\end{lem}
We will give the proof of this lemma after the proof of Theorem \ref{Stabilite dans Tilde H}.\\
By using Lemma \ref{II 06 10 2022 version gain de derive r et l sur a (dans le produit)}, we deduce that 
$$
\left|\int_{0}^{t}\langle v(\tau)\cdot\nabla_{\Hg}w(\tau),w(\tau)\rangle_{\Tilde{H}^d}d\tau\right|\lesssim \int_{0}^{t}\|\nabla_{\Hg}v(\tau)\|_{\Tilde{H}^d}\|w(\tau)\|_{\Tilde{H}^d}\|\nabla_{\Hg}w(\tau)\|_{\Tilde{H}^d}d\tau.
$$
and
\begin{align*}
 \bigg|\int_{0}^{t}\langle w(\tau)\cdot\nabla_{\Hg}u(\tau),w(\tau)\rangle_{\Tilde{H}^d}&+\langle v(\tau)\cdot\nabla_{\Hg}w(\tau),w(\tau)\rangle_{\Tilde{H}^d}d\tau\bigg|\\
 &\lesssim \int_{0}^{t}\|w(\tau)\|_{\Tilde{H}^d}\|\nabla_{\Hg}w(\tau)\|_{\Tilde{H}^d}\left(\|\nabla_{\Hg}u(\tau)\|_{\Tilde{H}^d}+\|\nabla_{\Hg}v(\tau)\|_{\Tilde{H}^d}\right)d\tau.
\end{align*}
Thus, from \eqref{I 7/01/2023} and the Young estimate, it follows that
$$
e(t)\leq \|w(0)\|_{\Tilde{H}^d}^{2}+C_{\star}\int_{0}^{t}\|w(\tau)\|_{\Tilde{H}^d}^{2}\left(\|\nabla_{\Hg}u(\tau)\|_{\Tilde{H}^d}^{2}+\|\nabla_{\Hg}v(\tau)\|_{\Tilde{H}^d}^{2}\right)d\tau+\int_{0}^{t}\|\nabla_{\Hg}w(\tau)\|_{\Tilde{H}^d}^{2}d\tau,
$$
for some positive constant $C_{\star}$, which is independent of $t$, $u$ and $v$. We conclude that the following estimate
\begin{align*}
\|w(t)\|_{\Tilde{H}^d}^{2}+\int_{0}^{t}\|\nabla_{\Hg}&w(\tau)\|_{\Tilde{H}^d}^{2}d\tau\\
&\leq \|w(0)\|_{\Tilde{H}^d}^{2}+C_{\star}\int_{0}^{t}\|w(\tau)\|_{\Tilde{H}^d}^{2}\left(\|\nabla_{\Hg}u(\tau)\|_{\Tilde{H}^d}^{2}+\|\nabla_{\Hg}v(\tau)\|_{\Tilde{H}^d}^{2}\right)d\tau,
\end{align*}
holds for any $t\in[0,T]$. Finally, Theorem \ref{Stabilite dans Tilde H} follows from the Gronwall estimate.
\end{proof}
\begin{proof}[Proof of Lemma \ref{II 06 10 2022 version gain de derive r et l sur a (dans le produit)}]
Let $\gamma\in[\![1,2d]\!]^d$ and $a$, $b$ and $c$, three smooth enough horizontal vector fields on $\Hg^d$ with $\divergence_{\Hg}(a)=0$. Then, according to \eqref{P Tilde P cummute}, $[\Tilde{P}^{\gamma},a\cdot\nabla_{\Hg}]b$ is the sum of terms of the form
$$
\Tilde{P}^{\alpha}a\cdot\nabla_{\Hg}\Tilde{P}^{\beta}b,
$$
with $\alpha\in[\![1,2d]\!]^\ell$ and $\beta\in[\![1,2d]\!]^{d-\ell}$ where  $\ell$ belongs to $[\![1,d]\!]$, according to $\divergence_{\Hg}(a)=0$, we have
$$
\Tilde{P}^{\alpha}a\cdot\nabla_{\Hg}\Tilde{P}^{\beta}b=\divergence_{\Hg}(\Tilde{P}^{\beta}b\otimes\Tilde{P}^{\alpha}a).
$$
Hence, we get
\begin{align*}
    \langle \Tilde{P}^{\alpha}a\cdot\nabla_{\Hg}\Tilde{P}^{\beta}b,c\rangle_{L^2}&=-\sum_{i,j\in[\![1,2d]\!]}\langle \Tilde{P}^{\alpha}a_i,\Tilde{P}^{\beta}b_jP_i c_j\rangle_{L^2} \\
    &\lesssim \|\Tilde{P}^{\alpha}a\|_{\dot{W}^{1,\frac{2Q}{Q-2(d-\ell)}}_{\Hg}}\|\Tilde{P}^{\beta}b\nabla_{\Hg}c\|_{\dot{W}^{-1,\frac{2Q}{Q+2(d-\ell)}}_{\Hg}}\\
    &\lesssim \|\nabla_{\Hg}\Tilde{P}^{\alpha}a\|_{L^{\frac{2Q}{Q-2(d-\ell)}}}\|\Tilde{P}^{\beta}b\nabla_{\Hg}c\|_{\dot{W}^{-1,\frac{2Q}{Q+2(d-\ell)}}_{\Hg}}.  
\end{align*}
Furthermore, since $\ell\neq 0$, we have the Sobolev embedding $L^{\frac{Q}{Q-\ell}}(\Hg^d)\hookrightarrow \dot{W}_{\Hg}^{-1,\frac{2Q}{Q+2(d-\ell)}}(\Hg^d)$ (see Proposition \ref{Propriete Sobolev sur Heisenberg}, Item 3), and then thanks to the H\"older estimate, we deduce that 
$$
\|\Tilde{P}^{\beta}b\nabla_{\Hg}c\|_{\dot{W}^{-1,\frac{2Q}{Q-2(d-\ell)}}}
\lesssim \|\Tilde{P}^{\beta}b\nabla_{\Hg}c\|_{L^{\frac{Q}{Q-\ell}}}\lesssim\|\Tilde{P}^{\beta}b\|_{L^{\frac{2Q}{Q-2\ell}}}\|\nabla_{\Hg}c\|_{L^2}.
$$
We conclude the proof of Lemma \ref{II 06 10 2022 version gain de derive r et l sur a (dans le produit)} by using the Sobolev embedding $\Tilde{H}^{d-\ell}(\Hg^d)\hookrightarrow L^{\frac{2Q}{Q-2(d-\ell)}}(\Hg^d)$ and $\Tilde{H}^{\ell}(\Hg^d)\hookrightarrow L^{\frac{2Q}{Q-2\ell}}(\Hg^d)$.
\end{proof}

\section{Smoothing effects: proof of Theorem \ref{th main:Regularity of the solution in TildeHd}}
\label{sec:Smoothing effects}
\subsection{Vertical regularity}
Let us begin by establishing some basic facts about the characterization of analyticity with respect to the vertical variable $s$ by using the Fourier transform on the Heisenberg group.
\subsubsection{Vertical analyticity and Fourier transform on $\Hg^d$}
Let $\zeta$ be a positive real number. Let $f\in\mathcal{S}(\Hg^d)$ so that the following map
$$
\mathsf{T}_{\zeta}:(n,m,\lambda)\in\Tilde{\Hg}^d\mapsto e^{\zeta|\lambda|}\mathcal{F}_{\Hg}(f)(n,m,\lambda),
$$
belongs to $L^{2}(\Tilde{\Hg}^d)$. When this definition makes sense, we write 
$$
e^{\zeta|D_s|}f:=\mathcal{F}_{\Hg}^{-1}(\mathsf{T}_{\zeta}f).
$$
As suggested by the Fourier inversion formula on $\Hg^d$, we aim to make the link between the Fourier multiplier $e^{\zeta|D_s|}$ on $\Hg^d$ and the Euclidean Fourier multiplier $e^{\zeta|\Lambda|}$ of symbol $e^{\zeta|\xi_s|}$ on $\R^{2d+1}$. It is well known that the operator $e^{\zeta|\Lambda|}$ acts on a space of analytic functions with respect to the vertical variable $s$.
\begin{rem}
\label{Gevrey Heisenberg to Gevrey euclidean}
Let $p\in(1,+\infty)$ and $\zeta\geq 0$. We define the two following spaces
 $$
 \mathcal{A}^{p,\zeta}_{\R}:=\enstq{f\in L^p(\Hg^d)}{e^{\zeta|\Lambda|}f\in L^p(\Hg^d)}\ \text{and}\ \ \mathcal{A}^{p,\zeta}_{\Hg}:=\enstq{f\in L^p(\Hg^d)}{e^{\zeta|D_s|}f\in L^p(\Hg^d)}.
 $$
 These two spaces are equal and we denote them by $\mathcal{A}^{p,\zeta}(\Hg^d)$. In fact, we have, for any $f\in \mathcal{A}^{p,\zeta}(\Hg^d)$,
\begin{equation}
\label{I 24 11 2022 19h}
e^{\zeta|D_s|}f=e^{\zeta|\Lambda|}f.
\end{equation}
\end{rem}

\begin{proof}[Proof of Remark \ref{Gevrey Heisenberg to Gevrey euclidean}]
With $W$ defined by the formula \eqref{Def-W-n-m-lambda},
if $f$ belongs to $L^1(\Hg^d)$, then
for any $(n,m,\lambda)\in\Tilde{\Hg}^d$, we have 
$$
\mathcal{F}_{\Hg}(f)(n,m,\lambda)=\int_{\R^{2d}}W(n,m,\lambda,Y)\mathcal{F}_{\R}(f(Y,\cdot))(\lambda)dY.
$$
Thus, if $f\in\mathcal{S}(\Hg^d)$, then $e^{-\zeta|\Lambda|}f\in L^1(\Hg^d)\cap L^2(\Hg^d)$ and we have
\begin{align*}
    \mathcal{F}_{\Hg}(e^{-\zeta|\Lambda|}f)(n,m,\lambda)=\int_{\R^{2d}}W(n,m,\lambda,Y)e^{-\zeta|\lambda|}\mathcal{F}_{\R}(f(Y,\cdot))(\lambda)dY=e^{-\zeta|\lambda|}\mathcal{F}_{\Hg}(f)(n,m,\lambda).
\end{align*}
We deduce that $e^{\zeta|D_s|}e^{-\zeta|\Lambda|}=\Id$ on $\mathcal{S}(\Hg^d)$. We conclude by using the density of $\mathcal{S}(\Hg^d)$ in $L^p(\Hg^d)$.
\end{proof}

\subsubsection{Analytic smoothing effects in the vertical variable}
Continuing our study of the Cauchy theory for solutions of \eqref{NSH}--\eqref{NSH initial condition} for initial data in the critical space $\Tilde{H}^d(\Hg^d)$, we now show that the solutions of \eqref{NSH} are instantaneously smoothed with respect to the variable $s$. More precisely, we aim to show the following theorem. 
\begin{theo}
\label{Analytic smoothing}
Let $\sigma\in(0,4d)$. There exists $\varepsilon_{\sigma}>0$ such that for any initial data $u_0\in \Tilde{H}^d(\Hg^d)$ satisfying $\divergence_{\Hg}(u_0)=0$ and $\|u_0\|_{\Tilde{H}^d}<\varepsilon_{\sigma}$, there exists a unique solution $u$ of \eqref{NSH} satisfying $u \in \mathcal{C}_b(\R_{+};\Tilde{H}^d)$ and $\nabla_{\Hg}u\in L^2(\R_{+};\Tilde{H}^{d})$ and the following energy estimate
\begin{equation}
\label{energie analytic}
\|e^{\sigma t|D_s|}u\|_{L^{\infty}(\Tilde{H}^d)}^{2}+\left(1-\frac{\sigma}{4d}\right)\|e^{\sigma t|D_s|}\nabla_{\Hg}u\|_{L^2(\Tilde{H}^d)}^{2}\leq \|u_0\|_{\Tilde{H}^d}^{2}.
\end{equation}
Moreover, there is a positive constant $A$, which is independent of $\sigma$, such that $\varepsilon_{\sigma}=A(1-\sigma/4d)$. 

In other words, for any $t>0$, the radius of analyticity (see \eqref{eq:Definition rayon d analyticite de Hd} for the definition) of $u(t)$ with respect to the variable $s$ is bounded from below by $\sigma t$. 
\end{theo}

\begin{proof}
The uniqueness of such solution follows from Theorem \ref{Stabilite dans Tilde H}. Let $k\in\N$, $\sigma\in(0,4d)$ and $u_0\in\Tilde{H}^d$. Let us consider the solution $u_k$ of the approximate problem \eqref{NSHk}-\eqref{NSHk initial condition}. In this proof the constant implied by $\lesssim$ is independent of $k$ and the time variable. By using the Plancherel theorem on $\Hg^d$, we deduce that
$$
\frac{1}{2}\frac{d}{dt}\|e^{\sigma t|D_s|}u_k\|_{\Tilde{H}^d}^{2}-\sigma\||D_s|^{\frac{1}{2}}e^{\sigma t|D_s|}u_k\|_{\Tilde{H}^d}^{2}=\langle e^{\sigma t|D_s|}\partial_t u_{k},e^{\sigma t|D_s|}u_{k}\rangle_{\Tilde{H}^d}.
$$
Besides, by computing the right-hand side of the above estimate thanks to \eqref{NSHk}, we obtain
\begin{align*}
\frac{1}{2}\frac{d}{dt}\|e^{\sigma t|D_s|}u_k\|_{\Tilde{H}^d}^{2}-\sigma\|&|D_s|^{\frac{1}{2}}e^{\sigma t|D_s|}u_k\|_{\Tilde{H}^d}^{2}\\
&=-\|e^{\sigma t|D_s|}\nabla_{\Hg}u_k\|_{\Tilde{H}^d}^{2}-\langle e^{\sigma t|D_s|}(u_k\cdot \nabla_{\Hg}\Jk u_k),e^{\sigma t|D_s|}\Jk u_{k}\rangle_{\Tilde{H}^d}.
\end{align*}
Due to the following estimate
$$
\||D_s|^{\frac{1}{2}}e^{\sigma t|D_s|}u_{k}\|_{\Tilde{H}^d}^{2}\leq \frac{1}{4d}\|(-\Laplace_{\Hg})^{\frac 12}e^{\sigma t|D_s|}u_k\|_{\Tilde{H}^d}^{2}=\frac{1}{4d}\|\nabla_{\Hg}e^{\sigma t|D_s|}u_k\|_{\Tilde{H}^d}^{2},
$$
we get
$$
\frac{1}{2}\frac{d}{dt}\|e^{\sigma t|D_s|}u_{k}\|_{\Tilde{H}^d}^{2}+\left(1-\frac{\sigma}{4d}\right)\|e^{\sigma t|D_s|}\nabla_{\Hg}u_{k}\|_{\Tilde{H}^d}^{2}\leq |\langle e^{\sigma t|D_s|}(u_k\cdot\nabla_{\Hg}\Jk u_k),e^{\sigma t|D_s|}\Jk u_k\rangle_{\Tilde{H}^d}|.
$$
Thus, we have
\begin{align}
     \|e^{\sigma \tau|D_s|}u_k&\|_{L^{\infty}_{t}(\Tilde{H}^d)}^{2}+2\left(1-\frac{\sigma}{4d}\right)\|e^{\sigma \tau|D_s|}\nabla_{\Hg}u_k\|_{L^{2}_{t}(\Tilde{H}^d)}^{2}\nonumber\\
     &\leq \|\TJk u_0\|_{\Tilde{H}^d}^{2}+\int_{0}^{t}|\langle e^{\sigma \tau|D_s|}(u_k(\tau)\cdot\nabla_{\Hg}\Jk u_k(\tau)),e^{\sigma \tau|D_s|}\Jk u_k(\tau)\rangle_{\Tilde{H}^d}|d\tau.\label{III 28 11 2022}
\end{align}
We now derive an estimate on the nonlinear term by using the following lemma that we shall temporarily assume (The proof will be given after the proof of Theorem \ref{Analytic smoothing}). 

\begin{lem}
\label{II 28 11 2022}
There exists a constant $C_d$ such that for any positive real numbers $\zeta$ and for every function $a$, $b$ and $c$ such that $e^{\zeta|D_s|}a$, $e^{\zeta|D_s|}b$ and $e^{\zeta|D_s|}\nabla_{\Hg}c$ belong to $\Tilde{H}^d(\Hg^d)$, we have 
$$
|\langle e^{\zeta |D_s|}(ab),e^{\zeta |D_s|}c\rangle_{\Tilde{H}^d}|\leq C_{d} \|e^{\zeta|D_s|}a\|_{\Tilde{H}^d}\|e^{\zeta|D_s|}b\|_{\Tilde{H}^d}\|e^{\zeta|D_s|}\nabla_{\Hg}c\|_{\Tilde{H}^d}.
$$
\end{lem}
Note that the constant $C_d$ is independent of $\zeta$, which is essential in order to get a suitable energy estimate \eqref{energie analytic} for global solutions.
\smallskip

We are now able to complete the proof of Theorem \ref{Analytic smoothing}. By applying Lemma \ref{II 28 11 2022}, Proposition \ref{Prop- I 25/11/2023}, Items 1 and 3, and using that $\TJk$ is bounded by $1$ in $\mc{L}(\Tilde{H}^d)$ (see Proposition \ref{Prop-Tilde Jk}, Item 1), we deduce from \eqref{III 28 11 2022} that for any $t\geq 0$,
\begin{align}
\|e^{\sigma \tau|D_s|}u_k\|_{L^{\infty}_{t}(\Tilde{H}^d)}^{2}+2\left(1-\frac{\sigma}{4d}\right)&\|e^{\sigma \tau|D_s|}\nabla_{\Hg}u_k\|_{L^{2}_{t}(\Tilde{H}^d)}^{2}\nonumber\\
&\leq \| u_0\|_{\Tilde{H}^d}^{2}+B\|e^{\sigma \tau|D_s|}u_k\|_{L^{\infty}_{t}(\Tilde{H}^d)}\|e^{\sigma \tau|D_s|}\nabla_{\Hg}u_k\|_{L^{2}_{t}(\Tilde{H}^d)}^{2},\label{I 9/01/2023}
\end{align}
where $B$ is a positive constant which is independent of $\sigma$. 

We can now perform the standard bootstrap argument. Let us define
$$
T_{\star}:=\sup\enstq{T\geq0}{\|e^{\sigma t|D_s|}u_k\|_{L^{\infty}_{T}(\Tilde{H}^d)}^{2}\leq 2\|u_0\|_{\Tilde{H}^d}^{2}}.
$$
We shall first show that $T_{\star}>0$. Since $u_k=\TJk u_k$, we have, for any $(n,m,\lambda)$ in $\N^d\times\N^d\times\R^{*}$ and $t\geq 0$,
$$
e^{\sigma |\lambda|t}|\mathcal{F}_{\Hg}(u_{k}(t))(n,m,\lambda)|\leq\mathbf{1}_{\{|\lambda|\leq 2^k\}}e^{\sigma |\lambda|t}|\mathcal{F}_{\Hg}(u_{k}(t))(n,m,\lambda)|\leq e^{\sigma 2^kt}|\mathcal{F}_{\Hg}(u_{k}(t))(n,m,\lambda)|.
$$
Then, according to the Plancherel formula on $\Hg^d$ (see Proposition \ref{inversion et Plancherel sur Heisenberg}), we deduce that
$$
\|e^{\sigma t|D_s|}u_{k}(t)\|_{\Tilde{H}^d}\leq e^{\sigma t2^k}\|u_{k}(t)\|_{\Tilde{H}^d}.
$$
Furthermore, by using the continuity of $t\mapsto u_{k}(t)$ from $\R_{+}$ into $\Tilde{H}^d(\Hg^d)$, we deduce that there exists $t_k>0$, such that $e^{\sigma 2^{k+1}t_k}\leq \frac{3}{2}$ and $\|u_{k}(t_k)\|_{\Tilde{H}^d}^{2}\leq \frac{4}{3}\|\TJk u_0\|_{\Tilde{H}^d}^{2}$. Then, since $\|\TJk\|_{\mathcal{L}(\Tilde{H}^d)}\leq 1$, we deduce that $T_{\star}\geq t_k>0$. Let us now assume that $T_{\star}<+\infty$ and show a contradiction for a suitable choice of $\varepsilon_\sigma$. Let $0<T<T_{\star}$. By definition of $T_{\star}$, we have 
$$
\|e^{\sigma t|D_s|}u_k\|_{L^{\infty}_{T}(\Tilde{H}^d)}\leq \sqrt{2}\|u_{0}\|_{\Tilde{H}^d}.
$$
Thus it follows from \eqref{I 9/01/2023} and the smallness condition on $\|u_0\|_{\Tilde{H}^d}$, that
\begin{align}
    \label{II 9/01/2023}
    \|e^{\sigma t|D_s|}u_k\|_{L^{\infty}_{T}(\Tilde{H}^d)}^{2}+2\left(1-\frac{\sigma}{4d}\right)\|e^{\sigma t|D_s|}&\nabla_{\Hg}u_k\|_{L^{2}_{T}(\Tilde{H}^d)}^{2}\nonumber\\
    &\leq \|u_0\|_{\Tilde{H}^d}^{2}+\varepsilon_{\sigma}B\sqrt{2}\|e^{\sigma t|D_s|}\nabla_{\Hg}u_k\|_{L_{T}^{2}(\Tilde{H}^d)}^{2}.
\end{align}
Then, if we chose $\varepsilon_{\sigma}:=\frac{1-\frac{\sigma}{4d}}{\sqrt{2}B}$, thanks to \eqref{II 9/01/2023}, we obtain
$$
\|e^{\sigma t|D_s|}u_k\|_{L^{\infty}_{T}(\Tilde{H}^d)}^{2}+\left(1-\frac{\sigma}{4d}\right)\|e^{\sigma t|D_s|}\nabla_{\Hg}u_k\|_{L^{2}_{T}(\Tilde{H}^d)}^{2}\leq \|u_0\|_{\Tilde{H}^d}^{2}.
$$
In view of the definition of $T_{\star}$, by the classical continuity argument this shows that $T_{\star}=+\infty$, and then the above estimate holds with $T = +\infty$ for any $k\in\N$.
Then we finish the proof by passing to the limit $k\to +\infty$ as in the proof of Theorem \ref{Global existence theorem in Tilde H1}.
\end{proof}
We now prove Lemma \ref{II 28 11 2022}. The proof relies on the multilinear Calderon-Zygmund theory (see \cite[Lemma 24.8, p. 252]{RecentdevelopmentsintheNavier-Stokesproblem}) to recover an estimate on the nonlinear term $(-\Tilde{\Laplace}_{\Hg})^{\frac{d}{2}}e^{\sigma t|D_s|}(ab)$. More precisely, in this paper we use an anisotropic version of \cite[Lemma 24.8, p. 252]{RecentdevelopmentsintheNavier-Stokesproblem}, that we will now establish.
We aim to study continuity properties of the following bilinear operators, defined for $\zeta \geq 0$, $A$ and $B$ in $\mathcal{S}(\Hg^d)$  and $(Y,s)\in \Hg^d$ by
\begin{align*}
M_{\zeta}(A,B)(Y,s)&:=e^{\zeta|D_s|}\left((e^{-\zeta|D_s|}A) (e^{-\zeta|D_s|}B)\right)(Y,s)\\
&=\frac{1}{4\pi^2}\int_{\R}\int_{\R}e^{is(\lambda+\mu)}e^{\zeta(|\lambda+\mu|-|\lambda|-|\mu|)}\mathcal{F}_{\R}(A(Y,\cdot))(\lambda)\mathcal{F}_{\R}(B(Y,\cdot))(\mu)d\lambda d\mu.
\end{align*}
\begin{lem}
\label{I 5/01/2023 Lemme}
Let $p$, $p_1$ and $p_2$ in $(1,+\infty)$ satisfying $1/p=1/p_1+1/p_2$. Then, there exists a constant $K$ such that for any $\zeta\geq 0$ and for all $A$ in $L^{p_1}(\Hg^d)$ and $B$ in $L^{p_2}(\Hg^d)$, we have $e^{\zeta|D_s|}\left((e^{-\zeta|D_s|}A) (e^{-\zeta|D_s|}B)\right)\in L^{p}(\Hg^d)$ and 
\begin{equation}
    \label{I 5/01/2023}
    \|e^{\zeta|D_s|}\left((e^{-\zeta|D_s|}A) (e^{-\zeta|D_s|}B)\right)\|_{L^p}\leq K\|A\|_{L^{p_1}}\|B\|_{L^{p_2}}.
\end{equation}
\end{lem}
Let us give the main arguments of the proof of Lemma \ref{I 5/01/2023 Lemme} (which is an adaptation of the proof of \cite[Lemma 24.8, p. 252]{RecentdevelopmentsintheNavier-Stokesproblem}) for the sake of clarity, since we use both Fourier transforms on $\R$ and on $\Hg^d$ in our analysis. 
\begin{proof}[Sketch of the proof of Lemma \ref{I 5/01/2023 Lemme}.] For any $f\in\mathcal{S}(\R)$ and $s\in\R$, we define
$$
K_1f(s):=\frac{1}{2\pi}\int_{0}^{+\infty}e^{is\lambda}\mathcal{F}_{\R}(f)(\lambda)d\lambda\ \ \ \text{and}\ \ \ K_{-1}f(s):=\frac{1}{2\pi}\int_{-\infty}^{0}e^{is\lambda}\mathcal{F}_{\R}(f)(\lambda)d\lambda.
$$
Moreover, for any $\zeta> 0$ and any $f\in\mathcal{S}(\R)$, we introduce
$$
L_{\zeta,1}f:=f\ \ \ \text{and}\ \ \ L_{\zeta,-1}f:=\frac{1}{2\pi}\int_{\R}e^{is\lambda}e^{-2\zeta|\lambda|}\mathcal{F}_{\R}(f)(\lambda)d\lambda.
$$
Finally, we set for any $\zeta> 0$ and $\alpha,\beta\in\{-1,1\}$
$$
\mathsf{K}_{\alpha}:= \Id_{\mathcal{S}^{'}(\R^{2d})}\otimes K_{\alpha}\ \ \text{and}\ \ \mathsf{Z}_{\zeta,\alpha,\beta}:=\mathsf{K}_{\alpha}\big(\Id_{\mathcal{S}^{'}(\R^{2d})}\otimes L_{\zeta,\alpha\beta}\big),
$$
where we denote by $\Id_{\mathcal{S}^{'}(\R^{2d})}$ the identity operator on $\mathcal{S}^{'}(\R^{2d})$. Then, we have 
$$
M_{\zeta}(A,B)=\sum_{(\gamma,\beta,\alpha)\in \{-1,1\}^3}\mathsf{K}_{\alpha}\left(\mathsf{Z}_{\zeta,\alpha,\beta}A\cdot \mathsf{Z}_{\zeta,\alpha,\gamma} B\right).
$$
Operators $\mathsf{K}_{\alpha}$ and $\mathsf{Z}_{\zeta,\alpha,\beta}$ are Fourier multipliers of order zero, hence they are bounded in any $L^p(\Hg^d)$ with $p\in(1,+\infty)$. The main point is the following: the family $\mathsf{Z}_{\zeta,\alpha,\beta}$ is uniformly bounded with respect to the parameter $\zeta>0$ in any $L^p(\Hg^d)$ with $p\in(1,+\infty)$. Indeed, the kernel of $L_{\zeta,-1}$ is $\frac{1}{\pi}\frac{2 \zeta}{s^2+4\zeta^2}:=k_{\zeta}(s)$ and we have $\|k_{\zeta}\|_{L^1(\R)}=\|k_1\|_{L^1(\R)}$, so that by Young's estimate with respect to the $s$ variable, we deduce that for any $p \in (1, +\infty)$ and for any $f\in L^p(\Hg^d)$, 
$$
\|
(\Id_{\mathcal{S}^{'}(\R^{2d})}\otimes L_{\zeta,-1})f\|_{L^p(\Hg^d)}\leq\|k_1\|_{L^1(\R)}\|f\|_{L^p(\Hg^d)}.
$$
Then Lemma \ref{I 5/01/2023 Lemme} follows from the H\"older estimate.
\end{proof}
We are now able to prove Lemma \ref{II 28 11 2022}.
\begin{proof}[Proof of Lemma \ref{II 28 11 2022}]
Let us begin by setting $A:=e^{\zeta|D_s|}a$ and $B:=e^{\zeta|D_s|}b$ so that $A$ and $B$ belong to $\Tilde{H}^d(\Hg^d)$. Assume that $A$ and $B$ belong to $\mathcal{S}(\Hg^d)$. Then, $(e^{-\zeta|D_s|}A)(e^{-\zeta|D_s|}B)$ belongs to $\mathcal{S}(\Hg^d)$ and for any $\gamma\in[\![1,2d]\!]^d$, by using the Leibniz formula, we get
\begin{equation}
\label{II 5/01/2023}
\left|e^{\zeta|D_s|}\Tilde{P}^{\gamma}\left((e^{-\zeta|D_s|}A)(e^{-\zeta|D_s|}B)\right)\right|\lesssim\sum_{\ell=0}^{d}\! \! \sum_{\substack{\alpha\in[\![1,2d]\!]^\ell\\ \beta\in [\![1,2d]\!]^{d-\ell}
}}\left|e^{\zeta|D_s|}\left(e^{-\zeta|D_s|}(\Tilde{P}^{\alpha}A)e^{-\zeta|D_s|}(\Tilde{P}^{\beta}B)\right)\right|.
\end{equation}

If we pick $\ell\in[\![0,d]\!]$, $\alpha\in[\![1,2d]\!]^{\ell}$, $\beta\in [\![1,2d]\!]^{d-\ell}$ and $\gamma\in[\![1,2d]\!]^{d}$, applying Lemma \ref{I 5/01/2023 Lemme}, we deduce that
 \begin{equation*}   \left|\!\left|e^{\zeta|D_s|}\left(e^{-\zeta|D_s|}(\Tilde{P}^{\alpha}A)e^{-\zeta|D_s|}(\Tilde{P}^{\beta}B)\right)\right|\!\right|_{L^{\frac{2Q}{Q+2}}}\leq C_{1}\|\Tilde{P}^{\alpha}A\|_{L^{\frac{2Q}{Q-2(d-\ell)}}}\|\Tilde{P}^{\beta}B\|_{L^{\frac{2Q}{Q-2\ell}}},
 \end{equation*}
 where $C_1$ is a positive constant which is independent of $\zeta$. Applying the Sobolev embeddings $\Tilde{H}^{d-\ell}(\Hg^d)\hookrightarrow L^{\frac{2Q}{Q-2(d-\ell)}}(\Hg^d)$ and $\Tilde{H}^{\ell}(\Hg^d)\hookrightarrow L^{\frac{2Q}{Q-2\ell}}(\Hg^d)$, and using $\alpha\in[\![1,2d]\!]^{\ell}$ and $\beta\in[\![1,2d]\!]^{d-\ell}$, we obtain that
 \begin{align*}
\|\Tilde{P}^{\alpha}A\|_{L^{\frac{2Q}{Q-2(d-\ell)}}}\lesssim\|A\|_{\Tilde{H}^d}\ \ \text{and}\ \ \|\Tilde{P}^{\beta}B\|_{L^{\frac{2Q}{Q-2\ell}}}\lesssim\|B\|_{\Tilde{H}^d}.
 \end{align*}
In view of \eqref{II 5/01/2023}, it follows that
\begin{equation}
    \label{III 5/01/2023}
    \left|\!\left|(-\Tilde{\Laplace}_{\Hg})^{\frac{d}{2}}e^{\zeta|D_s|}\left(e^{-\zeta|D_s|}Ae^{-\zeta|D_s|}B\right)\right|\!\right|_{L^{\frac{2Q}{Q+2}}}\leq C_2\|A\|_{\Tilde{H}^d}\|B\|_{\Tilde{H}^d},
\end{equation}
where $C_2$ depend only on the constant $C_1$, $d$ and the constants of the Sobolev embeddings. By density of $\mathcal{S}(\Hg^d)$ in $\Tilde{H}^d(\Hg^d)$, thanks to \eqref{III 5/01/2023}, we conclude that \eqref{III 5/01/2023} holds for any $A$ and $B$ in $\Tilde{H}^d(\Hg^d)$. In addition, by applying the Hölder estimate, we deduce that
$$
\left|\langle e^{\zeta|D_s|}(ab),e^{\zeta|D_s|}c\rangle_{\Tilde{H}^d}\right|\leq C_2\|e^{\zeta|D_s|}a\|_{\Tilde{H}^d}\|e^{\zeta|D_s|}b\|_{\Tilde{H}^d}\|(-\Tilde{\Laplace}_{\Hg})^{\frac{d}{2}}e^{\zeta|D_s|}c\|_{L^{\frac{2Q}{Q-2}}}.
$$
We conclude by using the Sobolev embedding $\dot{H}^1(\Hg^d)\hookrightarrow L^{\frac{2Q}{Q-2}}(\Hg^d)$.
\end{proof}
We have proved that the analytic smoothing effect holds at least in the vertical variable. In the next section, we will improve the natural smoothing effect from the left-invariant sub-Laplacian provided by the control of the horizontal gradient of the solution in Theorem \ref{Global existence theorem in Tilde H1}.
\subsection{Horizontal regularity}
We now analyze the smoothness of the solutions $u$ of \eqref{NSH}--\eqref{NSH initial condition} given by Theorem \ref{Analytic smoothing} with respect to the horizontal vector fields $(P_j)_{ j \in [\![1,2d]\!]}$, thus with respect to all variables. 
\begin{coro}
\label{strong solution}
Let $u$ be a solution of \eqref{NSH} constructed in Theorem \ref{Analytic smoothing}. Then for any nonnegative integers $\ell$, $\alpha\in[\![1,2d]\!]^\ell$, $\beta\in\N$ and $t>0$, we have 
\begin{equation}
\label{eq:global regularity sharp results}
P^{\alpha}\partial_{s}^{\beta}u(t)\in \Tilde{H}^d.
\end{equation}
Furthermore, $u$ and $p$ belong to $\mathcal{C}^{\infty}((0,+\infty)\times\R^{2d+1})$. Accordingly, such solution $u$ of \eqref{NSH}-\eqref{NSH initial condition} is a strong solution of \eqref{NSH}.
\end{coro}
For any $\sigma\in(0,4d)$, we set 
$$
\sigma_r:=\frac{\sigma}{2^{r}},\ \ \ r\in \N.
$$
For any $\ell\in\R$ and $\ell'\in[0,Q/2)$, we introduce the space $H_{\ell,\ell'}(\Hg^d)$ by setting
$$
H_{\ell,\ell'}(\Hg^d):=\enstq{f\in\mathcal{S}'(\Hg^d)}{(\Id-\Laplace_{\Hg})^{\frac{\ell}{2}}(-\Tilde{\Laplace}_{\Hg})^{\frac{\ell'}{2}}f\in L^2(\Hg^d)},
$$
and we denote by $\|\cdot\|_{H_{\ell,\ell'}}$ the corresponding norms and by $\langle\cdot,\cdot\rangle_{H_{\ell,\ell'}}$ the corresponding  scalar product.

\begin{lem}
\label{Lemme 2 08/08/2023} 
 Let $u$ be a solution of \eqref{NSH}-\eqref{NSH initial condition} constructed in Theorem \ref{Analytic smoothing} with corresponding radius of analyticity bounded from below by $t \mapsto \sigma t$ for $\sigma\in(0,4d)$, and $(u_k)$ the sequence of solutions of the approximate problems \eqref{NSHk}-\eqref{NSHk initial condition} which converges to $u$. Let $T$ and $T_\star$ be two real numbers such that $T>T_{\star}>0$ and $r\in\N$. Assume that there exist $T_r\in(0,T_{\star})$ and a positive constant $C_r$ such that for any $k\in\N$, we have
\begin{equation}
    \label{I 31/07/2023}
    \|e^{\sigma_{r}t|D_s|}u_{k}\|^{2}_{L^{\infty}([T_r,T];H_{r,d})}+\int_{T_r}^{T}\|e^{\sigma_{r}t|D_s|}u_{k}(t)\|_{H_{r+1,d}}^{2}dt\leq C_{r}.
\end{equation}
Then there exist $T_{r+1}\in[T_r,T_{\star})$ and a positive constant $C_{r+1}$ such that for any $k\in\N$,
\begin{equation}
    \label{II 31/07/2023}
    \|e^{\sigma_{r+1}t|D_s|}u_{k}\|_{L^{\infty}([T_{r+1},T];H_{r+1,d})}^{2}+\int_{T_{r+1}}^{T}\|e^{\sigma_{r+1}t|D|}u_{k}(t)\|_{H_{r+2,d}}^{2}dt\leq C_{r+1}.
\end{equation}
\end{lem}

\begin{proof}[Proof of Lemma \ref{Lemme 2 08/08/2023}] 
\textit{\underline{Energy estimate}:} Thanks to \eqref{I 31/07/2023}, we deduce that
\begin{equation}
    \label{III 31/07/2023}
    \forall k\in\N,\ \exists t_k\in \left(T_r, \frac{T_r+T_\star}{2}\right)\ \ \ \|e^{\sigma_{r+1}t_k|D_s|}u_{k}(t_k)\|_{H_{r+1,d}}^{2}\leq \frac{4C_r}{(T_\star-T_r)}.
\end{equation}
Indeed, if this is not the case, there exists $k\in\N$ such that for any $t\in \left(T_r,(T_\star+T_r)/2\right)$, we have 
$$
\|e^{\sigma_{r+1}t |D_s|}u_{k}(t)\|_{H_{r+1,d}}^{2}> \frac{4C_r}{(T_\star-T_r)}.
$$
This would entail, since $(T_\star+T_r)/2<T_\star<T$, that 
$$
\int_{T_r}^{T}\|e^{\sigma_{r+1}t_k|D_s|}u_{k}(t)\|_{H_{r+1,d}}^{2}dt>\frac{4C_r}{(T_\star-T_r)}\times \frac{T_\star-T_r}{2}=2C_r,
$$
which would contradict \eqref{I 31/07/2023}.

Now, let us consider two integers $a$ and $k$. Recall that $\Ja^2=\Ja$ and that $\Ja$ is self-adjoint (see Proposition \ref{Prop- I 25/11/2023} Items 5 and 7). Thanks to the commutation properties of $\Ja$ (see Proposition \ref{Prop- I 25/11/2023} Items 3 and 4) and using that $\Ja$ is bounded by $1$ in $\mathcal{L}(L^2)$, for any $t\in\left((T_r+T_\star)/2,T_\star\right)$, we have
\begin{align*}
\|e^{\sigma_{r+1}t|D_s|}\Ja u_{k}(t)&\|_{H_{r+1,d}}^{2}+2\left(1-\frac{\sigma_{r+1}}{4d}\right)\int_{t_k}^{t}\|e^{\sigma_{r+1}\tau|D_s|}(-\Laplace_{\Hg})^{\frac 12}\Ja u_{k}(\tau)\|_{H_{r+1,d}}^{2}d\tau\nonumber\\   &\leq\|e^{\sigma_{r+1}t_k|D_s|}u_{k}(t_k)\|_{H_{r+1,d}}^{2}\\
&\ \ \ +\int_{t_k}^{t}|\langle
e^{\sigma_{r+1}\tau|D_s|}(\Id-\mathbb{P})\circ(-\Laplace_{\Hg})u_{k}(\tau),e^{\sigma_{r+1}\tau|D_s|}\Ja u_{k}(\tau)\rangle_{H_{r+1,d}}|d\tau\\
&\ \ \ +\int_{t_k}^{t}|\langle e^{\sigma_{r+1}\tau|D_s|}\mathbb{P}\circ\Jk(u_{k}(\tau)\cdot\nabla_{\Hg}\Jk u_{k}(\tau)),e^{\sigma_{r+1}\tau|D_s|}\Ja u_{k}(\tau)\rangle_{H_{r+1,d}}|d\tau.
\end{align*}
Let us take $t\in[t_k,T]$. According to \eqref{III 31/07/2023}, we obtain that
\begin{align}
    \|e^{\sigma_{r+1}t|D_s|}\Ja u_{k}(t)\|_{H_{r+1,d}}^{2}+2&\left(1-\frac{\sigma_{r+1}}{4d}\right)\int_{t_k}^{t}\|e^{\sigma_{r+1}\tau|D_s|}(-\Laplace_{\Hg})^{\frac 12}\Ja u_{k}(\tau)\|_{H_{r+1,d}}^{2}d\tau\nonumber\\
    &\leq \frac{4C_r}{(T_\star-T_r)}+I_{1}^{a,k}(t)+I_{2}^{a,k}(t),\label{IX 01/08/2023}
\end{align}
where
$$
I_{1}^{a,k}(t):=\int_{t_k}^{t}|\langle
e^{\sigma_{r+1}\tau|D_s|}(\Id-\mathbb{P})\circ(-\Laplace_{\Hg})u_{k}(\tau),e^{\sigma_{r+1}\tau|D_s|}\Ja u_{k}(\tau)\rangle_{H_{r+1,d}}|d\tau
$$
and
$$
I_{2}^{a,k}(t):=\int_{t_k}^{t}|\langle e^{\sigma_{r+1}\tau|D_s|}\mathbb{P}\circ\Jk(u_{k}(\tau)\cdot\nabla_{\Hg}\Jk u_{k}(\tau)),e^{\sigma_{r+1}\tau|D_s|}\Ja u_{k}(\tau)\rangle_{H_{r+1,d}}|d\tau.
$$

\textit{\underline{Estimate on $I_{1}^{a,k}(t)$.}} Let us recall that the Leray projector $\mathbb{P}$ (defined in \eqref{def: Leray projector}) is of order $0$. Accordingly, $(\Id-\mathbb{P})\circ(-\Laplace_{\Hg})$ is of order $2$ which implies that to estimate $I_{1}^{a,k}(t)$, we need to control $e^{\sigma_{r+1}\tau|D_s|}u_{k}$ at least in $L^2((t_k,T); H_{r+2,d})$. Unfortunately, the condition \eqref{I 31/07/2023} gives only a control of $e^{\sigma_{r}\tau|D_s|}u_{k}$ in $L^{2}((t_k,T);H_{r+1,d})$. We thus need to gain one additional regularity level with respect to the horizontal derivative, to the price of possibly losing some regularity in the vertical variable, since $\sigma_r > \sigma_{r+1}$. The key idea to overcome this difficulty is to take advantage of the equation $\divergence_{\Hg}(u_{k})=0$, so that we can use Lemma \ref{I 28/08/2023} and the identity $(\Id-\mathbb{P})\circ(-\Laplace_{\Hg})u_{k}=\Pi_{\Hg}\circ\partial_s u_{k}$, where $\Pi_{\Hg}$ is of order $0$ and commutes with $|D_s|$.
This yields
\begin{align}
    I_{1}^{a,k}(t)&=\int_{t_k}^{t}|\langle
e^{\sigma_{r+1}\tau|D_s|}\partial_s\circ\Pi_{\Hg} u_{k}(\tau),e^{\sigma_{r+1}\tau|D_s|}\Ja u_{k}(\tau)\rangle_{H_{r+1,d}}|d\tau\nonumber\\
&\leq \int_{t_k}^{t}\|
e^{\sigma_{r+1}\tau|D_s|}|D_s|^{\frac 12}\Pi_{\Hg} u_{k}(\tau)\|_{H_{r+1,d}}\|e^{\sigma_{r+1}\tau|D_s|}|D_s|^{\frac 12}\Ja u_{k}(\tau)\|_{H_{r+1,d}}d\tau\nonumber\\
&\lesssim \int_{t_k}^{T}\|e^{\sigma_{r+1}\tau|D_s|}|D_s|^{\frac 12}u_{k}(\tau)\|_{H_{r+1,d}}^{2}d\tau.\label{IV 31/07/2023}
\end{align}
In order to bound $e^{\sigma_{r+1}\tau|D_s|}|D_s|^{\frac 12}u_{k}$ in $L^2((t_k,T);H_{r+1,d})$ we use Hypothesis \eqref{I 31/07/2023}: Indeed, for any $\lambda\in\R^{*}$ and $\tau$ larger than $t_k$, we have
$$
e^{2\sigma_{r+1}\tau|\lambda|}|\lambda|\leq \frac{e^{2\sigma_{r+1}\tau|\lambda|}(2\sigma_{r+1}\tau|\lambda|)}{2\sigma_{r+1}t_k}\leq \frac{e^{4\sigma_{r+1}\tau|\lambda|}}{2\sigma_{r+1}t_k}=\frac{e^{2\sigma_{r}\tau|\lambda|}}{\sigma_{r}t_k} \leq \frac{e^{2\sigma_{r}\tau|\lambda|}}{\sigma_{r}T_r} .
$$
Thus, using the Plancherel formula on $\Hg^{d}$, we get
\begin{equation*}
\|e^{\sigma_{r+1}\tau|D_s|}|D_s|^{\frac{1}{2}}u_{k}(\tau)\|^{2}_{H_{r+1,d}}\leq\frac{1}{\sigma_{r+1}T_r}\|e^{\sigma_{r}\tau|D_s|}u_{k}(\tau)\|_{H_{r+1,d}}^{2}.
\end{equation*}
Then, it follows from \eqref{IV 31/07/2023} and \eqref{I 31/07/2023} that there exists a constant $C_{I_1}$ independent of $k$ such that for any $t\in[t_k, T]$,
\begin{equation}
\label{V 31/07/2023}
I_{1}^{a,k}(t)\lesssim\int_{T_{r}}^{T}\|e^{\sigma_{r}\tau|D_s|}u_{k}(\tau)\|_{H_{r+1,d}}^{2}d\tau\leq C_{I_1}.
\end{equation}

\underline{\textit{Regularity of $u_k$.}} Let us first remark that, according to the properties of $\Ja$, we have
\begin{align*}
    I_{2}^{a,k}(t)&\leq\int_{0}^{T}|\langle e^{\sigma_{r+1}\tau|D_s|}(\Id-\Laplace_{\Hg})^{r+1}\circ\mathbb{P}\circ\Jk(u_{k}(\tau)\cdot\nabla_{\Hg}\Jk u_{k}(\tau)),e^{\sigma_{r+1}\tau|D_s|}\Ja u_{k}(\tau)\rangle_{\Tilde{H}^{d}}|d\tau\\
    &\leq T^{\frac 12}\|e^{\sigma_{r+1}\tau|D_s|}(\Id-\Laplace_{\Hg})^{r+1}\circ\mbb{P}\circ\Jk(u_k\cdot\nabla_{\Hg}\Jk u_k)\|_{L^{2}_{T}(\Tilde{H}^d)} \|e^{\sigma_{r+1}\tau|D_s|}u_k\|_{L^{\infty}(\Tilde{H}^d)}.
\end{align*}
Since $\mbb{P}$ commutes with $(-\Tilde{\Laplace}_{\Hg})^{\frac d2}$ and $e^{\sigma_{r+1}\tau|D_s|}$, and belongs to $\mc{L}(H^{r+1})$, we deduce that
\begin{align*}
\|e^{\sigma_{r+1}\tau|D_s|}(\Id-\Laplace_{\Hg})^{r+1}\circ\mbb{P}\circ\Jk(&u_k\cdot\nabla_{\Hg}\Jk u_k)\|_{L^{2}_{T}(\Tilde{H}^d)}\\
&\leq \|\mbb{P}\|_{\mc{L}(H^{r+1})}\|e^{\sigma_{r+1}\tau|D_s|}(-\Tilde{\Laplace}_{\Hg})^{\frac d2}\Jk(u_k\cdot\nabla_{\Hg}\Jk u_k)\|_{L^{2}_{T}(H^{r+1})}.
\end{align*}
Furthermore, it follows from Proposition \ref{Prop- I 25/11/2023}, Item $1$, that $\Jk e^{\sigma_{r+1}\tau|D_s|}(-\Tilde{\Laplace}_{\Hg})^{\frac d2}(u_k\cdot\nabla_{\Hg}\Jk u_k)$ belongs to $L^{2}_{T}(H^{r+1})$. Then, for any $k\in\N$, there exists a constant $C_k$ such that
$$
\sup_{a\in\N, t\in[t_k, T]}\lbrace I^{a,k}_{2}(t)\rbrace\leq C_k.
$$
Thus, according to \eqref{I 31/07/2023} and \eqref{V 31/07/2023}, it follows from \eqref{IX 01/08/2023} that the two quantities
$$
\sup_{a\in\N}\{\|e^{\sigma_{r+1}t|D_s|}\Ja u_{k}\|_{L^{\infty}(t_k,T;H_{r+1,d})}^{2}
\}
$$
and
$$
\sup_{a\in\N}\{\|e^{\sigma_{r+1}t|D_s|}(-\Laplace_{\Hg})^{\frac 12}\Ja u_{k}\|_{L^{2}(t_k, T;H_{r+1,d})}^{2}\}
$$
are finite. Let us set 
$$
F_k(t):=e^{\sigma_{r+1}t|D_s|}(\Id-\Laplace_{\Hg})^{r+1}(-\Tilde{\Laplace}_{\Hg})^{\frac d2}u_k(t),
$$
with $t\in[t_k, T]$. By using that $\Ja$ commute with $e^{\sigma_{r+1}t|D_s|}(\Id-\Laplace_{\Hg})^{r+1}(-\Tilde{\Laplace}_{\Hg})^{\frac d2}$ and $(-\Laplace_{\Hg})^{\frac 12}$ and a weak compactness argument, we deduce that, up to extract a subsequence and by identifying the limits in $\mathcal{D}'((t_k,T)\times\Hg^d)$, the sequences $(\Ja F_k)_{a\in\N}$ and $(\Ja(-\Laplace_{\Hg})^{\frac 12}F_k)_{a\in\N}$ converge, respectively, in $L^{\infty}(t_k,T;L^2)$ for the weak-$\star$ topology to $F_k$, and in $L^2(t_k,T;L^2)$ for the weak topology to $(-\Laplace_{\Hg})^{\frac 12}F_k$. Moreover, it follows from $\eqref{III 31/07/2023}$ and $\eqref{V 31/07/2023}$ that for any $t\in (t_k, T)$, we have
\begin{align*}
\liminf_{a\rightarrow+\infty}\|e^{\sigma_{r+1}t|D_s|}\Ja u_{k}&\|_{L^{\infty}(t_k,t;H_{r+1,d})}^{2}\\
&+2\left(1-\frac{\sigma_{r+1}}{4d}\right)\liminf_{a\rightarrow+\infty}\|e^{\sigma_{r+1}t|D_s|}(-\Laplace_{\Hg})^{\frac 12}\Ja u_{k}\|_{L^{2}(t_k, t;H_{r+1,d})}^{2}\\
&\leq \frac{4C_r}{(T_\star-T_r)}+C_{I_1}+\sup_{a\in\N}\{I^{a,k}_{2}(t)\}.
\end{align*}
Thus according to the properties of the weak and the weak-$\star$ convergence, we deduce that for any $t\in(t_k,T)$, we have
\begin{align}
\|e^{\sigma_{r+1}t|D_s|} u_{k}(t)\|_{H_{r+1,d}}^{2}+2&\left(1-\frac{\sigma_{r+1}}{4d}\right)\int_{t_k}^{t}\|e^{\sigma_{r+1}\tau|D_s|}(-\Laplace_{\Hg})^{\frac 12} u_{k}(\tau)\|_{H_{r+1,d}}^{2}d\tau\nonumber\\
&\leq \frac{4C_r}{(T_\star-T_r)}+C_{I_1}+\sup_{a\in\N}\lbrace I^{a,k}_{2}(t)\rbrace.\label{I 15/02/2024}  
\end{align}
We are thus reduced to estimating $I^{a,k}_{2}(t)$ uniformly with respect to the parameters $a\in\N$ and $k \in \N$.
\\

\underline{\textit{Estimate on $I_{2}^{a,k}(t)$.}} Let us begin by remarking that $\divergence_{\Hg}(u_{k})=0$ implies that
\begin{equation}
    \label{I 01/08/2023}    u_{k}\cdot\nabla_{\Hg}\Jk u_{k}=\divergence_{\Hg}(\Jk u_{k}\otimes u_{k}).
\end{equation}
Besides, let us write
\begin{align}
    (\Id-\Laplace_{\Hg})^{\frac{
r+1}{2}}\circ\mathbb{P}\circ\Jk\circ\divergence_{\Hg}&=(\Id-\Laplace_{\Hg})^{-\frac{r+1}{2}}\circ\Gamma_{k}\circ(\Id-\Laplace_{\Hg})^{r+1},\label{II 01/08/2023}
\end{align}
with
$$
\Gamma_{k}:=(\Id-\Laplace_{\Hg})^{r+1}\circ\mathbb{P}\circ\Jk\circ\divergence_{\Hg}\circ(\Id-\Laplace_{\Hg})^{-(r+1)}.
$$
Note that $\Gamma_{k}$ is an operator of order $1$ with respect to the left-invariant sub-Laplacian, which maps $2d\times 2d$ matrix value functions to horizontal vector fields, and thus its adjoint 
$$
\Gamma_{k}^{*}=-(\Id-\Laplace_{\Hg})^{-(r+1)}\circ\nabla_{\Hg}\circ\Jk\circ\mathbb{P}\circ(\Id-\Laplace_{\Hg})^{r+1},
$$
is also of order 1 with respect to the left-invariant sub-Laplacian and maps horizontal vector fields to $2d\times 2d$ matrix value functions. Then, using \eqref{I 01/08/2023}, \eqref{II 01/08/2023} and the binomial expansion on $(\Id-\Laplace_{\Hg})^{r+1}$, we get
\begin{align}
  |  \langle e^{\sigma_{r+1}\tau|D_s|}\mathbb{P}\circ\Jk(u_{k}\cdot\nabla_{\Hg} \Jk&u_{k}),e^{\sigma_{r+1}\tau|D_s|}\Ja u_{k}\rangle_{H_{r+1,d}}
  |
  \nonumber\\
&=|\langle e^{\sigma_{r+1}\tau|D_s|}\mathbb{P}\circ\Jk\circ\divergence_{\Hg}(\Jk u_{k}\otimes u_{k}),e^{\sigma_{r+1}\tau|D_s|}\Ja u_{k}\rangle_{H_{r+1,d}}|\nonumber\\
    &=|\langle e^{\sigma_{r+1}\tau|D_s|}(\Id-\Laplace_{\Hg})^{r+1}(\Jk u_{k}\otimes u_{k}),e^{\sigma_{r+1}\tau|D_s|}\Gamma_{k}^{*}\Ja u_{k}\rangle_{\Tilde{H}^d}|\nonumber\\
&\lesssim \sum_{\ell=0}^{r+1}\sum_{\gamma,\gamma'\in [\![1,2d]\!]^{\ell}}|\langle e^{\sigma_{r+1}\tau|D_s|}P^{\gamma}(\Jk u_{k}\otimes u_{k}),e^{\sigma_{r+1}\tau|D_s|}P^{\gamma'}\Gamma_{k}^{*}\Ja u_{k}\rangle_{\Tilde{H}^d}|.\nonumber
\end{align}
Furthermore, if we take $\ell\in[\![0,r+1]\!]$ and $\gamma,\gamma'\in[\![1,2d]\!]^\ell$, we have
\begin{align*}
    |\langle e^{\sigma_{r+1}\tau|D_s|}P^{\gamma}(&\Jk u_{k}\otimes u_{k}),e^{\sigma_{r+1}\tau|D_s|}P^{\gamma'}\Gamma_{k}^{*}\Ja u_{k}\rangle_{\Tilde{H}^d}|\\
    &\lesssim \|e^{\sigma_{r+1}\tau|D_s|}P^{\gamma}(u\otimes u)\|_{\Tilde{H}^d}\|e^{\sigma_{r+1}\tau|D_s|}\Gamma^{*}_k \Ja u_k\|_{H_{\ell,d}}\\
    &\lesssim \left(\sum_{\Tilde{\gamma}\in [\![1,2d]\!]^{d}}\|e^{\sigma_{r+1}\tau|D_s|}\Tilde{P}^{\Tilde{\gamma}}P^{\gamma}(\Jk u_{k}\otimes u_{k})\|_{L^2}\right)\|e^{\sigma_{r+1}\tau|D_s|}\Gamma_{k}^{*}\Ja u_{k}\|_{H_{r+1,d}}.
\end{align*}
On one hand, using that $\mbb{P}$ commutes with $(-\Tilde{\Laplace}_{\Hg})^{\frac d2}$, we have
\begin{align*}
\|e^{\sigma_{r+1}\tau|D_s|}\Gamma_{k}^{*}\Ja u_{k}\|_{H_{r+1,d}}&=\|\nabla_{\Hg}\circ\Jk\circ\mbb{P}\circ(\Id-\Laplace_{\Hg})^{r+1}e^{\sigma_{r+1}\tau|D_s|}\Ja u_k\|_{\Tilde{H}^d}\\
&=\|(-\Laplace_{\Hg})^{\frac 12}\circ\Jk\circ\mbb{P}\circ\Ja\circ(\Id-\Laplace_{\Hg})^{r+1}(-\Tilde{\Laplace}_{\Hg})^{\frac d2}e^{\sigma_{r+1}\tau|D_s|} u_k\|_{L^2}.
\end{align*} 
Then, using that $\Jk$ and $\Ja$ are bounded by $1$ in $\mathcal{L}(\dot{H}^1)$, we deduce that
\begin{align*}
\|e^{\sigma_{r+1}\tau|D_s|}\Gamma_{k}^{*}\Ja u_{k}\|_{H_{r+1,d}}&\leq\|\mbb{P}\|_{\mc{L}(\dot{H}^1)}\|(-\Laplace_{\Hg})^{\frac 12}(\Id-\Laplace_{\Hg})^{r+1}(-\Tilde{\Laplace}_{\Hg})^{\frac d2}e^{\sigma_{r+1}\tau|D_s|} u_k\|_{\Tilde{H}^d}\\
&=\|\mbb{P}\|_{\mc{L}(\dot{H}^1)}\|e^{\sigma_{r+1}\tau|D_s|}(-\Laplace_{\Hg})^{\frac 12}u_k\|_{H_{r+1,d}}.
\end{align*}
On the other hand, by applying the Leibniz formula for $P^{\gamma}$  with $\gamma\in[\![1,2d]\!]^{\ell}$ and for $\Tilde{P}^{\Tilde{\gamma}}$ with $\Tilde{\gamma}\in[\![1,2d]\!]^d$, we deduce that
$$
\|e^{\sigma_{r+1}\tau|D_s|}P^{\gamma}\Tilde{P}^{\Tilde{\gamma}}(\Jk u_{k}\otimes u_{k})\|_{L^2}\lesssim \sum_{\substack{0\leq i\leq \ell\\ 0\leq j\leq d}}\sum_{\substack{\alpha\in[\![1,2d]\!]^i\\ \beta\in[\![1,2d]\!]^{\ell-i}}}\sum_{\substack{\Tilde{\alpha}\in[\![1,2d]\!]^j\\ \Tilde{\beta}\in[\![1,2d]\!]^{d-j}}}\|e^{\sigma_{r+1}\tau|D_s|}(\Tilde{P}^{\Tilde{\alpha}}P^{\alpha}\Jk u_{k}\otimes \Tilde{P}^{\Tilde{\beta}}P^{\beta}u_{k})\|_{L^2}.
$$
Then we have
\begin{align}
   | \langle e^{\sigma_{r+1}\tau|D_s|}&\mathbb{P}\circ\Jk(u_{k}\cdot\nabla_{\Hg}\Jk u_{k}),e^{\sigma_{r+1}\tau|D_s|}\Ja u_{k}\rangle_{H_{r+1,d}} |\nonumber\\
    &\lesssim \|e^{\sigma_{r+1}\tau|D_s|}\nabla_{\Hg}u_{k}\|_{H_{r+1,d}}\nonumber\\
    &\ \ \ \ \ \times\sum_{\ell=0}^{r+1}\sum_{\substack{0\leq i\leq \ell\\ 0\leq j\leq d}}\sum_{\substack{\alpha\in[\![1,2d]\!]^i\\ \beta\in[\![1,2d]\!]^{\ell-i}}}\sum_{\substack{\Tilde{\alpha}\in[\![1,2d]\!]^j\\ \Tilde{\beta}\in[\![1,2d]\!]^{d-j}}}\|e^{\sigma_{r+1}\tau|D_s|}(\Tilde{P}^{\Tilde{\alpha}}P^{\alpha}\Jk u_{k}\otimes \Tilde{P}^{\Tilde{\beta}}P^{\beta}u_{k})\|_{L^2}.\label{IV 01/08/2023}
\end{align}
Let us point out that the Calderón-Zygmund theory using in Lemma \ref{I 5/01/2023 Lemme} cannot be applied in order to obtain a suitable estimate on $I_2^{a,k}(t)$ in the cases $(i,j)=(r+1,d)$ and $(i,j)=(0,0)$. Indeed, in these cases, we would like to choose one of the indexes $ p_1$ or $p_2$  in Lemma \ref{I 5/01/2023 Lemme} as $1$ or $+ \infty$, which is not allowed. In order to work around this difficulty, we will use interpolation inequalities.
Let us fix $\theta\in(0,1)$. Let $\ell\in[\![0,r+1]\!]$, $i\in[\![0,\ell]\!]$, $j\in[\![0,d]\!]$, $\alpha\in[\![1,2d]\!]^{i}$, $\Tilde{\alpha}\in [\![1,2d]\!]^j$, $\beta\in[\![1,2d]\!]^{\ell-i}$ and $\Tilde{\beta}\in[\![1,2d]\!]^{d-j}$. Then, using Lemma \ref{I 5/01/2023 Lemme}, we get
\begin{align*}
   \|e^{\sigma_{r+1}\tau|D_s|}(\Tilde{P}^{\Tilde{\alpha}}P^{\alpha}\Jk u_{k}\otimes\Tilde{P}^{\Tilde{\beta}}&P^{\beta}u_{k})\|_{L^2}\\
   &\leq \|e^{\sigma_{r+1}\tau|D_s|}\Tilde{P}^{\Tilde{\alpha}}P^{\alpha}\Jk u_{k}\|_{L^{\frac{2Q}{Q-2(d-j+\theta)}}}\|e^{\sigma_{r+1}\tau|D_s|}\Tilde{P}^{\Tilde{\beta}}P^{\beta}u_{k}\|_{L^{\frac{2Q}{Q-2(j+1-\theta)}}}.
\end{align*}
Thanks to the Sobolev embeddings $\Tilde{W}^{d-j,\frac{2Q}{Q-2\theta}}_{\Hg}\hookrightarrow L^{\frac{2Q}{Q-2(d-j+\theta)}}$ and $\Tilde{W}^{j,\frac{2Q}{Q-2(1-\theta)}}_{\Hg}\hookrightarrow L^{\frac{2Q}{Q-2(j+1-\theta)}}$, we obtain
\begin{align*}
   \|e^{\sigma_{r+1}\tau|D_s|}\Tilde{P}^{\Tilde{\alpha}}P^{\alpha}\Jk u_{k}\|_{L^{\frac{2Q}{Q-2(d-j+\theta)}}}\lesssim \|e^{\sigma_{r+1}\tau|D_s|}(-\Tilde{\Laplace}_{\Hg})^{\frac{d}{2}}(-\Laplace_{\Hg})^{\frac{i}{2}}\Jk u_{k}\|_{L^{\frac{2Q}{Q-2\theta}}}
\end{align*}
and 
\begin{align*}
   \|e^{\sigma_{r+1}\tau|D_s|}\Tilde{P}^{\Tilde{\beta}}P^{\beta}u_{k}\|_{L^{\frac{2Q}{Q-2(j+1-\theta)}}}\lesssim \|e^{\sigma_{r+1}\tau|D_s|}(-\Tilde{\Laplace}_{\Hg})^{\frac{d}{2}}(-\Laplace_{\Hg})^{\frac{\ell-i}{2}}u_{k}\|_{L^{\frac{2Q}{Q-2(1-\theta)}}}.
\end{align*}
Furthermore, using interpolation, the Sobolev embedding $\dot{H}^1(\Hg^d)\hookrightarrow L^{\frac{2Q}{Q-2}}(\Hg^d)$ and the properties of $\Jk$, we deduce that
\begin{align*}
   \|e^{\sigma_{r+1}\tau|D_s|}(-\Tilde{\Laplace}_{\Hg})^{\frac{d}{2}}(-&\Laplace_{\Hg})^{\frac{i}{2}}\Jk u_{k}\|_{L^{\frac{2Q}{Q-2\theta}}}\nonumber\\
   &\leq \|e^{\sigma_{r+1}\tau|D_s|}(-\Tilde{\Laplace}_{\Hg})^{\frac{d}{2}}(-\Laplace_{\Hg})^{\frac{i}{2}}u_{k}\|_{L^2}^{1-\theta}\|e^{\sigma_{r+1}\tau|D_s|}(-\Tilde{\Laplace}_{\Hg})^{\frac{d}{2}}(-\Laplace_{\Hg})^{\frac{i}{2}}\Jk u_{k}\|^{\theta}_{L^{\frac{2Q}{Q-2}}}\\
   &\lesssim\|e^{\sigma_{r+1}\tau|D_s|}u_{k}\|_{H_{i,d}}^{1-\theta}\|e^{\sigma_{r+1}\tau|D_s|}\nabla_{\Hg}u_{k}\|_{H_{i,d}}^{\theta}
\end{align*}
 and
\begin{align*}
\|e^{\sigma_{r+1}\tau|D_s|}(-\Tilde{\Laplace}_{\Hg}&)^{\frac{d}{2}}(-\Laplace_{\Hg})^{\frac{\ell-i}{2}}u_{k}\|_{L^{\frac{2Q}{Q-2(1-\theta)}}}\\
&\leq\|e^{\sigma_{r+1}\tau|D_s|}(-\Tilde{\Laplace}_{\Hg})^{\frac{d}{2}}(-\Laplace_{\Hg})^{\frac{\ell-i}{2}}u_{k}\|_{L^2}^{\theta}\|e^{\sigma_{r+1}\tau|D_s|}(-\Tilde{\Laplace}_{\Hg})^{\frac{d}{2}}(-\Laplace_{\Hg})^{\frac{\ell-i}{2}}u_{k}\|_{L^{\frac{2Q}{Q-2}}}^{1-\theta}\\
&\lesssim \|e^{\sigma_{r+1}\tau|D_s|}u_{k}\|_{H_{\ell-i,d}}^{\theta}\|e^{\sigma_{r+1}\tau|D_s|}\nabla_{\Hg}u_{k}\|_{H_{\ell-i,d}}^{1-\theta}.
\end{align*}
Then, it follows that
\begin{align}
    \|&e^{\sigma_{r+1}\tau|D_s|}(\Tilde{P}^{\Tilde{\alpha}}P^{\alpha}\Jk u_{k}\otimes \Tilde{P}^{\Tilde{\beta}}P^{\beta}u_{k})\|_{L^2}\nonumber\\
    &\lesssim \|e^{\sigma_{r+1}\tau|D_s|}u_{k}\|_{H_{i,d}}^{1-\theta}\|e^{\sigma_{r+1}\tau|D_s}\nabla_{\Hg}u_{k}\|_{H_{i,d}}^{\theta}\|e^{\sigma_{r+1}\tau|D_s|}u_{k}\|_{H_{\ell-i,d}}^{\theta}\|e^{\sigma_{r+1}\tau|D_s|}\nabla_{\Hg}u_{k}\|^{1-\theta}_{H_{\ell-i,d}}. \label{I 25/08/2023}
\end{align}
We have to distinguish three cases in order to estimate the right-hand side of \eqref{I 25/08/2023}.\\

$\bullet$ \textit{Case $\ell<r+1$ or $\ell=r+1$ and $0<i<r+1$.} In this case, we have $i\leq r$ and $\ell-i\leq r$. It then follows that
\begin{align*}    \|e^{\sigma_{r+1}\tau|D_s|}u_{k}\|_{H_{i,d}}^{1-\theta}\|e^{\sigma_{r+1}\tau|D_s|}\nabla_{\Hg}u_{k}\|_{H_{i,d}}^{\theta}\|e^{\sigma_{r+1}\tau|D_s|}u_{k}\|_{H_{\ell-i,d}}^{\theta}\|&e^{\sigma_{r+1}\tau|D_s|}\nabla_{\Hg}u_{k}\|_{H_{\ell-i,d}}^{1-\theta}\\
&\leq \|e^{\sigma_{r+1}\tau|D_s|}u_{k}\|^{2}_{H_{r+1,d}}.
\end{align*}

$\bullet$ \textit{Case $\ell=r+1$ and $i=0$.} In this case, since $\sigma_{r+1}<\sigma_{r}$, according to \eqref{I 31/07/2023}, we get 
\begin{align*}
  \|e^{\sigma_{r+1}\tau|D_s|}u_{k}\|_{\Tilde{H}^{d}}^{1-\theta}\|e^{\sigma_{r+1}\tau|D_s|}&\nabla_{\Hg}u_{k}\|_{\Tilde{H}^{d}}^{\theta}\|e^{\sigma_{r+1}\tau|D_s|}u_{k}\|_{H_{r+1,d}}^{\theta}\|e^{\sigma_{r+1}\tau|D_s|}\nabla_{\Hg}u_{k}\|_{H_{r+1,d}}^{1-\theta}\\
  &\leq  C_{r}^{1-\theta}\|e^{\sigma_{r+1}\tau|D_s|}u_{k}\|_{H_{r+1,d}}^{2\theta}\|e^{\sigma_{r+1}\tau|D_s|}\nabla_{\Hg}u_{k}\|_{H_{r+1,d}}^{1-\theta}.
\end{align*}

$\bullet$ \textit{Case $i=\ell=r+1$.} Similarly as above, it follows from \eqref{I 31/07/2023} that
\begin{align*}
    \|e^{\sigma_{r+1}\tau|D_s|}u_{k}\|_{H_{r+1,d}}^{1-\theta}\|e^{\sigma_{r+1}\tau|D_s|}\nabla_{\Hg}u_{k}&\|_{H_{r+1,d}}^{\theta}\|e^{\sigma_{r+1}\tau|D_s|}u_{k}\|_{\Tilde{H}^{d}}^{\theta}\|e^{\sigma_{r+1}\tau|D_s|}\nabla_{\Hg}u_{k}\|_{\Tilde{H}^{d}}^{1-\theta}\\
    &\leq C_{r}^{\theta}\|e^{\sigma_{r+1}\tau|D_s|}\nabla_\Hg u_{k}\|_{H_{r+1,d}}^{\theta}\|e^{\sigma_{r+1}\tau|D_s|} u_{k}\|_{H_{r+1,d}}^{2(1-\theta)}.
\end{align*}
Thus, according to \eqref{IV 01/08/2023} and \eqref{I 25/08/2023}, we deduce that there exists a positive constant $C_{\star}$ independent of $k$ and $a$ such that
\begin{align*}
    I_{2}^{a,k}(t)\leq& C_{\star}\int_{t_k}^{t}\|e^{\sigma_{r+1}\tau|D_s|}\nabla_{\Hg}u_{k}(\tau)\|_{H_{r+1,d}}\|e^{\sigma_{r+1}\tau|D_s|}u_{k}(\tau)\|_{H_{r+1,d}}^{2}d\tau\\
    &+C_{\star}\int_{t_k}^{t}\|e^{\sigma_{r+1}\tau|D_s|}u_{k}(\tau)\|_{H_{r+1,d}}^{2\theta}\|e^{\sigma_{r+1}\tau|D_s|}\nabla_{\Hg}u_{k}(\tau)\|_{H_{r+1,d}}^{2-\theta}d\tau\\
    &+C_{\star}\int_{t_k}^{t}\|e^{\sigma_{r+1}\tau|D_s|}\nabla_{\Hg}u_{k}(\tau)\|_{H_{r+1,d}}^{1+\theta}\|e^{\sigma_{r+1}\tau|D_s|}u_{k}(\tau)\|_{H_{r+1,d}}^{2(1-\theta)}d\tau .
\end{align*}
In view of $\theta\in(0,1)$, by using the Young estimates with $1/2+1/2=1$, $\theta/2+(2-\theta)/2=1$ and $(1+\theta)/2+(1-\theta)/2=1$ we deduce that there exists a positive constant $\Tilde{C}_{*}$, such that 
\begin{align}
    I_{2}^{a,k}(t)\leq& \left(1-\frac{\sigma_{r+1}}{4d}\right)\int_{t_k}^{t}\|e^{\sigma_{r+1}\tau|D_s|}\nabla_{\Hg}u_{k}(\tau)\|^{2}_{H_{r+1,d}}d\tau
    \nonumber\\
    &
    +\Tilde{C}_*\int^{t}_{t_k}\|e^{\sigma_{r+1}\tau|D_s|}u_{k}(\tau)\|_{H_{r+1,d}}^{4}d\tau.\label{I 09/08/2023}
\end{align}

\underline{\textit{Conclusion.}} By combining \eqref{IV 31/07/2023} and \eqref{I 09/08/2023},  according to \eqref{IX 01/08/2023}, we conclude that for all $t \in (t_k, T)$, 
\begin{align*}
    \|e^{\sigma_{r+1}t|D_s|}u_{k}(t)&\|_{H_{r+1,d}}^{2}+\left(1-\frac{\sigma_{r+1}}{4d}\right)\int_{t_k}^{t}\|e^{\sigma_{r+1}\tau|D_s|}\nabla_{\Hg}u_{k}(\tau)\|_{H_{r+1,d}}^{2}d\tau\\
    &\leq 
    \frac{4C_r}{(T_\star-T_r)}
    +C_{I_1}+\Tilde{C}_{*}\int_{t_k}^{t}\|e^{\sigma_{r+1}\tau|D_s|}u_{k}(\tau)\|_{H_{r+1,d}}^{4}d\tau.
\end{align*}
We deduce that there exists a positive constant $C^{\star}$ independent of $k$ such that for all $t \in (t_k,T)$, 
\begin{align}
\|e^{\sigma_{r+1}t|D_s|}u_{k}(t)\|_{H_{r+1,d}}^{2}+\int_{t_k}^{t}\|e^{\sigma_{r+1}\tau|D_s|}&\nabla_{\Hg}u_{k}(\tau)\|_{H_{r+1,d}}^{2}d\tau\nonumber\\
&\leq C^{\star}+C^{\star}\int_{t_k}^{t}\|e^{\sigma_{r+1}\tau|D_s|}u_{k}(\tau)\|^{4}_{H_{r+1,d}}d\tau.\label{XII 01/08/2023}
\end{align}
By applying the Gronwall lemma, we deduce from \eqref{XII 01/08/2023}, \eqref{I 31/07/2023}, $\sigma_{r+1}<\sigma_{r}$ and $t_k >T_r$ that
\begin{equation*}
\|e^{\sigma_{r+1}t|D_s|}u_{k}\|_{L^{\infty}([t_k,T];H_{r+1,d})}^{2}\leq C^{\star}e^{C^{\star}\int_{t_k}^{T}\|e^{\sigma_{r+1}\tau|D_s|}u_{k}(\tau)\|_{H_{r+1,d}}^{2}d\tau}\leq C^{\star}e^{C^{\star}C_{r}}.
\end{equation*}
Thus, according to \eqref{XII 01/08/2023}, we conclude that for any $t\in(t_k,T)$
\begin{align*}
    \|e^{\sigma_{r+1}t|D_s|}u_{k}(t)\|_{H_{r+1,d}}^{2}+\int_{t_k}^{t}\|e^{\sigma_{r+1}\tau|D_s|}\nabla_{\Hg}u_{k}(\tau)\|_{H_{r+1,d}}^{2}d\tau\leq C^{\star}+C^{\star}e^{C^{\star}C_{r}(T-t_k)}.
\end{align*}
Then, by setting 
$$
T_{r+1}:=\sup\{t_k\}\in \left(T_r,T_\star \right),
$$
since $T_r< t_k\leq T_{r+1}$, we deduce that
\begin{align*}
    \|e^{\sigma_{r+1}t|D_s|}u_{k}\|_{L^{\infty}([T_{r+1},T];H_{r+1,d})}^{2}+\int_{T_{r+1}}^{T}\|e^{\sigma_{r+1}\tau|D_s|}\nabla_{\Hg}u_{k}(\tau)\|_{H_{r+1,d}}^{2}d\tau\leq C^{\star}+C^{\star}e^{C^{\star}C_{r}(T-T_r)}.
\end{align*}
This complete the proof of the lemma.
\end{proof}

\begin{rem}
\label{rem:propagation des champs horizontaux}
The proof of Lemma \ref{Lemme 2 08/08/2023} contains two main difficulties: the main difficulty due to the term $(\Id-\mathbb{P})\circ(-\Laplace_{\Hg})$ appears in the estimate of the term $I^{a,k}_{1}$; in order to handle $I_2^{a,k}$, we have to use the operators $\Ja$ to show that each $u_k$ belongs to the appropriate space of regularity with respect to the left invariant sub-Laplacian $\Laplace_\Hg$, so that it can be used to absorb a part of the nonlinear term with the diffusive term, see \eqref{I 09/08/2023}.
\end{rem}

Now, we are able to show Corollary \ref{strong solution}. 
\begin{proof}[Proof of Corollary \ref{strong solution}.] We shall perform a bootstrap argument in order to show that the solution is smooth. Let us consider two real numbers $T_{\star}$ and $T^{\star}$ such that $T^{\star}>T_{\star}>0$. According to \eqref{energie analytic}, we deduce that there exists a positive real number $C_0$ such that
$$
\|e^{\sigma t|D_s|}u\|_{L^{\infty}([0,T^{\star}];\Tilde{H}^d)}+\int_{0}^{T^{\star}}\|e^{\sigma \tau|D_s|}\nabla_{\Hg}u(\tau)\|_{\Tilde{H}^d}^{2}d\tau\leq C_{0}.
$$
By reasoning by induction using Lemma \ref{Lemme 2 08/08/2023} and passing to the limit in $(u_k)$, up to extract a subsequence, we deduce that for any $r\in\N$ and $\beta\in\N$ 
\begin{equation}
    \label{I 26/08/2023}
    \|u\|_{L^{\infty}([T_{\star},T^{\star}];H_{r,d})}\lesssim 1.
\end{equation}
Let $\alpha\in\N^{2d+1}$. There exists a family of polynomial function, on the horizontal variable $Y$, $(\mu_{\beta,\gamma})_{(\beta,\gamma)\in\Gamma_{\alpha}}$ where $\Gamma_{\alpha}$ is a subset of $\left(\bigcup_{i=0}^{|\alpha|}[\![1,2d]\!]^{i}\right)\times[\![1,|\alpha|]\!]$, such that 
$$
\partial^{\alpha}=\sum_{(\beta,\gamma)\in\Gamma_{\alpha}}\mu_{\beta,\gamma}P^{\beta}\partial_{s}^{\gamma}.
$$
Thus, for any open bounded subset $\Omega$ of $\R^{2d+1}$ and for any $t\in[T_{\star},T^{\star}]$, we have
\begin{align}
    \label{II 26/08/2023}   \|\partial^{\alpha}u(t)\|_{L^{\infty}(\Omega)}\lesssim\sum_{i=0}^{|\alpha|}\sum_{\substack{\beta\in[\![1,2d]\!]^{i}\\ \gamma\in[\![1,|\alpha|]\!]}}\|P^{\beta}\partial_{s}^{\gamma}u(t)\|_{L^{\infty}(\Omega)}.
\end{align}
We have a continuous embedding from $W^{2,Q}_{\Hg}(\Hg^d)$ into $\mathcal{C}(\R^{2d+1})\cap L^{\infty}(\R^{2d+1})$. By using the Sobolev embedding of $\Tilde{H}^{d}(\Hg^d)$ into $L^{Q}(\Hg^d)$, we conclude that $H_{2,d}(\Hg^d)$ is continuously embedding in $\mathcal{C}_{b}(\R^{2d+1})$. Thus, we deduce from \eqref{II 26/08/2023} and Proposition \ref{prop- Ds est d'ordre 2}, that for any $t\in[T_{\star},T^{\star}]$, we have
\begin{align*}
    \|\partial^{\alpha}u(t)\|_{L^{\infty}(\Omega)}&\lesssim\sum_{i=0}^{|\alpha|}\sum_{\substack{\alpha\in[\![1,2d]\!]^{i}\\ \gamma\in[\![1,|\alpha|]\!]}}\|P^{\beta}\partial_{s}^{\gamma}u(t)\|_{H_{2,d}}
    \lesssim \|u(t)\|_{H_{3|\alpha|+2,d}}.
\end{align*}
It follows from \eqref{I 26/08/2023} that for any $t\in [T_{\star},T^{\star}]$, we have $\partial^{\alpha}u(t)\in L^{\infty}(\Omega)$. Since $T_{\star} < T^{\star}$, $\alpha$ and $\Omega$ are arbitrarily chosen, this shows that for any $t>0$, $u(t)$ belongs to $\mathcal{C}^{\infty}(\R^{2d+1})$. Since $p$ is explicitly given by \eqref{p-function-of-u}, by hypoellipticity (see \cite{Hypoellipticsecondorderdifferentialequations}) of the left-invariant sub-Laplacian $\Laplace_{\Hg}$, we also have that for all $t >0$, $p(t) \in \mathcal{C}^{\infty}(\R^{2d+1})$.
Besides, for any $\ell\in \N^{*}$, we have
$$
\partial_{t}^{\ell}u=\Laplace_{\Hg}\partial_{t}^{\ell-1}u-\sum_{k=0}^{\ell-1}\binom{\ell-1}{k}\partial_{t}^{k}u\cdot\nabla_{\Hg}\partial_{t}^{\ell-1-k}u-\nabla_{\Hg}\partial_{t}^{\ell-1}p.
$$ 
By induction on $\ell$, we deduce that for any $t>0$ and $\ell\in\N$, $\partial_{t}^{\ell}u(t)$, and then, according to the hypoellipticity of $\Laplace_\Hg$, $\partial_{t}^{\ell}p(t)$, belong to $\mathcal{C}^{\infty}(\R^{2d+1})$. We conclude that $u$ and $p$ belongs to $\mathcal{C}^{\infty}((0,+\infty)\times\R^{2d+1})$.
\end{proof}
\section{Long time existence in $\dot{H}^d(\Hg^d)$}
\label{sec:Long time existence in dotHd}
In this subsection, we discuss  the existence of solution for \eqref{NSH} in $\dot{H}^d(\Hg^d)$, and explain how to obtain long time existence results in this framework.
\begin{theo}
\label{theo:Existence dans dotHd}
Let $T_\star>0$. Then there exist $a>0$ and $\varepsilon>0$ such that for any horizontal vector field $u_0\in \dot{H}^d(\Hg^d)$ with $\divergence_\Hg(u_0)=0$ and $\|e^{a|D_s|}u_0\|_{\dot{H}^d}<\varepsilon$, there exists a solution $u$ of \eqref{NSH}-\eqref{NSH initial condition} satisfying $u\in\mathcal{C}_b([0,T_\star];\dot{H}^d)$ and $\nabla_{\Hg}u\in L^2((0,T_\star);\dot{H}^d)$.
\end{theo}

Note that Theorem \ref{theo:Existence dans dotHd} requires the initial datum to have some analyticity properties with respect to the vertical variable. Also note that the uniqueness of the solution of \eqref{NSH}-\eqref{NSH initial condition} satisfying $u\in\mathcal{C}_b([0,T_\star];\dot{H}^d)$ and $\nabla_{\Hg}u\in L^2((0,T_\star);\dot{H}^d)$ is an open problem.

\begin{proof}
We only give the \textit{a priori} estimates on smooth solution. The the convergence is left to the reader. Let $a>0$, which will be chosen later and $u_0$ be a horizontal vector field in $\dot{H}^d(\Hg^d)$ so that
$$
\divergence_{\Hg}(u_0)=0\ \ \text{and}\ \ \|e^{a|D_s|}u_0\|_{\dot{H}^d}<\varepsilon,
$$
for some $\varepsilon>0$ small enough which we will fix later in the proof. Let us consider a decreasing positive function $\delta$ in $\mathcal{C}^1(\R_{+})$ (which depend on the time variable $t$). Let $u$ a solution of \eqref{NSHk}-\eqref{NSHk initial condition} for some index $k\in\N$. We have
\begin{multline}
\frac{1}{2}\frac{d}{dt}\|e^{\delta|D_s|}u\|_{\dot{H}^d}^{2}-\dot{\delta}\|e^{\delta|D_s|}|D_s|^{\frac 12}u\|_{\dot{H}^d}^{2}+\|e^{\delta|D_s|}\nabla_{\Hg}u\|_{\dot{H}^d}^{2}
\\
=-\langle e^{\delta|D_s|}\Jk(\Id-\mathbb{P})\circ(-\Laplace_{\Hg})u,e^{\delta|D_s|}u\rangle_{\dot{H}^d}-\langle e^{\delta|D_s|}\Jk\mathbb{P}(u\cdot\nabla_{\Hg}\Jk u),e^{\delta|D_s|}u\rangle_{\dot{H}^d}.\label{V 24/05/2023}
\end{multline}
Since $(\Id-\mathbb{P})\circ(-\Laplace_{\Hg})u=\Pi_{\Hg}\partial_su$ and thanks to Proposition \ref{Prop- I 25/11/2023} Items 3, 4 and 5, we deduce that there exists a constant $C_{\star}$ which does not depend on $k\in\N$ such that
$$
|\langle e^{\delta|D_s|}\Jk(\Id-\mathbb{P})\circ(-\Laplace_{\Hg})u,e^{\delta|D_s|}u\rangle_{\dot{H}^d}|\leq C_{\star}\|e^{\delta|D_s|}|D_s|^{\frac{1}{2}}u\|_{\dot{H}^d}^{2}.
$$
Then, it  follows from Proposition \ref{prop- Ds est d'ordre 2} and \eqref{V 24/05/2023} that 
\begin{equation}
	\label{Est-Intermediaire}
\frac{1}{2}\frac{d}{dt}\|e^{\delta|D_s|}u\|_{\dot{H}^d}^{2}+\left(1-\frac{(\dot{\delta}+C_{\star})}{4d}\right)\|e^{\delta|D_s|}\nabla_{\Hg}u\|_{\dot{H}^d}^{2}\leq
-\langle e^{\delta|D_s|}\Jk\mathbb{P}(u\cdot\nabla_{\Hg} \Jk u),e^{\delta|D_s|}u\rangle_{\dot{H}^d}.
\end{equation}
Using Proposition \ref{Prop- I 25/11/2023} Items 3, 4 and 5 and performing the same argument as in the proof of Lemma \ref{II 06 10 2022} we deduce that 
\begin{align}
|\langle e^{\delta|D_s|}\Jk\mathbb{P}(u\cdot\nabla_{\Hg} \Jk u), e^{\delta|D_s|}u\rangle_{\dot{H}^d}|\leq C^{\star}\|e^{\delta|D_s|}u\|_{\dot{H}^d}\|e^{\delta|D_s|}\nabla_{\Hg}u\|_{\dot{H}^{d}}^{2}.\label{VI 24/05/2023}
\end{align}
Accordingly, using \eqref{Est-Intermediaire} and \eqref{VI 24/05/2023}, we get
\begin{equation}
\label{eq: Energy in dotHd}
\frac{1}{2}\frac{d}{dt}\|e^{\delta|D_s|}u\|_{\dot{H}^d}^{2}+\left(1-\frac{(\dot{\delta}+C_{\star})}{4d}\right)\|e^{\delta|D_s|}\nabla_{\Hg}u\|_{\dot{H}^d}^{2}\leq C^{\star}\|e^{\delta|D_s|}u\|_{\dot{H}^d}\|e^{\delta|D_s|}\nabla_{\Hg}u\|_{\dot{H}^d}^{2}.
\end{equation}
We now choose $T_{\star}>0$, $a$ and $\delta$ a positive function in $\mathcal{C}^1(\R_{+})$ such that
$$
C_{\star}+\max_{t\in[0,T_\star]}\lbrace\dot{\delta}(t)\rbrace<4d\ \  \text{and}\ \ \ \delta(0)=a.
$$
This can be done for instance by choosing $\delta_1 \in \R$ such that $C_\star + \delta_1 < 4d$, $a \geq \max\{0, -\delta_1 T_\star\}$, and $\delta (t) = a + \delta_1 t$.

Then, if we set $\delta_{\star}:=(C_{\star}+\max_{t\in[0,T_\star]}\lbrace\dot{\delta}(t)\rbrace)/4d$, we deduce that for any $T$ in  $[0,T_{\star}]$, we have
\begin{align*}
\|e^{\delta|D_s|}u\|_{L^{\infty}_{T}(\dot{H}^d)}^{2}+2(1-\delta_{\star})&\|e^{\delta|D_s|}\nabla_{\Hg}u\|_{L^{2}_{T}(\dot{H}^d)}^{2}\\
&\leq \|e^{a|D_s|}u_0\|_{\dot{H}^d}^{2}+2C^{\star}\|e^{\delta|D_s|}u\|_{L^{\infty}_{T}(\dot{H}^d)}\|e^{\delta|D_s|}\nabla_{\Hg}u\|_{L^{2}_{T}(\dot{H}^d)}^{2}.
\end{align*}
We deduce from bootstrap type arguments that for $\varepsilon>0$ chosen small enough, the following inequality holds
$$
\|u\|_{L^{\infty}_{T_{\star}}(\dot{H}^{d})}^{2}+\|\nabla_{\Hg}u\|_{L^{2}_{T_{\star}}(\dot{H}^d)}^{2}\leq \|e^{\delta|D_s|}u\|_{L^{\infty}_{T_{\star}}(\dot{H}^{d})}^{2}+\|e^{\delta|D_s|}\nabla_{\Hg}u\|_{L^{2}_{T_{\star}}(\dot{H}^d)}^{2}\leq \|e^{a|D_s|}u_0\|_{\dot{H}^d}^{2}.
$$
Using a compactness argument, we deduce from the above energy inequality that there is a solution $u$ in $\mathcal{C}([0,T_{\star}];\dot{H}^d)$ so that $u_{|_{t=0}}=u_0$ and the above inequality holds.
\end{proof}
The proof of Theorem \ref{theo:Existence dans dotHd} suggests the following important remark. In view of \eqref{eq: Energy in dotHd}, if the constants $C_\star$ is strictly smaller than $4d$, the analyticity assumption on the initial datum is not needed: in such case, one could then take $a = 0$ and $\delta_1 \in (0,4d - C_\star)$, and $\delta (t) = \delta_1 t$ in the above proof. In this sense, the size of $C_\star$ measures the loss of regularity (or dissipation) generated by the term $(\Id-\mathbb{P})\Laplace_\Hg u$ in $\dot{H}^d(\Hg^d)$. However, even if $C_\star<4d$, it is not clear that we can obtain an analog of the stability estimate of Theorem \ref{Stabilite dans Tilde H} in the $\dot{H}^d(\Hg^d)$ framework. Indeed, we cannot obtain a commutator estimate similar to Lemma \ref{II 06 10 2022 version gain de derive r et l sur a (dans le produit)} with $P^\alpha$ instead $\Tilde{P}^\alpha$, essentially due to the fact that the commutators $[X_j,\Xi_j]=-4\partial_s$, with $j\in[\![1,d]\!]$, are homogeneous left-invariant operators of order two (see Lemma \ref{prop- Ds est d'ordre 2}).

\section{Sub-Riemannian Euler and Navier-Stokes systems for general left-invariant structures}
\label{sec:Comments on the generalization and open problems}
Let $G$ be a stratified Lie group of step $r$ and dimension $N$ (which we identify to $\R^N$ with a suitable group law) and let us denote by $\mathfrak{g}$ its Lie algebra of left-invariant vector fields. We fix a stratification $\mathfrak{g}=\bigoplus_{j=1}^{r}\mathfrak{g}_j$ and we consider the Jacobian generators (see \cite[Definition 1.4.1, p. 56]{StratifiedLieGroupsandPotentialTheoryfortheirSubLaplacians}) $Z_1,\dots,Z_{N'}$, with $N':=\dim(\mathfrak{g}_1)\leq N$, of $\mathfrak{g}$, that is: $Z_1,\dots,Z_{N'}$ are generators of $\mathfrak{g}$ such that $\mathfrak{g}_1=\mathrm{Span}(Z_1,\dots,Z_{N'})$. Then, for any $j\in[\![1,N']\!]$, we have $Z_j=\sum_{k=1}^{N}b_{k}^{j}\partial_k$, with $b_{j}^{j}=1$, $b_{k}^{j}=0$ for any $k\in[\![1,N']\!]\setminus\{j\}$ and, if $N'<N$, for any $k\in[\![N'+1,N]\!]$, the function $b_{k}^{j}$ is polynomial and independent of $x_k$. Let us set $\mathcal{R}_{\mathfrak{g}_1}:=(b_{k}^{j})_{1\leq k\leq N, 1\leq j\leq N'}$. Then $\mathcal{R}_{\mathfrak{g}_1}:\R^{N}\rightarrow\mathcal{L}(\R^{N'},\R^N)$. Now, if we consider the left-invariant sub-Riemannian structure on $G$, then $v=\!^t(v_1,\dots,v_N)$ is a horizontal vector field if and only if $\mathcal{R}_{\mathfrak{g}_1}\!^t(v_1,\dots,v_{N'})=v$. We also say that $\!^t(v_1,\dots,v_{N'})$ is a horizontal vector field. Generalizing the strategy of Section \ref{subsec:Presentation of the model} to $G$ (in such context $\mathcal{R}_{\mathfrak{g}_1}$ plays the same role as $\mathcal{R}$) for the sub-Riemannian Euler and Navier-Stokes system on the Heisenberg group, we obtain the following system:
\begin{equation}
\label{eq:NSG derivation}
\begin{cases}
\partial_tu+u\cdot\nabla_{G}u-\nu\Laplace_{G}u=-\nabla_{G}p,\\
\divergence_{G}(u)=0,
\end{cases}
\end{equation}
where $\nu$ belongs to $[0,+\infty)$, $\nabla_{G}:=\!^t\mathcal{R}_{\mathfrak{g}_1}\nabla$, $\divergence_G:=\divergence(\mathcal{R}_{\mathfrak{g}_1}\cdot)$ and $\Laplace_G:=\divergence(\mathcal{R}_{\mathfrak{g}_1}\!^t\mathcal{R}_{\mathfrak{g}_1}\nabla\cdot)$. Let us remark that if $G=\Hg^d$, then $\mathcal{R}_{\mathfrak{h}_1}=\mathcal{R}$. Furthermore, if $N'=N$, then $G=(\R^N,+)$, $\mathcal{R}_{\mathfrak{g}_1}$ is the identity matrix and thus \eqref{eq:NSG derivation} is the incompressible Euler system on $\R^N$ when $\nu=0$, or the incompressible Navier-Stokes system on $\R^N$ when $\nu>0$. Moreover, for any smooth enough solution $(u,p)$ of \eqref{eq:NSG derivation}, we have
$$
\frac{d}{dt}\|u\|_{L^2}^{2}+\nu\|\nabla_G u\|_{L^2}^{2}=0.
$$
\paragraph{Existence of weak solutions.} In this paragraph, we show that we can generalize Theorem \ref{th main:Existence of weak solutions} to the case of System \eqref{eq:NSG derivation} with $\nu >0$. Let us first give the definition of a global weak solution for the Cauchy problem  \eqref{eq:NSG derivation} (which reduces to Definition \ref{def:Definition of weak solution for NSH} when $G$ is the Heisenberg group and to the usual Leray solutions if $G = \R^N$).
\begin{defi}
\label{def:global weak solutions on G}
Let $\nu >0$, and $u_0\in L^2(G)^{N'}$ be a horizontal vector field satisfying $\divergence_G(u_0)=0$. We say that $u\in L^{2}_{loc}(\R_+\times G)^{N'}$ is a global weak solution of 
\begin{equation}
\label{eq:NSG derivation nu 1 domain}
\begin{cases}
\partial_tu-\nu \Laplace_{G}u+u\cdot\nabla_{G} u=-\nabla_G p\  &\text{in}\ \ (0,+\infty)\times G,\\
\divergence_{G}(u)=0\  &\text{in}\ \ (0,+\infty)\times G
\end{cases}
\end{equation}
and 
\begin{equation}
    \label{eq:NSG initial condition}
    u_{|_{t=0}}=u_0\ \text{in}\ \ G,
\end{equation}
if 
\begin{enumerate}[topsep=0pt,parsep=0pt,leftmargin=*]
\item (Integrability conditions) $u$ belongs to $\mathcal{C}_{w}([0,+\infty);L^2)\cap L^{\infty}(\R_+; L^2)\cap L^2(\R_+;\dot{H}^1(G))$,
\item (Initial condition) $\lim_{t\rightarrow 0^+}u(t)=u_0$ for the weak topology of $L^2(G)$,
\item (Momentum equation) for any $t' \leq t$ in $[0,+\infty)$ and for any $\vphi\in\mathcal{D}((0,+\infty)\times G)^{N'}$ such that $\divergence_G(\vphi)=0$, we have
$$
\int_{G} u(t)\cdot\vphi(t)dx - \int_{t'}^{t}\int_{G}\left(u\cdot\partial_t\vphi+\nu u\cdot\Laplace_G \vphi+(u\otimes u)\cdot\nabla_G \vphi\right) dxd\tau = \int_{G}u(t')\cdot\vphi(t')dx,
$$
\item (Continuity equation) For all $t>0$, we have $\divergence_{G}(u(t))=0$ in $\mathcal{D}'(G)^{N'}$.
\end{enumerate} 
\end{defi}
By adapting the proof of Theorem \ref{th main:Existence of weak solutions} (the proof is left to the reader), we obtain the following theorem.
\begin{theo}
\label{th main:Existence of weak solutions-G}
Let $\nu >0$ and $u_0$ be a horizontal vector field belonging to $L^2(G)$ and satisfying $\divergence_{G}(u_0)=0$. Then there exists a global weak solution $u$ of \eqref{eq:NSG derivation nu 1 domain}-\eqref{eq:NSG initial condition}, satisfying the following energy estimate
\begin{equation*}
\|u\|_{L^{\infty}(L^2)}^{2}+2\nu \|\nabla_{G} u\|_{L^2(L^2)}^{2}\leq \|u_0\|_{L^2}^{2}.
\end{equation*}
\end{theo}

By setting $\mathbb{P}^{G}:=\Id+\nabla_G\circ(-\Laplace_G)^{-1}\circ\divergence_G$, we can rewrite System \eqref{eq:NSG derivation nu 1 domain} as follows
\begin{equation}
\label{eq:NSG derivation nu 1 sans pression}
\begin{cases}
\partial_tu-\nu \mathbb{P}^{G}\Laplace_{G}u+\mathbb{P}^{G}(u\cdot\nabla_{G}u)=0,\\
\divergence_{G}(u)=0.
\end{cases}
\end{equation}
Let us now give the main elements allowing to prove Theorem \ref{th main:Existence of weak solutions-G}.

\paragraph{Analysis on stratified Lie group.} Let us denote by $\widehat{G}$ the unitary dual of $G$. For any $\pi\in\widehat{G}$ we denote by $H_\pi$ the associated representation space and by $\mathscr{F}_G$ the Fourier transform on $G$. Note that $H_\pi$ is a Hilbert space and, thanks to the Kirillov theory (see \cite[Theorem 3, p. 103]{LecturesontheOrbitMethod}), can be realized as the space of square integrable functions on a Euclidean space. Now, for any $X\in\mathfrak{g}$, we define the right-invariant vector field $\Tilde{X}$ by $\Tilde{X}f:=-X\check{f}$, where for any $w\in G$, $\check{f}(w):=f(w^{-1})$ and $w^{-1}$ denotes  the inverse of $w$ for the law of $G$. We also denote by $\Tilde{\Laplace}_G$ the right-invariant sub-Laplacian on $G$ define by $\Tilde{\Laplace}_G=\sum_{j=1}^{N}\Tilde{Z}_{j}^{2}$. Let $T\in\{-\Laplace_G,Z_1,\dots,Z_{N'}\}$ and $\Tilde{T}\in\{-\Tilde{\Laplace}_G,\Tilde{Z}_1,\dots,\Tilde{Z}_{N'}\}$. Then $T$ and $\Tilde{T}$ commute. Moreover, if $X\in\mathfrak{g}$, then we have $\mathscr{F}_{G}(X f)(\pi)=\mathscr{F}_{G}(f)(\pi)\circ \pi(X)$ and $\mathscr{F}_{G}(\Tilde{X} f)(\pi)=\pi(X)\circ\mathscr{F}_{G}(f)(\pi)$. Also, for any $\pi\in\widehat{G}$, the operator $\pi(-\Laplace_G)$ is a self-adjoint compact operator on a separable Hilbert space $H_\pi$ and its spectrum $\{E(m,\pi)\}_{m\in\N}$ lies in $(0,+\infty)$. Thus if we denote by $(h_{m}^{\pi})_{m\in\N}$ an orthonormal basis of $H_\pi$ of eigenvectors of $\pi(-\Laplace_G)$, then for any $f\in L^1(G)$, we can set
$$
\mathcal{F}_{G}(f)(n,m,\pi):=\langle\mathscr{F}_G(f)(\pi)h_{m}^{\pi},h_{n}^{\pi}\rangle_{H_\pi},
$$
for any $(n,m,\pi)\in\N\times\N\times\widehat{G}$. (Note that, by using the Kirillov theory, we can parametrize $\widehat{G}$ by the orbit of the coadjoint action from $G$ on $\mathfrak{g}^{'}$ in order to obtain a more concrete description of the unitary dual, see \cite{LecturesontheOrbitMethod}). Similarly as in Sections \ref{subsubsec:Pseudo-differential calculus on Hd} and \ref{sec:Sobolev spaces}, for any $\ell\in\R$, we define the operators $(-\Laplace_G)^{\ell/2}:\mathrm{Dom}((-\Laplace_G)^{\ell/2})\subset L^2(G)\rightarrow L^2(G)$ and we denote by $\dot{H}^{\ell}(G)$ the closure of $\mathcal{S}(G)\cap \mathrm{Dom}((-\Laplace_G)^{\ell/2}))$ for the norm $\|(-\Laplace_G)^{\ell}\cdot\|_{L^2}$. By the same way, we define $(-\Tilde{\Laplace}_G)^{\ell/2}$ and $\Tilde{H}^{\ell}(G)$, with $\ell\in\R$. In particular, Proposition \ref{prop:Remark definition des puissance fractionnaire}, Lemma \ref{theo:symbole homogene} and the identities involving the sub-Laplacian in Proposition \ref{Fourier diagonalise the sublaplacian} can be adapted to $G$ (the details are left to the reader). Thanks to classical results, Propositions \ref{Propriete Sobolev sur Heisenberg} and \ref{proprietes des Sobolev H} also hold on $G$ (see \cite[Section 4]{QuantizationonNilpotentLieGroups}).
 
\paragraph{Approximate problem.} Our strategy in this article, based on \textit{a priori} estimates on suitable approximate problems, can be generalized to $G$ as follows. We first introduce for any $k\in\N$ the operators $\Jk^G$ and $\TJk^G$ by setting for any $f\in \mathcal{S}(G)$ and $(n,m,\pi)\in \N\times\N\times\widehat{G}$
$$
\mathcal{F}_G(\Jk^G f)(n,m,\pi):=\mathbf{1}_{\{\frac{1}{2^{k+1}}\leq E(n,\pi)\leq 2^k\}}(n,m,\pi)\mathbf{1}_{\{\frac{1}{2^{k+1}}\leq E(m,\pi)\leq 2^k\}}(n,m,\pi)\mathcal{F}_{G}(f)(n,m,\pi)
$$
and 
$$
\mathcal{F}_G(\TJk^G f)(n,m,\pi):=\mathbf{1}_{\{\frac{1}{2^{k+1}}\leq E(n,\pi)\leq 2^k\}}(n,m,\pi)\mathcal{F}_{G}(f)(n,m,\pi).
$$
Thus Proposition \ref{Prop- I 25/11/2023}, Items 1, 2, 3, 5, 6 and 7, Proposition \ref{Prop-Tilde Jk}, Items 1, 2, 3, 5, 6 and 7, and Proposition \ref{Prop- Leray projector}, Items 1, 2, 3, 4, 5, 7, 8 and 9 can be generalized to $\Jk^G$, $\TJk^G$ and $\mathbb{P}^G$. Furthermore, we can adapt the strategy of Section \ref{sec:Derivation of a suitable approximate system} to System \eqref{eq:NSG derivation nu 1 sans pression} in order to derive the following approximate problems depending on $k \in \N$:
\begin{equation}
\label{NSGk}
\begin{cases}
\partial_tu_k-\nu \mathbb{P}^G\Laplace_{G}u_k+\mathbb{P}^G\Jk^G (u_k\cdot\nabla_{G}\Jk^G u_k)=0\  &\text{in}\ \ (0,+\infty)\times G,\\
\divergence_{G}(u_k)=0\  &\text{in}\ \ (0,+\infty)\times G,\\
\TJk^G u_k=u_k\  &\text{in}\ \ (0,+\infty)\times G\\
\end{cases}
\end{equation}
and 
\begin{equation}
    \label{NSGk initial condition}
    u_{k|_{t=0}}=\TJk^G u_0\ \text{in}\ \ G,
\end{equation}
where $u_0$ is a horizontal vector field belonging to $L^2(G)^{N'}$ or $\Tilde{H}^\ell(G)^{N'}$ for $\ell\in \R$ and $k\in\N$. Thus, the solution $u_k$ belongs to $L^2(L^2)\cap L^2(\dot{H}^1(G))\cap\mathcal{C}_b(\Tilde{H}^{\ell'}(G))$ with $\nabla_G u_k\in L^2(\Tilde{H}^{\ell'}(G))$, for any $\ell'\in\R$ and satisfies
\begin{equation}
\label{energy L2 pour tout le monde G}
\|u_k\|_{L^{\infty}(L^2)}^{2}+2\nu \|\nabla_{G}u_k\|_{L^2(L^2)}^{2}\leq \|\TJk^G u_0\|_{L^2}^{2}.
\end{equation}
\appendix

\paragraph{Open problems.} 
We end this section with some open questions related to \eqref{eq:NSG derivation nu 1 domain}-\eqref{eq:NSG initial condition}.

A first open problem concerns the global well-posedness in the case of a general stratified Lie group $G$ for initial data in the critical space $\Tilde{H}^{Q/2-1}(G)$. 
Indeed, when the homogeneous dimension is odd, this imposes to work with  fractional order Sobolev spaces. In such case, estimating the nonlinear terms is challenging because we need to simultaneously address the regularity associated with both the right-invariant and left-invariant sub-Laplacians on $G$ (in this article, since $Q/2-1=d \in\N$ when $G = \Hg^d$, we only need to use the Leibniz formula and the Sobolev's embedding.) 

A second question concerns the regularity of the solution. In particular, when the step of $G$ is strictly larger than $2$, the interplay between the different stratums of vector fields is not so clear. Therefore, the regularization properties of the solutions of  \eqref{eq:NSG derivation nu 1 domain}-\eqref{eq:NSG initial condition} become much more difficult to analyze. Indeed, looking at the proof of Theorem \ref{th main:Regularity of the solution in TildeHd}, we compensate the loss of derivatives due to the degenerate diffusion operator $(\Id-\mathbb{P}^G)\circ\Laplace_G$ by a regularity of infinite order in the other directions. Such strategy is made possible in the case of the Heisenberg group by the crucial facts that a parametrization of $\widehat{G}$ is simply given by $\R^*$, and  that we have a spectral gap 
in the sense of Proposition \ref{prop- Ds est d'ordre 2}, revealed by the analysis  of the symbol of the sub-Laplacian (see also \cite{SpectralsummabilityforthequarticoscillatorwithapplicationstotheEngelgroup} in the context of the Engel group).
\section{A result of anisotropic analytic hypoellipticity of the fractional power of the sub-Laplacian on $\Hg^d$}
\label{sec:Radius of analyticity in hypoelliptic framework}
In this section we state and prove a result which interprets the radius of analyticity with respect to the vertical variable in the framework of the homogeneous Sobolev-type space $\Tilde{H}^\ell(\Hg^d)$ (the same result holds if we replace $\Tilde{\Laplace}_\Hg$ by $\Laplace_\Hg$, with the suitable modifications):
\begin{theo}
\label{theo:interpretation rayon d'analyticite Hg}
Let $\ell>0$. Let $f\in \Tilde{H}^\ell(\Hg^d)$ and $\sigma>0$. If $e^{\sigma|D_s|}f\in\Tilde{H}^\ell(\Hg^d)$, then there exist two functions $g$ and $h$ 
satisfying $$f=g+h$$ such that 
\begin{enumerate}[topsep=0pt,parsep=0pt,leftmargin=*]
    \item $g$ belongs to $\Tilde{H}^\ell(\Hg^d)\cap\mathcal{C}^{\infty}(\R^{2d+1})$ and for any $Y\in\R^{2d}$, the function $g(Y,\cdot)$ can be extended to a holomorphic function on $\C$,
    \item $h$ belongs to $L^2(\Hg^d)\cap\Tilde{H}^\ell(\Hg^d)$ and for any $Y\in\R^{2d}$, the function $h(Y,\cdot)$ can be extended to a holomorphic function in $\mathsf{S}_\sigma:=\enstq{z\in\C}{|\Im(z)|< \sigma}$.
\end{enumerate}
In particular, $f$ is analytic with respect to the vertical variable $s$. 
\end{theo}

Let us remark that this theorem implies partial (and global) analytic hypoellipticity properties for $(-\Tilde{\Laplace}_\Hg)^\ell$. Indeed, if $f$ is a complex value smooth function on $\Hg^d$, such that $(-\Tilde{\Laplace}_{\Hg})^\ell f$ belongs to $L^2(\Hg^d)$ and, for any $Y\in \R^{2d}$, the function $(-\Tilde{\Laplace}_{\Hg})^\ell f(Y,\cdot)$ extends to a holomorphic function on $\mathsf{S}_\sigma$ satisfying  
$$
\sup_{|\eta|<\sigma}\{\|(-\Tilde{\Laplace}_{\Hg})^\ell f(Y,\cdot+i\eta)\|_{L^2(\R)}\}<+\infty,
$$
then we have $e^{\sigma|D_s|}(-\Tilde{\Laplace}_\Hg)^\ell f\in L^{2}(\Hg^d)$ (this can be done for instance by adapting the proof of \cite[Theorem 1.1]{Cauchytheoryforthewaterwavessysteminananalyticframework} to the case of $\R$). It then follows from Theorem \ref{theo:interpretation rayon d'analyticite Hg} that $f(Y,\cdot)$ is analytic on $\R$ and extends to a holomorphic function on $\mathsf{S}_\sigma$.

Additionally, let us point out that, in the case of $\Hg^d$, the exponential decay of the Fourier transform does not imply analyticity. Indeed, there exists a smooth function $\psi\in L^2(\Hg^d)$ such that the support of $\mathcal{F}_{\Hg}(f)$ is included in the set 
$$
\enstq{(n,m,\lambda)}{4|\lambda|(2|n|+d)\leq R}
$$ 
for some $R>0$, and which cannot be extended into an analytic function on $\C^{2d+1}$ (see \cite{AnalysisontheHeisenberggroup}).

\begin{proof}[Proof of Theorem \ref{theo:interpretation rayon d'analyticite Hg}] Let $\psi_1$ be a smooth function on $\R$ with value in $[0,1]$ and satisfying $supp(\psi_1)\subset (-2,2)$ and $\psi_1=1$ on $[-1,1]$. Let us set $\psi_2:=1-\psi_1$. Since $\psi_1$ and $\psi_2$ are bounded, thanks to Proposition \ref{prop:Remark definition des puissance fractionnaire}, we can write 
$$
f=\psi_1(-\Tilde{\Laplace}_\Hg)f+\psi_2(-\Tilde{\Laplace}_\Hg)f.
$$
Because $\psi_1$ is compactly supported smooth function, we have $\psi_1(-\Tilde{\Laplace}_\Hg)f\in\mathcal{C}^{\infty}(\R^{2d+1})$ (use Proposition \ref{prop:Remark definition des puissance fractionnaire} to ensure that this function belongs to $\Tilde{H}^{\ell'}(\Hg^d)$ for any $\ell'\in [\ell,+\infty)$ and the Sobolev embedding in order to recover the continuity of all derivative with respect to the right-invariant vector fields). Besides, 
according to the Hulanicki theorem (see \cite[Corollary 4.5.2, p. 252]{QuantizationonNilpotentLieGroups}), the operators $\psi_1(-\Tilde{\Laplace}_\Hg)$ maps $\mathcal{S}(\Hg^d)$ to $\mathcal{S}(\Hg^d)$. Assume that $f\in\mathcal{S}(\Hg^d)$. In this case $\psi_1(-\Tilde{\Laplace}_\Hg)f$ belongs to $\mathcal{S}(\Hg^d)$ and we have (see Formula \eqref{eq:Formul 1.9})
\begin{align*}
\mathcal{F}_\Hg\left(\psi_1(-\Tilde{\Laplace}_\Hg)f\right)(n,m,\lambda) &=\mathbf{1}_{\{|\lambda|\leq 2\}}(n,m,\lambda)\mathcal{F}_\Hg\left(\psi_1(-\Tilde{\Laplace}_\Hg)f\right)(n,m,\lambda)\\
&=\mathcal{F}_{\Hg}\left(\mathbf{1}_{\{|D_s|\leq 2\}}\psi_1(-\Tilde{\Laplace}_\Hg)f\right)(n,m,\lambda).
\end{align*}
According to the Fourier inversion formula in Proposition \ref{inversion et Plancherel sur Heisenberg}, we thus get
$$
\psi_1(-\Tilde{\Laplace}_\Hg)f=\mathbf{1}_{\{|D_s|\leq 2\}}\psi_1(-\Tilde{\Laplace}_{\Hg})f.
$$
Since $\psi_1(-\Tilde{\Laplace}_{\Hg})$ and $\mathbf{1}_{\{|D_s|\leq 2\}}\psi_1(-\Tilde{\Laplace}_\Hg)$ are bounded operators on $\Tilde{H}^\ell(\Hg^d)$, by using the density of $\mathcal{S}(\Hg^d)$ in $\Tilde{H}^\ell(\Hg^d)$, we deduce that the above formula holds for any $f\in\Tilde{H}^\ell(\Hg^d)$. Then for any $Y\in\R^{2d}$, the Fourier transform on $\R$ of $\psi_1(-\Tilde{\Laplace}_{\Hg})f(Y,\cdot)$ is compactly supported in $\R$. Then according to the Paley-Wiener theorem, we deduce that $\psi_1(-\Tilde{\Laplace}_{\Hg})f(Y,\cdot)$ extends to a holomorphic function on $\C$. 

On the other hand, since $supp(\psi_2)\subset\R\setminus[-1,1]$ and $\|\psi_2\|_{L^\infty}\leq 1$, for any $(n,m,\lambda)\in\N^d\times\N^d\times\R^{*}$, we have
$$
|\mathcal{F}_\Hg\left(\psi_2(-\Tilde{\Laplace}_{\Hg})e^{\sigma|D_s|}f\right)(n,m,\lambda)|\leq (4|\lambda|(2|m|+d))^{\ell}|\mathcal{F}_\Hg\left(e^{\sigma|D_s|}f\right)(n,m,\lambda)|.
$$
Thus, using that $e^{\sigma|D_s|}f\in\Tilde{H}^\ell(\Hg^d)$, we deduce that
$e^{\sigma|D_s|}\psi_2(-\Tilde{\Laplace}_\Hg)f=\psi_2(-\Tilde{\Laplace}_\Hg)e^{\sigma|D_s|}f$ belongs to $L^{2}(\Hg^d)$. Then, we deduce that for any $Y\in\R^{2d}$, the function $\psi_2(-\Tilde{\Laplace}_{\Hg})f(Y,\cdot)$ can be extended to a holomorphic function with respect to the variable $s$ in $\mathsf{S}_\sigma$. 
\end{proof}
\section{Symbol of the negative powers of the sub-Laplacian}
\label{sec:Symbol of the negative powers of the sub-Laplacian}

Our first goal is to prove the following proposition, which is deduced by suitably combining several results and arguments scattered in \cite{QuantizationonNilpotentLieGroups}:
\begin{prop}
\label{prop:Remark definition des puissance fractionnaire}
For any $\lambda\in\R^{*}$, the infinitesimal representation $\mathrm{U}^{\lambda}(-\Laplace_{\Hg})$ of $-\Laplace_\Hg$ is given by
$$
\mathrm{U}^{\lambda}(-\Laplace_{\Hg})=-\Laplace_{osc}^{\lambda}
\ \text{ where }
-\Laplace_{osc}^{\lambda}:=4\left(\sum_{j=1}^{d}- \partial_{x_j}^{2} + |\lambda|^2x_{j}^{2}\right), 
$$
and the space of smooth vector fields of $(\mathrm{U}^{\lambda},L^2(\R^d))$ is $\mathcal{S}(\R^d)$. Moreover, for $f\in L^2(\Hg^d)$ and $\vphi\in L^{\infty}(\R_+; \R)$, we have
\begin{align}
\label{eq:symbol of the sublaplacian 1}
&\mathscr{F}_\Hg(\varphi(-\Laplace_{\Hg})f)(\mathrm{U}^\lambda)=\mathscr{F}_\Hg(f)(\mathrm{U}^\lambda)\circ \vphi(-\Laplace_{osc}^{\lambda})
\\
\label{eq:symbol of the sublaplacian 2}
&\mathscr{F}_\Hg(\vphi(-\Tilde{\Laplace}_\Hg)f)(\mathrm{U}^\lambda)=\vphi(-\Laplace_{osc}^{\lambda})\circ\mathscr{F}_{\Hg}(f)(\mathrm{U}^\lambda).
\end{align}
where $\vphi(-\Laplace_\Hg)$ and $\vphi(-\Tilde{\Laplace}_\Hg)$ are bounded self-adjoint operators on $L^2(\Hg^d)$ and
\begin{equation}
\label{eq:formule definition calcul fonctionnel pour l'oscillateur harmonique}
\vphi(-\Laplace_{osc}^{\lambda}):=\sum_{m\in\N^d}\vphi(4|\lambda|(2|m|+d))\mathcal{P}_{m,\lambda},
\end{equation}
where $\mathcal{P}_{m,\lambda}:=\langle\cdot, h_{m,\lambda}\rangle_{L^2(\R^d)}h_{m,\lambda}$ is a bounded self-adjoint operator on $L^2(\R^d)$.

If $\vphi$ is a measurable function on $\R_+$ such that there exist positive constants $C$ and $\ell\in[0,+\infty)$ such that
\begin{equation}
	\label{Cond-vphi-l}
	\forall \mu \in (0,\infty), \quad |\vphi(\mu)|\leq C(1+|\mu|)^{\ell},
\end{equation}
then formula \eqref{eq:symbol of the sublaplacian 1} and \eqref{eq:symbol of the sublaplacian 2} hold for $f\in \mathcal{S}(\Hg^d)$.
\end{prop}

\begin{rem}
Let us point out that in \cite{QuantizationonNilpotentLieGroups} (see \cite[Proposition 1.7.6]{QuantizationonNilpotentLieGroups} where it is given for $\vphi (x)= x$), the formula \eqref{eq:symbol of the sublaplacian 1} and \eqref{eq:symbol of the sublaplacian 2} are inverted due to the difference of convention: in \cite{QuantizationonNilpotentLieGroups} the authors define the Fourier transform as $\int_{\Hg^d}f(w)\mathrm{U}_{-w}^{\lambda}dw$ while, we use $\int_{\Hg^d}f(w)\mathrm{U}_{w}^{\lambda}dw$.
\end{rem}
\begin{proof}
	Let $\lambda\in\R^{*}$. For any $\vphi\in\mathcal{S}(\Hg^d)$ and $j\in[\![1,d]\!]$, we have $\mathrm{U}^{\lambda}_{\exp_{\Hg^d}(tX_j)}\vphi(x)=\vphi(x-2te_j)$ and $\mathrm{U}^{\lambda}_{\exp_{\Hg^d}(t\Xi_j)}\vphi(x)=e^{-2i\lambda tx_j}\vphi(x)$, we deduce that the values of the infinitesimal
representation (see \cite[Proposition 1.7.3, p. 38]{QuantizationonNilpotentLieGroups}) associated to $\mathrm{U}^\lambda$ in $X_j$ and $\Xi_j$ are respectively $\mathrm{U}^{\lambda}(X_j)=-2\partial_{x_j}$ and $\mathrm{U}^{\lambda}(\Xi_j)=-2i\lambda x_j$. We thus deduce that, for any $\lambda\in\R^{*}$, we have
\begin{equation}
	\label{Identity-U-laplace-Laplace-Lambda}
\mathrm{U}^{\lambda}(-\Laplace_{\Hg})=-\sum_{j=1}^{d}\left(\mathrm{U}^\lambda(X_j)^2+\mathrm{U}^\lambda(\Xi_j)^2\right)=-\Laplace_{osc}^{\lambda}.
\end{equation}
If $\vphi$ belongs to $L^{\infty}(\R_+)$, then $\vphi$ is bounded on the spectrum $\mathrm{Sp}(-\Laplace_{osc}^{\lambda})=\{4|\lambda|(2|m|+d)\}_{m\in\N^d}$ of $-\Laplace_{osc}^{\lambda}$, and consequently Formula \eqref{eq:formule definition calcul fonctionnel pour l'oscillateur harmonique} defines a bounded self-adjoint operator on $L^2(\R^d)$. Let $\vphi$ be a real-valued function on $L^{\infty}(\R_+)$. We now prove \eqref{eq:symbol of the sublaplacian 1} and \eqref{eq:symbol of the sublaplacian 2}. Thanks to the spectral theorem for unbounded self-adjoint operators on $L^2(\Hg^d)$, the operators $\vphi(-\Laplace_\Hg)$ and $\vphi(-\Tilde{\Laplace}_\Hg)$ are bounded self-adjoint operators since $\vphi$ is real valued and bounded on the spectrum of $-\Laplace_\Hg$ and $-\Tilde{\Laplace}_\Hg$. Moreover, it is easy to check that $\vphi(-\Laplace_\Hg)$ and $\vphi(-\Tilde{\Laplace}_\Hg)$ are respectively left-invariant and right-invariant operators on $L^2(\Hg^d)$. Thus according to the Schwartz kernel theorem on Lie group (see for instance \cite[Corollary 3.2.1, p. 133]{QuantizationonNilpotentLieGroups}), we deduce that there exist $\kappa_{\vphi}^{r}$ and $\kappa_{\vphi}^{\ell}$ in $\mathcal{S}'(\Hg^d)$ such that
\begin{equation}
\label{eq: spectral multiplier to convolution}
\vphi(-\Laplace_\Hg)f=f\star\kappa_{\vphi}^{\ell}\ \text{and}\ \vphi(-\Tilde{\Laplace}_\Hg)f=\kappa_{\vphi}^{r}\star f,
\end{equation}
where $\star$ denotes the (non-commutative) convolution product on $\Hg^d$ extended to $\mathcal{S}'(\Hg^d)\times\mathcal{S}(\Hg^d)$ and define for any functions $a$ and $b$ in $L^1(\Hg^d)$ by 
$$
	(a\star b)(Y,s):=\int_{\Hg^d}a(Y-Y',s-s'-\langle\mathfrak{S}Y,Y'\rangle_{\R^{2d}})b(Y',s')dY'ds', \quad \text{ for }(Y,s)\in\Hg^d.
$$
 For any $a\in\mathcal{S}'(\Hg^d)$, we denote $\check{a}:=a(-\cdot)$ (in the sense of distributions). Since $\Tilde{\Laplace}_\Hg f=(\Laplace_\Hg \check{f})(-\cdot)$ for any $f\in \mathcal{S}(\Hg^d)$, we deduce that the spectral measures $E$ and $\Tilde{E}$ corresponding to $\Laplace_\Hg$ and $\Tilde{\Laplace}_\Hg$ are linked by the relation $\Tilde{E}f=(E\check{f})(-\cdot)$, by the uniqueness of the spectral measure of $-\Tilde{\Laplace}_\Hg$. 
 Then, up to a suitable approximation, we deduce that for any $\vphi\in L^{\infty}(\R_+)$ ($\vphi$ is real valued) and $f\in L^2(\Hg^d)$, that
$$
\vphi(-\Tilde{\Laplace}_\Hg)f=(\vphi(-\Laplace_\Hg)\check{f})(-\cdot).
$$
Thus, thanks to \eqref{eq: spectral multiplier to convolution}, we deduce from the above identity that $\kappa_{\vphi}^{r}\star f=(\check{f}\star\kappa_{\vphi}^{\ell})(-\cdot) =\check{\kappa}_{\vphi}^{\ell}\star f$, and by the uniqueness of the Schwartz kernel theorem, we deduce that
\begin{equation}
\label{eq:}
\kappa_{\vphi}^{r}=\check{\kappa}_{\vphi}^{\ell}.
\end{equation}
By applying the Fourier transform on the two identities in \eqref{eq: spectral multiplier to convolution}, thanks to the abstract Plancherel theorem on $\Hg^d$ (see for instance \cite[Theorem 1.8.11, p. 52]{QuantizationonNilpotentLieGroups}) (keeping in mind that $f\mapsto f\star\kappa_{\vphi}^{\ell}$ and $f\mapsto\kappa_{\vphi}^{r}\star f$ are bounded operators on $L^2(\Hg^d)$), we deduce that for any $\lambda\in\R^{*}$, we have
\begin{align}
\label{eq:Fourier multiplicateur spectral gauche}
&\mathscr{F}_\Hg(\vphi(-{\Laplace}_\Hg)f)(\mathrm{U}^\lambda)=\mathscr{F}_\Hg(f)(\mathrm{U}^\lambda)\circ\mathcal{F}_{\Hg}(\kappa_{\vphi}^{\ell})(\mathrm{U}^\lambda),
\\
\label{eq:Fourier multiplicateur spectral droite}
&\mathscr{F}_\Hg(\vphi(-\Tilde{\Laplace}_\Hg)f)(\mathrm{U}^\lambda)=\mathscr{F}_{\Hg}(\check{\kappa}_{\vphi}^{\ell})(\mathrm{U}^\lambda)\circ\mathscr{F}_\Hg(f)(\mathrm{U}^\lambda).
\end{align}
Following the same lines as in the proof of \cite[Corollary 4.1.16, Identity (4.5), p. 179]{QuantizationonNilpotentLieGroups}, we deduce the formula 
$$
\mathscr{F}_\Hg(\check{\kappa}_{\vphi}^{\ell})(\mathrm{U}^\lambda)=\vphi(-\Laplace_{osc}^{\lambda}).
$$
The first step of this proof consists of applying formula \eqref{eq:Fourier multiplicateur spectral droite} up to a suitable approximation $\vphi=\Id_{\R_+}$ and $f\in\mathcal{S}(\Hg^d)$. Then, according to the Dixmier-Malliavin theorem (see \cite[Theorem 1.7.8, p. 42]{QuantizationonNilpotentLieGroups}), the uniqueness of the spectral measure of $\mathrm{U}^{\lambda}(-\Laplace_\Hg)$ and the identity $\mathrm{U}^{\lambda}(-\Laplace_\Hg)=-\Laplace_{osc}^{\lambda}$ in \eqref{Identity-U-laplace-Laplace-Lambda}, we deduce the above identity.

Then, since $\mathscr{F}_{\Hg}(\kappa_{\vphi}^{\ell})(\mathrm{U}^\lambda)^{*}=\mathscr{F}_{\Hg}(\check{\kappa}_{\vphi}^{\ell})(\mathrm{U}^\lambda)$ and $\vphi(-\Laplace_{osc}^{\lambda})$ is self-adjoint, we deduce from the above identity that 
\begin{equation}
\label{eq: egualite noyaux droit guauche}
\mathscr{F}_{\Hg}(\kappa_{\vphi}^{\ell})(\mathrm{U}^\lambda)=\vphi(-\Laplace_{osc}^{\lambda})=\mathscr{F}_{\Hg}(\check{\kappa}_{\vphi}^{\ell})(\mathrm{U}^\lambda)=\mathscr{F}_{\Hg}(\kappa_{\vphi}^{r})(\mathrm{U}^\lambda).
\end{equation}
The identities \eqref{eq:symbol of the sublaplacian 1} and \eqref{eq:symbol of the sublaplacian 2} then follow from \eqref{eq:Fourier multiplicateur spectral gauche}-\eqref{eq:Fourier multiplicateur spectral droite} and \eqref{eq: egualite noyaux droit guauche}. The fact that these identities can be extended to any $\vphi$ satisfying \eqref{Cond-vphi-l} can be done as in \cite[Example 5.1.27, p. 291]{QuantizationonNilpotentLieGroups}.
\end{proof}

According to Proposition \ref{prop:Remark definition des puissance fractionnaire}, if $\vphi(\mu)=\mu^\ell$ for $\ell \geq 0$, then \eqref{eq:symbol of the sublaplacian 1} and \eqref{eq:symbol of the sublaplacian 2} hold for every $f\in\mathcal{S}(\Hg^d)$. This is no longer the case when $\ell<0$, which deserves some special attention:
\begin{lem}
\label{theo:symbole homogene}
Let $\ell>0$. If $f$ belongs to $\mathcal{S}(\Hg^d)\cap \mathrm{Dom}((-\Laplace_{\Hg})^{-\ell})$, respectively $f\in\mathcal{S}(\Hg^d)\cap \mathrm{Dom}((-\Tilde{\Laplace}_{\Hg})^{-\ell})$, then for almost every $\lambda\in\R^{*}$ we have
$$
\mathscr{F}_{\Hg}((-\Laplace_{\Hg})^{-\ell}f)(\mathrm{U}^{\lambda})=\mathscr{F}_\Hg(f)(\mathrm{U}^{\lambda})\circ (-\Laplace_{osc}^{\lambda})^{-\ell},
$$
and if $f$ belongs to $\mathcal{S}(\Hg^d)\cap \mathrm{Dom}((-\Tilde{\Laplace}_{\Hg})^{-\ell})$, then for almost every $\lambda\in\R^{*}$ we have
$$
\mathscr{F}_{\Hg}((-\Tilde{\Laplace}_{\Hg})^{-\ell}f)(\mathrm{U}^{\lambda})= (-\Laplace_{osc}^{\lambda})^{-\ell}\circ\mathscr{F}_\Hg(f)(\mathrm{U}^{\lambda}).
$$
\end{lem}
\begin{proof}
Let $\ell>0$ and $f\in\mathcal{S}(\Hg^d)\cap \mathrm{Dom}((-\Laplace_{\Hg})^{-\ell})$. For any $j\in\N$, we set $\chi_j:=\mathbf{1}_{\{\mu^{-\ell}\leq 2^j\}}$ and we define the function $\vphi_j$ by setting, for any $\mu>0$,
$$
\vphi_j(\mu):=\chi_j(\mu)\mu^{-\ell}.
$$
Since $\vphi_j$ is a bounded measurable function, according to Proposition \ref{prop:Remark definition des puissance fractionnaire}, we have
\begin{equation}
\label{eq:Identification symbols 1}
\mathscr{F}_{\Hg}(\vphi_j(-\Laplace_\Hg)f)(\mathrm{U}^{\lambda})=\mathscr{F}_\Hg(f)(\mathrm{U}^\lambda)\circ \vphi_j(-\Laplace_{osc}^{\lambda}).
\end{equation}
Moreover, 
\begin{align*}
\|\mathscr{F}_{\Hg}(f)(\mathrm{U}^{\lambda})\circ&\vphi_j (-\Laplace_{osc}^{\lambda})-\mathscr{F}_{\Hg}(f)(\mathrm{U}^{\lambda})\circ (-\Laplace_{osc}^{\lambda})^{-\ell}\|_{\mathrm{HS}(L^2(\R^d))}^{2}\\
&=\sum_{(m,n)\in\N^{2d}}|\langle\mathscr{F}_{\Hg}(f)(\mathrm{U}^{\lambda})\circ\left(\vphi_j (-\Laplace_{osc}^{\lambda})-(-\Laplace_{osc}^{\lambda})^{-\ell}\right)h_{m,\lambda},h_{n,\lambda}\rangle_{L^2}|^2.
\end{align*}
Besides, thanks to the functional calculus of $-\Laplace_{osc}^{\lambda}$ (see Proposition \Ref{prop:Remark definition des puissance fractionnaire}), for any $m\in \N^d$,  we have
$$
\vphi_j(-\Laplace_{osc}^{\lambda})h_{m,\lambda}=\varphi_j(4|\lambda|(2|m|+d))h_{m,\lambda}\text{ and }(-\Laplace_{osc}^{\lambda})^{-\ell}h_{m,\lambda}=(4|\lambda|(2|m|+d))^{-\ell}h_{m,\lambda}.
$$
Thus
\begin{align*}
\|\mathscr{F}_{\Hg}(f)(\mathrm{U}^{\lambda})\circ&\vphi_j (-\Laplace_{osc}^{\lambda})-\mathscr{F}_{\Hg}(f)(\mathrm{U}^{\lambda})\circ (-\Laplace_{osc}^{\lambda})^{-\ell}\|_{\mathrm{HS}(L^2(\R^d))}^{2}\\
&\leq\sum_{(m,n)\in\N^{2d}}\left(\vphi_j(4|\lambda|(2|m|+d))-(4|\lambda|(2|m|+d))^{-\ell}\right)^{2}|\langle\mathscr{F}_\Hg(f)(\mathrm{U}^{\lambda})h_{m,\lambda},h_{n,\lambda}\rangle_{L^2}|^2.
\end{align*}
Furthermore, for any $j\in\N^{*}$ and $(m,\lambda)\in\N^d\times\R^{*}$, we have
\begin{align*}
\left(\vphi_j(4|\lambda|(2|m|+d))-(4|\lambda|(2|m|+d))^{-\ell}\right)^{2}&=\mathbf{1}_{(0,2^{-j/\ell})}\left(4|\lambda|(2|m|+d)\right)(4|\lambda|(2|m|+d))^{-2\ell}\\
&\leq \frac{2^{-\frac{j}{\ell}}}{(4|\lambda|(2|m|+d))^{2\ell+1}}
\leq \frac{2^{-\frac{j}{\ell}}}{|\lambda|^{2\ell+1}}.
\end{align*}
Thus, we obtain
\begin{align}
\|\mathscr{F}_{\Hg}(f)(\mathrm{U}^{\lambda})\circ\vphi_j (-\Laplace_{osc}^{\lambda})-\mathscr{F}_{\Hg}(f)(\mathrm{U}^{\lambda})\circ (-&\Laplace_{osc}^{\lambda})^{-\ell}\|_{\mathrm{HS}(L^2(\R^d))}^{2}\nonumber\\
&\leq \frac{2^{-j/\ell}}{|\lambda|^{2\ell+1}}\|\mathscr{F}_{\Hg}(f)(\mathrm{U}^{\lambda})\|_{\mathrm{HS}(L^2(\R^d))}^{2}.\label{eq:Identification symbols 2}
\end{align}
The quantity in the right-hand side of the above inequality is finite because $f\in\mathcal{S}(\Hg^d)$ implies that $\mathscr{F}_{\Hg}(f)(\mathrm{U}^{\lambda})\in \mathrm{HS}(L^{2}(\R^d))$ (see Proposition \ref{inversion et Plancherel sur Heisenberg}). Combining \eqref{eq:symbol of the sublaplacian 1} and \eqref{eq:symbol of the sublaplacian 2}, we deduce that
\begin{equation}
\label{eq:Identification symbols 3}
\forall\lambda\in\R^{*},\ \ \lim_{j\rightarrow+\infty}\mathscr{F}_\Hg(\vphi_j(-\Laplace_{\Hg})f)(\mathrm{U}^\lambda)=\mathscr{F}_\Hg(f)(\mathrm{U}^\lambda)\circ(-\Laplace_{osc}^{\lambda})^{-\ell}.
\end{equation}
Moreover, on one hand, since $f\in \mathrm{Dom}((-\Laplace_{\Hg})^{-\ell})$, we have $(-\Laplace_{\Hg})^{-\ell}f\in L^2(\Hg^d)$. On the other hand, using that $f \in L^2(\Hg^d)$ and $\vphi_j\in L^{\infty}(\R_+)$, we get $\varphi_j(-\Laplace_{\Hg})f\in L^2(\Hg^d)$. Then, we deduce from the Plancherel theorem on $\Hg^d$ (see Proposition \ref{inversion et Plancherel sur Heisenberg}) and the functional calculus for $-\Laplace_{\Hg}$ that
\begin{align*}
\int_{\R^{*}}\|\mathscr{F}_\Hg(\vphi_j(-\Laplace_{\Hg})f)(\mathrm{U}^\lambda)-&\mathscr{F}_\Hg((-\Laplace_{\Hg})^{-\ell}f)(\mathrm{U}^\lambda)\|_{\mathrm{HS}(L^2(\R^d))}^{2}|\lambda|^d d\lambda\\ &=\frac{\pi^{d+1}}{2^{d-1}}\|\mathbf{1}_{(0,2^{-j/\ell})}\left(-\Laplace_{\Hg}\right) \circ(-\Laplace_{\Hg})^{-\ell}f\|_{L^2(\Hg^d)}^{2}\\
&=\frac{\pi^{d+1}}{2^{d-1}}\int_{0}^{+\infty}\mathsf{1}_{(0,2^{-j/\ell})}(\mu)|\mu|^{-2\ell}d\langle E(\mu)f,f\rangle_{L^2(\Hg^d)}.
\end{align*}
Since $f\in \mathrm{Dom}((-\Laplace_{\Hg})^{-\ell})$, we have $|\cdot |^{-2\ell}\in L^1\left(\R_+, d\langle E f,f\rangle_{L^2(\Hg^d)}\right)$. Then, according to the dominated convergence theorem, the right-hand side of the above identity converges to $0$ when $j$ goes to $+\infty$. Hence, we deduce that
\begin{equation*}
a.\ e.\ \lambda\in\R^{*},\ \ \ \mathscr{F}_{\Hg}((-\Laplace_{\Hg})^{-\ell}f)(\mathrm{U}^{\lambda})=\mathscr{F}_{\Hg}(f)(\mathrm{U}^\lambda)\circ (-\Laplace_{osc}^{\lambda})^{-\ell},
\end{equation*}
in $\mathrm{HS}(L^2(\R^d))$ (the class of Hilbert-Schmidt operators in $L^2(\R^d)$). The case of $(-\Tilde{\Laplace}_{\Hg})^{-\ell}$ follows similarly (see Proposition \ref{prop:Remark definition des puissance fractionnaire}).
\end{proof}
\section{Compactness results for sub-elliptic Sobolev spaces}
\label{sec:Compactness results}

We have the following Rellich-type theorem.
\begin{prop}\label{Sobolev compacity on homogeneous Lie groups}
Let $\ell$ and $\ell'$ be two real numbers such that $\ell<\ell'$ and $p\in(1,+\infty)$. The multiplication by any element of $\mathcal{D}(\Hg^d)$ is a compact operator from $W^{\ell', p}_{\Hg}(\Hg^d)$ to $W^{\ell/2,p}_{\Hg}(\Hg^d)$.
\end{prop}
\begin{proof}
Let $\vphi\in\mathcal{D}(\Hg^d)$. The multiplication operator by $\vphi$ denoted by $M_{\varphi}$ is a bounded linear operator from $W_{\Hg}^{\ell',p}(\Hg^d)$ to $W^{\ell'/2,p}(\R^{2d+1})$ (see \cite[Theorem 4.4.24, p. 240]{QuantizationonNilpotentLieGroups}). Besides, the multiplication by a smooth compactly supported function is a compact operator from $W^{\ell'/2,p}(\R^{2d+1})$ to $W^{\ell/2,p}(\R^{2d+1})$ and a bounded operator from $W^{\ell/2,p}(\R^{2d+1})$ to $W^{\ell/2,p}_{\Hg}(\Hg^{d})$ (see \cite[Theorem 4.4.24, p. 240]{QuantizationonNilpotentLieGroups}). Consequently, $M_{\varphi}$ is a compact operator from $W_{\Hg}^{\ell',p}(\Hg^d)$ to $W_{\Hg}^{\ell/2,p}(\Hg^d)$.
\end{proof}
{\small
\bibliographystyle{abbrv}
\bibliography{Bibliographie}
}
\end{document}